\documentclass[onefignum,onetabnum]{siamonline250211}

\usepackage{bm}
\usepackage{booktabs}
\usepackage{bbold}
\usepackage{mathtools}
\usepackage{amssymb}
\usepackage{stmaryrd}
\usepackage{graphicx}
\usepackage{relsize}
\usepackage{algorithm}
\usepackage{algpseudocode}
\usepackage{nicefrac}
\usepackage{multirow}
\usepackage{enumitem}

\crefformat{equation}{(#2#1#3)}
\Crefformat{equation}{#2Equation#3~(#2#1#3)}
\crefrangeformat{equation}{(#3#1#4)--(#5#2#6)}
\Crefrangeformat{equation}{#3Equations#4~(#3#1#4)--(#5#2#6)}
\crefmultiformat{equation}{(#2#1#3)}{ and~(#2#1#3)}{, (#2#1#3)}{, and~(#2#1#3)}
\Crefmultiformat{equation}{#2Equations#3~(#2#1#3)}{ and~(#2#1#3)}{, (#2#1#3)}{, and~(#2#1#3)}
\crefrangemultiformat{equation}{(#3#1#4)--(#5#2#6)}{ and~(#3#1#4)--(#5#2#6)}{, (#3#1#4)--(#5#2#6)}{, and~(#3#1#4)--(#5#2#6)}
\Crefrangemultiformat{equation}{#3Equations#4~(#3#1#4)--(#5#2#6)}{ and~(#3#1#4)--(#5#2#6)}{, (#3#1#4)--(#5#2#6)}{, and~(#3#1#4)--(#5#2#6)}

\newsiamremark{remark}{Remark}
\AddToHook{env/remark/begin}{\crefalias{theorem}{remark}}

\crefname{section}{Section}{Sections}
\crefname{appendix}{Appendix}{Appendices}
\crefname{theorem}{Theorem}{Theorems}
\crefname{proposition}{Proposition}{Propositions}
\crefname{corollary}{Corollary}{Corollaries}
\crefname{lemma}{Lemma}{Lemmas}
\crefname{remark}{Remark}{Remarks}
\crefname{example}{Example}{Examples}
\crefname{algorithm}{Algorithm}{Algorithms}
\crefname{figure}{Figure}{Figures}
\crefname{table}{Table}{Tables}

\allowdisplaybreaks

\DeclareMathOperator*{\argmax}{argmax}

\DeclareMathOperator*{\tr}{tr}
\DeclareMathOperator{\diam}{diam}
\DeclareMathOperator*{\spann}{span}
\DeclareMathOperator*{\supp}{supp}

\DeclareMathOperator{\id}{Id}

\DeclareMathOperator*{\argmin}{argmin}

\DeclareMathOperator{\W}{W}
\DeclareMathOperator{\GW}{GW}

\newcommand{\Pio}{\Pi_{\mathrm{o}}}

\newcommand{\XX}{\mathbb{X}}
\newcommand{\YY}{\mathbb{Y}}

\newcommand{\G}{\mathcal{G}}

\newcommand{\1}{\mathbb{1}}

\newcommand{\weakly}{\rightharpoonup}
\newcommand{\N}{\mathbb{N}}
\newcommand{\M}{\mathcal{M}}

\newcommand{\R}{\ensuremath{\mathbb{R}}}

\newcommand{\dx}{\,\mathrm{d}}

\newcommand{\p}{\mathcal{P}}

\newcommand{\sym}{\mathrm{sym}}

\newcommand{\GM}{\mathfrak{G\!M}}

\newcommand{\qW}{\mathrm{q}_{n}^{\W}}
\newcommand{\qGW}{\mathrm{q}_{n}^{\GW}}

\newcommand{\m}{\mathrm{m}}

\numberwithin{equation}{section}

\headers{Gromov--Wasserstein Quantization and Clustering}
{F. Beier, S. Eckstein}

\title{Gromov--Wasserstein Quantization and Clustering: Structure, Rates, and Algorithms}

\author{Florian Beier%
\thanks{Department of Mathematics, University of Tübingen, Germany,
(\email{florian.beier@uni-tuebingen.de}, \email{stephan.eckstein@uni-tuebingen.de}).%
\funding{FB is grateful for funding by the German Research Foundation
-- Project number 564105502. SE is grateful for support by the German Research Foundation through Project 553088969 as well as the Cluster of Excellence “Machine Learning --- New Perspectives for Science” (EXC 2064/1 number 390727645).}}
\and
Stephan Eckstein\footnotemark[1]%
}

\begin{document}

\maketitle

\begin{abstract}
Clustering is a fundamental class of data analysis techniques with the most important representatives being centroid-based methods like $k$-means. Such methods are strongly connected to quantization problems, which aim to approximate general probability measures with discrete ones. For example, $k$-means corresponds to quantization with respect to the Wasserstein distance. While Wasserstein quantization clusters points within a fixed space, this paper studies Gromov--Wasserstein (GW) quantization, which additionally aims at clustering the ambient geometry of the space. We show existence of solutions to the GW quantization problem and give a characterization that justifies an analogue to the $k$-means algorithm (Lloyd's algorithm) to approximate them numerically. We further calculate the quantization rate for usual Euclidean geometries that are used in the GW context, and relate it to standard Wasserstein quantization rates. Finally, numerical experiments show that GW quantization opens up many modeling possibilities beyond normal clustering methods (e.g., for geodesic distances of 3D shapes or structured pruning of neural networks) and that the introduced algorithm leads to useful numerical solutions with approximation quality often in line with theoretically optimal rates.
\end{abstract}

\section{Introduction}
Objects in data science are increasingly large, from data sets to neural networks and high-resolution measurements.
It is thus often desirable to suitably approximate 
these large objects with smaller ones, 
which are, 
for instance, 
more tractable computationally and easier to interpret. 
As many objects can be modeled as weighted sets of points, 
a common way to
achieve such a reduction in size is through clustering, 
which replaces the points
in each cluster by a single representative. 
This is precisely Wasserstein
quantization of a finitely supported measure, 
for which Lloyd's algorithm
is the standard method.
However, 
such a clustering approach, 
by definition, 
primarily cares about distances between points in a cluster and their representatives,
and thus only indirectly about downstream geometric properties 
of the resulting object as a whole. 
To capture such global structure directly,
we model objects as gauged measure spaces,
consisting of points, weights, 
\emph{and} a relational structure on the set of points (the gauge).
The Gromov--Wasserstein (GW) distance
is a type of optimal transport distance
that measures the overall minimal gauge-distortion
between two gauged measure spaces.
This makes the GW distance 
well suited for comparing objects 
by their intrinsic relational structure.

In this paper, 
we study the GW quantization problem 
for arbitrary gauged measure spaces.
In contrast to the Wasserstein quantization problem,
its objective
is the direct and global retention 
of the specified gauge
(i.e., the pairwise relationship between points) 
while reducing the object's size.
Our interest is
twofold. 
On the practical side, 
moving from the Wasserstein to the GW setting
affords considerable modeling freedom: an object may be encoded directly through
its pairwise relations, 
which need not form a metric 
(see \Cref{subsec:pruningnn} for one such example).
This extends quantization to settings where ordinary
clustering does not apply. 
For this flexibility to be practically useful, 
a natural goal is to lift Lloyd's algorithm,
the standard tool in the Wasserstein case, 
to the GW setting.
On the theoretical side, 
it raises structural
questions of independent interest, 
concerning the form of optimal GW quantizers
and whether the optimal quantization rate
matches the one in the Wasserstein setting.

\paragraph{Main contributions} 
Answering the previous motivation, 
the main contributions of this paper are:
\begin{itemize}
    \item We introduce the GW quantization problem for general gauged measure spaces, show that it admits a solution, and characterize it in \cref{thm:GW_quantization_existence_characterization}.

    \item In Euclidean spaces of dimension $d$, we establish two-sided bounds relating the GW
    quantization value to constant multiples of the Wasserstein one
    (\cref{prop:GW_quant_upper_bound},
    \cref{thm:lb_euclidean_gauges,thm:lower_and_upper_bound_scalar_product}), which
    yield the sharp $n^{-1/d}$ quantization rates
    (\cref{cor:GW_quant_rate_1,cor:GW_quant_rate_2}) for approximation with $n$ support points.
    \item Existence of Monge maps between an input endowed with the Euclidean scalar product and its quantization is shown in \cref{thm:scalar_product_partition_attainment}.
    
    \item We lift Lloyd's algorithm to the GW setting, casting the general problem as
    a conditional gradient method (\cref{alg:1}) 
    that, for a suitable specification, monotonically decreases the objective 
    and has stationary cluster points (\cref{thm:convergence}).
    
    \item We validate the approach on proof-of-concept experiments: quantizing 3D
    shapes (\Cref{subsec:exp-mesh}) and pruning neural networks
    (\Cref{subsec:pruningnn}), accelerating pairwise GW computations across many objects
    (\Cref{subsec:accelerated_pairwise}), and empirically confirming the quantization
    rates (\Cref{subsec:numerics_rates}).
\end{itemize}

\paragraph{Outline of the paper} The paper is structured as follows: In the remainder of this section, we discuss important related literature. In \cref{sec:OT}, we briefly discuss the classical Wasserstein quantization problem and its relation to clustering. In \cref{sec:gw}, we introduce the GW quantization problem and discuss basic properties related to existence and characterization of optimizers. \Cref{sec:euclidean} studies properties of GW quantization for Euclidean spaces and particularly the quantization rates are derived in this section. \Cref{sec:alg} studies the conditional gradient algorithm to (locally) solve the GW quantization problem, which is loosely inspired by Lloyd's algorithm. Finally, \cref{sec:numerics} contains the numerical experiments. Supplementary statements and proofs are outsourced to the appendix.

\subsection{Related Literature}\label{subsec:literature}

\paragraph{Quantization of probability measures}
Quantization of probability measures studies how to approximate arbitrary probability measures in Wasserstein distance with ones that have a predefined finite support size, see \cite{GL2000} for a monograph covering the topic. More broadly, this relates to signal processing, rate--distortion theory and vector quantization, see for instance \cite{gray1998quantization} and \cite{pages2015introduction}. For absolutely continuous measures on $\mathbb{R}^d$ with light enough tails, Zador's theorem \cite{zador1982asymptotic} shows that the quantization error decays at the rate $n^{-1/d}$. 
Naturally, optimal quantization has strong relations to statistical estimation of probability measures, see for instance \cite{canas2012learning}. 
Closely related, 
the approximation of measures
by point masses via discrepancies
(dithering)
and extensions thereof
have been studied both algorithmically
and through their connection to optimal transport
\cite{EGNS2021,graef2012quadrature,neumayer2021optimal,teuber2011dithering}.
Finally, quantization can be restricted to manifolds \cite{kloeckner2012approximation} or applied in general metric measure spaces \cite{aydin2026asymptotics}. We note that even in these cases, the ambient metric space is fixed and only the measure is approximated, whereas in GW quantization the ambient metric (or more generally the gauge) is approximated simultaneously.

\paragraph{Clustering and Lloyd's algorithm}
When the measure to be approximated is already finitely supported, Wasserstein quantization coincides with $k$-means clustering (cf.~\cref{sec:OT}), which is an NP-hard problem (see, e.g., \cite{aloise2009np,mahajan2012planar}). The most frequently employed algorithm to find a local solution is Lloyd's algorithm \cite{lloyd1982}, whose convergence is for instance studied in \cite{portales2025sequential, selim1984kmeans}. As convergence is only local, initialization of Lloyd's algorithm is often important \cite{arthur2007kmeanspp}. Variants of the algorithm exist, like ones where cluster sizes are prescribed (corresponding to fixed weights of the approximating measure) \cite{ng2000note}, and many of the core algorithmic steps reappear in algorithms for Wasserstein barycenters \cite{cuturi2014fast,lindheim2023simple}. 

\paragraph{GW distances}

Originally, 
GW distances were introduced by
Sturm \cite{sturm2006geometry} and M\'emoli \cite{memoli2011gromov}
in order to analyze
the space of metric measure spaces.
They are distinct natural extensions
of two equivalent Gromov--Hausdorff distance
formulations.
This paper focuses on M\'emoli's proposed GW distance
which measures the overall minimal distortion
required to go from one metric measure space to another.
Although, 
in general,
it amounts to numerically solving a 
non-convex and quadratic problem,
it has been used for a wide variety of practical tasks,
such as 
shape matching and classification \cite{BBS2022linear,memoli2011gromov,solomon2016entropic}, 
unsupervised translation \cite{alvarez2018gromov}, 
graph matching and node embedding \cite{XLZD2019}, 
multi-omics \cite{demetci2022scot},
and
particle dynamics \cite{B2023ssvm}.
Moreover, 
many theoretical insights into GW problems
and variations thereof 
have been obtained recently \cite{BHS21,delon2022gromov,DLV2024,memoli2024comparison,SejViaPey21,vayer2020fused,ZGMS2024}.
In particular,
when relaxing metrics to so-called gauges 
(symmetric square-integrable functions),
which we also rely on in this paper,
the space of gauged measure spaces endowed with 
the GW distance forms a complete metric space 
that offers a rich geometry 
\cite{sturm2023space}.
Furthermore, 
as it is related to the current work,
we particularly want to draw attention to
the recent work on semi-discrete GW \cite{RGK2023}, 
where one marginal is finitely supported.

Motivated by the poor tractability of solving the GW problem directly,
many relaxations with faster solvers have been proposed.
Examples of this are 
entropic regularization 
\cite{flamary2021pot,houry2026gromov,PCS2016,solomon2016entropic}, 
slicing approaches \cite{piening2025slicing,piening2026novel,sliced_gw}, 
low-rank and sample-based variants \cite{kerdoncuff2021sampled,scetbon2022linear} 
or iterative partitioning \cite{xu2019scalable}. 
In the context of low-dimensional Euclidean space 
equipped with the squared Euclidean norm,
the GW problem becomes more manageable 
and global solutions can be obtained in reasonable time \cite{ryner2023_globally}. 

\paragraph{GW quantization and related methods}
GW quantization is introduced in \cite{memoli2018sketching} for metric measure spaces. A focus is on a type of equivalence (called duality therein) between GW quantization and clustering. Notably, both the standard GW distance from \cite{memoli2011gromov} and Sturm's version \cite{sturm2006geometry} are treated, which lead to different results. In contrast to \cite{memoli2018sketching}, the current work aims to establish results closer to classical Wasserstein quantization theory for general GW objectives. In particular, existence and characterization of optimal solutions, Lloyd-type iterations and general gauges are not treated in \cite{memoli2018sketching}. The strongest relation between the duality established in \cite{memoli2018sketching} and our work is given by the bounds in \cref{sec:euclidean}, where particularly some of the upper bounds could be derived from \cite{memoli2018sketching}. Nevertheless, while \cite{memoli2018sketching} focuses on bounds involving universal constants, \cref{sec:euclidean} focuses on asymptotic rates, and importantly the main results in this section involve constants which depend on the measure which is quantized. 
In \cite{vanassel2025distributional}, a particular empirical GW quantization problem is studied in which all cluster centers have the same mass. The goal of \cite{vanassel2025distributional} is to unify both clustering and dimensionality reduction via GW quantization, which is achieved by considering GW quantization on $\mathbb{R}^p$ with quantized gauges being restricted to arise from points in $\mathbb{R}^d$ with $d < p$. This is in contrast to our paper, where both initial and quantized gauges are not restricted to have a euclidean structure.
The quantized GW approach of \cite{chowdhury2021quantized} treats the use of any GW quantization (e.g., also suboptimal ones obtained through normal clustering) to speed up the computation of GW distances. 
Calculating GW barycenters \cite{BB2024tangential,BBS2023multi,chowdhury2020gromov}, especially with fixed barycenter size \cite{PCS2016}, is a form of GW quantization (when calculating the barycenter of a single space), and some of the algorithmic steps therein relate to the ones in this paper (cf.~the discussion after \cref{thm:pointwise_lb_gauge}). Related methods further include coupling a graph to a small template, which yields graph partitioning methods \cite{chowdhury2021generalized}, GW-based dictionary learning \cite{vincent2021online,xu2020gromov}, and the semi-relaxed GW divergence \cite{vincent2022semirelaxed}, which includes one free marginal as in the GW quantization problem. More broadly related approximation methods include graph coarsening \cite{bravohermsdorff2019unifying,chen2023gromov,loukas2019graph,taveras2026gromov}.

\subsection{Notation}
\begin{description}[
        font=\normalfont,
        align=left,
        labelwidth=4.1em,
        labelsep=1.5em,
        leftmargin=11em,
        itemsep=0.4mm]
    \item[{$[n]$}] index set $\{1,\dots,n\}$, for $n \in \N$.
    \item[{$\langle\cdot,\cdot\rangle, \|\cdot\|$}] Euclidean scalar product and norm on $\R^d$.
    \item[{$\M_+(X)$}] finite non-negative measures on $X$.
    \item[{$\p(X)$}] probability measures on $X$.
    \item[{$\p^{(2)}(X)$}] probability measures with finite second moment (cf.\ \cref{sec:OT}).
    \item[{$\p_n(X)$}] probability measures supported on at most $n$ points (cf.\ \cref{sec:OT}).
    \item[{$\supp(\xi)$}] support of $\xi\in\M_+(X)$, the smallest closed set of full $\xi$-measure.
    \item[{$\xi\vert_A$}] restriction of $\xi$ to a measurable $A\subset X$, i.e.\ $\xi\vert_A(B)=\xi(A\cap B)$.
    \item[{$\diam(A)$}] diameter $\sup_{x,x'\in A}d_X(x,x')$ of $A\subset X$.
    \item[{$\xi_k\weakly\xi$}] weak convergence of finite measures.
    \item[{$\Phi_\#\xi$}] push-forward of $\xi$ by measurable $\Phi\colon X\to Y$, i.e.\ $(\Phi_\#\xi)(\cdot)=\xi(\Phi^{-1}(\cdot))$.
    \item[{$P_X,\ P_Y$}] canonical projections of $X\times Y$ onto $X$ and $Y$.
    \item[{$\Pi(\xi,\upsilon)$}] couplings between $\xi$ and $\upsilon$.
    \item[{$\W_2$}] Wasserstein distance of order 2 (cf.\ \cref{sec:OT}).
    \item[{$\GM$}] gm-spaces up to homomorphic equivalence (cf.\ \cref{sec:gw}).
    \item[{$\G^{(n)}_{\sym}$}] symmetric matrices $G\in\R^{n\times n}$, called gauge matrices.
    \item[{$\GM_n$}] gm-spaces on at most $n$ points, i.e.\ $([n],G,\upsilon)$ with $G\in\G^{(n)}_{\sym}$ and $\upsilon\in\p([n])$.
    \item[{$\GW_2$}] Gromov--Wasserstein distance of order 2 (cf.\ \cref{sec:gw}).
    \item[{$\Pio(\XX,\YY)$}] set of optimal plans between $\XX$ and $\YY$.
    \item[{$\Pi(\xi,*_n)$}] couplings $\pi\in\p(X\times[n])$ with $(P_X)_\#\pi=\xi$.
    \item[{$\pi_i$}] $i$-th slice $\pi(\cdot\times\{i\})\in\M_+(X)$ of a coupling $\pi\in\p(X\times[n])$.
    \item[{$\upsilon_i$}] mass $\upsilon(\{i\})$ of $\upsilon\in\p([n])$ at $i\in[n]$.
    \item[{$(V_i)_{i=1}^n$}] Voronoi partition of $X$ with respect to cost functions $(c_i)_{i=1}^n$ (cf.\ \cref{eq:voronoi_partition}).
    \item[{$\m(\pi)$}] normalized mean $\tfrac{1}{\pi(X)}\int_X x\dx\pi(x)$ of $\pi\in\M_+(X)$ with $\pi(X)>0$.
    \item[{$\m_\xi(V)$}] block mean $\m(\xi\vert_V)$ for measurable $V\subset X$ with $\xi(V)>0$.
    \item[{$M_\xi$}] second-moment matrix $\int_{\R^d}xx^\top\dx\xi(x)$ of $\xi\in\p^{(2)}(\R^d)$.
    \item[{$\qW, \qGW$}] Wasserstein and GW quantization values (\cref{eq:W_quantization,eq:GW_quant}).
    \item[{$Q_{\GW}$}] objective for the GW quantization problem (cf.\ \cref{sec:gw}).
    \item[{$R$}] reduced objective $\inf_{G}Q_{\GW}(G,\cdot)$ of the GW quantization problem (cf.\ \cref{sec:alg}).
\end{description}

\section{Wasserstein Quantization and Clustering}\label{sec:OT}
In the following, 
we briefly review Wasserstein quantization
of probability measures 
and refer to \cite{GL2000, aydin2026asymptotics, portales2025sequential} 
for more background.

Let $(X,d_X)$ be a Polish space, 
and denote by $\M_+(X)$ and $\p(X)$ 
the spaces of finite
non-negative and of probability 
measures on the Borel-$\sigma$-algebra of $X$
induced by $d_X$, respectively.
For Polish spaces $X,Y$, 
let $P_X\colon X\times Y\to X$ 
and $P_Y\colon X\times Y\to Y$
denote the canonical projections.
For
$\xi\in\p(X)$ 
and $\upsilon\in\p(Y)$, 
the set of \emph{couplings}
or \emph{transport plans} 
of $\xi$ and
$\upsilon$ is
\[
\Pi(\xi,\upsilon)\coloneqq\{\pi\in\p(X\times Y):(P_X)_\#\pi=\xi,\ (P_Y)_\#\pi=\upsilon\}.
\]
The \emph{Wasserstein space (of order 2)}
is given by the set of measures 
with finite second moment
\begin{align*}
\p^{(2)}(X) &\coloneqq 
\{\xi \in \p(X) : \int_{X} d_X^2(x_0,x) \dx \xi(x) < \infty\},
\quad x_0 \in X \text{ arbitrary},
\end{align*}
endowed with the \emph{Wasserstein distance (of order 2)}
\begin{align*}
    \W_2(\xi,\upsilon) = \inf_{\pi \in \Pi(\xi,\upsilon)} \biggl(
    \int_{X \times X} d_X^2(x,y) \dx \pi(x,y)
    \biggr)^{\frac{1}{2}},
    \qquad \xi,\upsilon \in \p^{(2)}(X).
\end{align*}
In order to define 
the Wasserstein quantization problem,
we require
the set of measures on $X$ with at most $n$ support points,
denoted by
\[
\p_n(X) \coloneqq 
\biggl\{\sum_{i=1}^n \upsilon_i \delta_{y_i}
: 
y_i \in X, 
\upsilon_i \geq 0, i=1,\dotsc,n,
\sum_{i=1}^n \upsilon_i = 1\biggr\}.
\]
Let $n \in \N$ and $\xi \in \p^{(2)}(X)$,
the \emph{Wasserstein quantization problem} reads
\begin{equation}\label{eq:W_quantization}
\qW(\xi)
\coloneqq 
\inf_{\upsilon \in \p_n(X)}
\W_2(\xi,\upsilon).
\end{equation}
While more general distances can be employed, 
in this paper we focus on just $\W_2$ 
(and later the GW distance of order 2).
For convex $X \subset \R^d$, existence of \emph{quantizations},
i.e.\ solutions $\upsilon^*$
to \cref{eq:W_quantization}, 
is guaranteed by
\cite[Thm.~4.12]{GL2000}.
If further 
$\xi = \frac{1}{N}\sum_{i=1}^N \delta_{x_i} \in \p(X)$ is an empirical measure,
the quantization problem
\cref{eq:W_quantization} is precisely 
Euclidean $k$-means (minimum-sum-of-squares
clustering) with $k=n$, see \cite{GL2000}, which is NP-hard \cite{aloise2009np}.

The interest of solving \cref{eq:W_quantization} is twofold.
Firstly,
a solution $\upsilon = \sum_{i=1}^n \upsilon_i \delta_{y_i} \in \mathcal{P}_n(X)$ 
provides a quantization (i.e.\ a smaller approximate representation) of $\xi$ 
according to the Wasserstein distance $\W_2$.
Secondly,
any $\pi \in \Pi(\xi, \upsilon)$,
that solves the inner problem
provides a clustering of $\xi$ according to $d_X$.
More precisely, 
the measures $\pi_i \coloneqq \pi(\cdot \times \{y_i\})$ 
describe the part of the measure $\xi$ which belongs to the $i$-th cluster with center $y_i$. 
With this in mind, 
it becomes clear that we can rewrite
\begin{equation}\label{eq:qW_as_min_QW}
\qW(\xi)^2
=
\inf_{\substack{
(y_i)_{i=1}^n \subset X
\\
(\pi_i)_{i=1}^n \subset \M_+(X), \sum_{i=1}^n \pi_i = \xi
}}
\sum_{i=1}^n \int_X d_X^2(x,y_i) \dx \pi_i(x).
\end{equation}
Minimizing out the $(\pi_i)_{i=1}^n$ by assigning each point to a nearest center,
\cref{eq:qW_as_min_QW} 
can be reduced 
to the $n$-point covering form
\begin{equation}\label{eq:n-center-covering-radius}
\qW(\xi)^2
= \inf_{\substack{Y\subset X\\ |Y|\le n}} \int_X d_X^2(x,Y)\,\dx\xi(x),
\qquad d_X(x,Y) \coloneqq \min_{y\in Y} d_X(x,y).
\end{equation}

Lloyd's algorithm \cite{lloyd1982} is based on an alternating optimization of
$(y_i)_{i=1}^n$ and $(\pi_i)_{i=1}^n$.
Optimizing over one parameter and keeping the other one fixed
allows for closed form solutions.
The first insight in this regard is that,
given $(y_i)_{i=1}^n$, 
$\pi_i$ is concentrated on 
$\{x \in X : d_X(x,y_i) = \min_{j=1,\dotsc,n} d_X(x,y_j)\}$,
which gives rise to the importance of 
so-called Voronoi partitions of the space $X$. 
Conversely, the second insight is that given $\pi_i$, 
the $y_i$ are simply given as suitable means of the parts of $\xi$ corresponding to $\pi_i$. 
For $X \subset \R^d$ convex 
and the Euclidean distance 
$d_X(\cdot, \cdot) = \|\cdot - \cdot\|$, 
this is just the standard arithmetic mean,
i.e.\ 
\cref{eq:qW_as_min_QW} 
simplifies further to
\begin{equation}\label{eq:W_quant_as_partitioning_problem}
\qW(\xi)^2
=
\inf_{
\substack{
V_1,\dotsc,V_n \subset X \text{ disjoint}
\\
\bigcup_{i=1}^n V_i = X
}
}
\sum_{i=1}^n \int_{V_i} 
\|x - 
\m_\xi(V_i)
\|^2 \dx \xi(x),
\qquad
\m_\xi(V_i) \coloneqq 
\frac{1}{\xi(V_i)}\int_{V_i} x \dx \xi(x)
\end{equation}
with the convention that
$\m_\xi(V_i) \in X$ is set arbitrarily for $\xi(V_i) = 0$.
We derive analogues to these properties for GW quantization in \cref{sec:gw}, which are the basis for the algorithmic ideas in \cref{sec:alg}.  

Regarding approximation quality, Zador's theorem characterizes the asymptotic rate of approximation, 
i.e.\ the behavior of the value 
$\qW(\xi)$ 
for $n\rightarrow \infty$. 
More precisely, consider $\xi \in \mathcal{P}(\mathbb{R}^d)$ 
with $\int \|x\|^{2+\delta} \dx \xi(x) < \infty$ 
for some $\delta > 0$ 
and assume that $\xi$ contains an absolutely continuous part 
with non-zero mass. 
Then, Zador's theorem \cite[Thm.~6.2]{GL2000} yields $\lim_{n\rightarrow \infty} n^{1/d} \qW(\xi) \in (0, \infty)$.
We exploit this in \cref{sec:euclidean}
to show similar rates of convergence 
for GW quantization problems 
in common Euclidean settings as well.

\section{
Gromov--Wasserstein Quantizations:
Existence and Characterization
}\label{sec:gw}

The GW distance 
\cite{memoli2011gromov,sturm2023space}
is a type of optimal transport problem
which aligns and compares based on
intrinsic (dis-)similarities
as opposed to a given ambient metric.
After giving a brief summary and fixing the notation,
we formulate and analyze the GW quantization problem
below.

A \emph{gauged measure space} (gm-space)
is a triple
$(X,g,\xi)$,
where
$X$
is a Polish space,
$\xi \in \p(X)$ is a probability measure
and $g \in L^2_\sym(\xi \otimes \xi)$
is an associated \emph{gauge},
where
\[
L^2_\sym(\xi \otimes \xi) 
\coloneqq 
\{
g \in L^2(\xi \otimes \xi)
:
g \text{ is symmetric}
\}.
\]
We denote the set of gm-spaces by $\GM$.
For two gm-spaces 
$\XX = (X,g,\xi) \in \GM$, 
$\YY = (Y,h,\upsilon) \in \GM$,
the \emph{GW distance (of order 2)} is given as
\begin{equation}
    \label{eq:GW}
    \GW_2(\XX,\YY) 
    \coloneqq
    \inf_{\pi \in \Pi(\xi,\upsilon)}
    \Bigl(
    \iint_{(X \times Y)^2} 
    \lvert
    g(x,x') - h(y,y')
    \rvert^2 
    \dx \pi(x,y) \dx \pi(x',y') 
    \Bigr)^{\frac{1}{2}}.
\end{equation}
The infimum in \cref{eq:GW} is always attained
and the set of solutions is denoted as $\Pio(\XX,\YY)$.

The GW distance is only a pseudometric on gm-spaces.
More precisely,
$\GW_2(\XX,\YY)$ may vanish for distinct
$\XX = (X,g,\xi)$ and $\YY = (Y,h,\upsilon)$.
The gm-spaces $\XX,\YY$ with $\GW_2(\XX,\YY) = 0$
are called \emph{homomorphic}
which admits an exact characterization,
for which we refer to \cite{sturm2023space}.
For our purposes only the following sufficient condition
is needed.
If there exists a Borel-measurable map
$\Phi \colon X \to Y$
with
$\upsilon = \Phi_\# \xi$
and
$h(\Phi(x),\Phi(x')) = g(x,x')$
for $\xi \otimes \xi$-almost every $(x,x')$,
then $\XX$ and $\YY$ are homomorphic.
We call such a map $\Phi$ a \emph{homomorphism} from $\XX$ to $\YY$.
By identifying homomorphic gm-spaces,
$\GW_2$ becomes a complete metric on the resulting quotient of $\GM$.
Throughout,
we write $\XX \in \GM$ for a gm-space as well as for its
(homomorphism) equivalence class.

Let $n \in \N$ and set $[n] \coloneqq \{1,\dotsc,n\}$.
Following \cite{memoli2018sketching} for metric measure spaces, we want to consider the 
natural generalization of the
Wasserstein quantization problem \cref{eq:W_quantization}
to the GW case
for which we consider
gm-spaces $\YY = (Y,h,\upsilon) \in \GM$
whose space $Y$ 
has
$n \in \N$ elements.
As renaming the elements constitutes a homomorphism,
any such gm-space $\YY$ always admits a representative
whose space is $[n]$.
Moreover,
we identify the gauges $h$ for such spaces
as \emph{gauge matrices}
\begin{align*}
G \in \G^{(n)}_{\sym} &\coloneqq \{G \in \R^{n \times n} : G \text{ symmetric}\},
\end{align*}
via $h(i,i') \coloneqq  G_{i,i'}$, $i,i' \in [n]$.
The space of gm-spaces that are supported on at most
$n$ points
$\GM_{n} \subset \GM$
is given by
\begin{align*}
\GM_{n}
&\coloneqq
\bigl\{ 
 ([n],G,\upsilon)
:
G \in \G^{(n)}_{\sym},
\upsilon \in \p([n])
\bigr\}.
\end{align*}
Let $n \in \N$
and 
$\XX \in \GM$ 
be an arbitrary
gm-space.
We consider the 
\emph{GW quantization problem}
\begin{equation}\label{eq:GW_quant}
    \qGW(\XX) 
    \coloneqq 
    \inf_{\YY \in \GM_{n}}
    \GW_2(\XX,\YY).
\end{equation}
Since GW is invariant under homomorphisms,
so is $\qGW$.
Before we show existence of solutions and characterize them, 
we make the following observation 
in terms of the support size
of elements in $\GM_n$.

\begin{proposition}\label{prop:rep-with-larger-cardinality}
Let $\YY \in \GM_n$.
There exists a representative $\tilde{\YY} = ([n],\tilde{G},\tilde{\upsilon})$
of $\YY$ with fully supported measure $\tilde{\upsilon} \in \p([n])$.
\end{proposition}

The proof,
presented in \cref{subapp:proof_of_3_1},
repeatedly duplicates a positive-mass point 
of $\YY$ into
two copies that inherit its gauge values and share its mass.
Each such split
leaves the gm-space unchanged up to homomorphism while enlarging the support,
until all $n$ points carry positive mass.

\cref{prop:rep-with-larger-cardinality}
allows us to always
consider fully supported representatives of 
any elements in $\GM_{n}$
which we will do throughout.
Similar to \cref{sec:OT},
we encounter
transport plans
of the (semi-discrete) form
$\pi \in \p(X \times [n])$,
where $X$ is some Polish space.
In this setting, 
we use the notation
$
\pi_i \coloneqq \pi(\cdot \times \{i\}) \in \M_+(X),
$
as well as 
$\upsilon_i\coloneqq\upsilon(\{i\})$ 
for any $\upsilon\in\p([n])$,
$i \in [n]$.
With this notation,
the marginal of $\pi$ on $[n]$
can be expressed as
\[
(P_{[n]})_\# \pi 
= \pi(X \times \cdot)
= \sum_{i=1}^n \pi(X \times \{i\})
\cdot \delta_{i}
= \sum_{i=1}^n \pi_i(X) 
\cdot \delta_{i}.
\]
Furthermore, 
the marginal on $X$ can be expressed as
\[
(P_X)_\# \pi
= \pi(\cdot \times [n])
= \sum_{i=1}^n \pi_i.
\]
For $\xi \in \p(X)$, 
we set 
$
\Pi(\xi, *_n) \coloneqq \{\pi \in \p(X \times [n]) : (P_X)_\# \pi = \xi\}.
$
Since the discrete space is fixed as $[n]$,
the measure positions of $(P_{[n]})_\# \pi$, $\pi \in \Pi(\xi, *_n)$,
do not vary.
Instead their geometry varies 
via the associated gauge matrix.
This allows us to uniquely identify
$\pi \in \Pi(\xi, *_n)$
by
$(\pi_i)_{i=1}^n \subset \M_+(X)$,
$\sum_{i=1}^n \pi_i = \xi$.
This stands in contrast to
the Wasserstein setting,
where the measure positions vary directly.
Throughout,
we will frequently make use of this
one-to-one correspondence 
in the GW setting,
that is, we frequently consider $\pi = (\pi_i)_{i=1}^n$.

We turn our attention 
to showing existence of GW quantizations.
We follow a similar alternating approach 
as the one that Lloyd's algorithm for
the Wasserstein case is based on.
That is, 
we analyze
the minimization with respect to the gauge matrix
and the inner GW transport
in the quantization problem \cref{eq:GW_quant}
disjointly at first.
For $\XX = (X,g,\xi)$, 
we set
\[
Q_{\GW}(G,(\pi_i)_{i=1}^n)
\coloneqq
\sum_{i,i'=1}^n \iint_{X^2} (g(x,x') - G_{i,i'})^2 \dx \pi_i(x) \dx \pi_{i'}(x'),
\]
so that the GW quantization problem 
\cref{eq:GW_quant}
can be written as
\begin{equation}\label{eq:qGW_as_min_Q_GW}
\qGW(\XX)^2 = \inf_{
\substack{
G \in \G^{(n)}_{\sym}
\\
(\pi_i)_{i=1}^n \subset \M_+(X), \sum_{i=1}^n \pi_i = \xi 
}
}
Q_{\GW}(G,(\pi_i)_{i=1}^n).
\end{equation}
In the Wasserstein case,
the minimization with respect to the measure positions
admits closed form.
Analogously for the GW case, 
we obtain a similar result
for the minimization 
with respect to the gauge matrix.

\begin{theorem}\label{thm:pointwise_lb_gauge}
    Let $\XX = (X, g, \xi) \in  \GM$ be arbitrary
    and fix
    $\pi \in \p(X \times [n])$ 
    with $(P_X)_\# \pi = \xi$
    and $\pi_i(X) >0$ 
    for all $i \in [n]$.
    Then 
    $
    \inf_{
    G \in \G^{(n)}_{\sym}
    }
    Q_{\GW}(G,(\pi_i)_{i=1}^n)$
    is uniquely solved by 
    \begin{equation}\label{eq:optimal_G_for_GW_quant}
    \hat{G} \in \G_{\sym}^{(n)},
    \qquad 
    \hat{G}_{i,i'}
    \coloneqq
    \frac{1}{\pi_i(X) \pi_{i'}(X)}
    \iint_{X^2} g(x,x')
    \dx \pi_{i}(x)
    \dx \pi_{i'}(x'),
    \quad
    i,i' \in [n].
    \end{equation}
    Set 
    $\YY = ([n],G,\upsilon)$ for $G \in \G_{\sym}^{(n)}$ arbitrary
    and $\hat{\YY} \coloneqq ([n],\hat{G},\upsilon)$,
    where $\upsilon = (P_{[n]})_\# \pi$.
    If additionally 
    $\pi \in \Pio(\XX,\YY)$,
    then the following estimate holds
    \begin{equation}\label{eq:pointwise_lb}
    \GW_2^2(\XX,\YY) \geq
    \GW_2^2(\XX,\hat{\YY})
    + \|G - \hat{G}\|_{L^2(\upsilon \otimes \upsilon)}^2.
    \end{equation}
\end{theorem}

The proof,
presented in \cref{subapp:proof_of_3_2},
hinges analogously to the Wasserstein setting
on the identification 
of the $L^2$ mean.
We note that the optimal gauge $\hat G$ 
coincides with the barycenter-matrix update 
of the GW barycenter algorithm 
in \cite{PCS2016} 
for a single input space and that the same formula also occurs in \cite{chen2023gromov} in a framework related to graph coarsening.

For cost functions $(c_i)_{i=1}^n$ on $X$, the associated \emph{Voronoi partition}
$(V_i)_{i=1}^n$ of $X$ is defined as
\begin{equation}\label{eq:voronoi_partition}
V_i \coloneqq T^{-1}(\{i\}), \qquad
T(x) \coloneqq \min\bigl\{i\in[n] : c_i(x) = \min\nolimits_{j\in[n]} c_j(x)\bigr\}.
\end{equation}
In contrast to the Wasserstein case,
the minimization of
$Q_{\GW}(G,\cdot)$
over $\Pi(\xi, *_n)$
for a fixed gauge matrix
\smash{$G \in \G_\sym^{(n)}$}
does not admit closed form.
This is due to the 
fact that the problem is 
quadratic and non-convex.
Nevertheless,
we can show
the following characterization.

\begin{theorem}\label{thm:Q_GW_measure_min_characterization}
    Let $\XX = (X,g,\xi) \in \GM$ be arbitrary
    and $G \in \G_{\sym}^{(n)}$ be fixed.
    Then $Q_{\GW}(G,\cdot)$ attains its
    minimum over $\Pi(\xi, *_n)$.
    Let 
    $(\hat{\pi}_{i})_{i=1}^n \subset \M_+(X)$
    with $\sum_{i=1}^n \hat{\pi}_i = \xi$
    be any minimizer.
    Then, 
    for every $i \in [n]$,
    the measure $\hat{\pi}_i$ 
    is concentrated on
    \[
    A_i \coloneqq \{x \in X : c_i^{(\hat{\pi})}(x) = \min_{j} c_j^{(\hat{\pi})}(x)\}, ~~ \text{ where } c_i^{(\hat{\pi})}(x) \coloneqq \sum_{i'=1}^n \int_{X} (g(x,x') - G_{i,i'})^2 \dx \hat{\pi}_{i'}(x').
    \]
    If additionally $\xi(A_i \cap A_j) = 0$ for all $i \neq j$,
    then $\hat{\pi}_i = \xi \vert_{V_i}$, $i \in [n]$,
    where $(V_i)_{i=1}^n$ is the Voronoi partition of $X$
    with respect to $(c_i^{(\hat{\pi})})_{i=1}^n$.
\end{theorem}

The proof is given in \cref{subapp:proof_of_3_3}.
Existence follows from
standard arguments.
For the characterization, 
the key step is to show that 
a minimizer 
$\hat{\pi}$
of the quadratic problem also minimizes its linearization at
$\hat{\pi}$,
whose
minimizers satisfy the given properties.

We remark that
the condition $\xi(A_i\cap A_j)=0$ cannot be dropped. 
If $G_{i,i'}=c$ for all $i,i' \in [n]$,
then 
\[
Q_{\GW}(G,(\pi_i)_{i=1}^n)
=\sum_{i,i'=1}^n \iint_{X^2}(g(x,x')-c)^2\dx\pi_i\dx\pi_{i'}
=\iint_{X^2}(g(x,x')-c)^2\dx\xi\dx\xi
\]
is independent of
$(\pi_i)$.
Thus,
every feasible coupling is a minimizer, and correspondingly,
$c_i^{(\hat\pi)}\equiv\int_X(g(\cdot,x')-c)^2\dx\xi$ 
is independent of $i$
so that
$A_i=X$, $i \in [n]$, and
$\xi(A_i\cap A_j)=1$, $i \neq j$.
In particular, minimizers need not be Voronoi partitions.

On the basis of the previous analysis,
we are now able to establish existence for
the GW quantization problem \cref{eq:GW_quant} 
and give a characterization of solutions.

\begin{theorem}\label{thm:GW_quantization_existence_characterization}
    Let $\XX \in \GM$.
    The quantization problem \cref{eq:GW_quant}
    admits a solution. 
    Moreover,
    let 
    $\hat{\YY} = ([n], \hat{G}, \hat{\upsilon}) \in \GM_n$ 
    be any solution of
    \cref{eq:GW_quant}, 
    with fully supported 
    $\hat{\upsilon} \in \p([n])$, 
    see
    \cref{prop:rep-with-larger-cardinality}, 
    and let
    $\hat{\pi} \in \Pio(\XX, \hat{\YY})$. 
    Then the following hold:
    \begin{enumerate}[label=(\roman*)]
        \item \label{item:thm35_gauge}
        $\displaystyle
        \hat{G}_{i,i'}
        = \frac{1}{\hat{\upsilon}_i \hat{\upsilon}_{i'}}
        \iint_{X^2} g(x,x') \,\mathrm{d}\hat{\pi}_i(x) \,\mathrm{d}\hat{\pi}_{i'}(x'),
        \quad i, i' \in [n].$
        \item \label{item:thm35_concentration}
        For every $i \in [n]$, the measure $\hat{\pi}_i$ is concentrated on
        \[
        A_i \coloneqq \Big\{x \in X :
        \hat{c}_i(x) = \min_{j \in [n]} \hat{c}_j(x)\Big\},
        \quad \text{where} \quad
        \hat{c}_i(x) \coloneqq \sum_{i'=1}^n \int_X
        \big(g(x,x') - \hat{G}_{i,i'}\big)^2 \,\mathrm{d}\hat{\pi}_{i'}(x').
        \]
        \item \label{item:thm35_value}
        $\displaystyle
        \qGW(\XX)^2 =
        \GW_2^2(\XX, \hat{\YY})
        = \iint_{X^2} g^2(x,x') \,\mathrm{d}\xi(x) \,\mathrm{d}\xi(x')
        - \sum_{i,i'=1}^n \hat{G}_{i,i'}^2 \,\hat{\upsilon}_i \hat{\upsilon}_{i'}.$
    \end{enumerate}
    If additionally $\xi(A_i \cap A_j) = 0$ for all $i \neq j$, then
    $\hat{\pi} = (\mathrm{Id}, T)_\# \xi$, where $T(x) = i$ if $x \in V_i$, for the Voronoi
    partition $(V_i)_{i=1}^n$ of $X$ with respect to $(\hat{c}_i)_{i=1}^n$.
    In this case, $\hat{\upsilon} = \sum_{i=1}^n \xi(V_i) \cdot \delta_{i}$ and
    \[
    \hat{G}_{i,i'}
    = \frac{1}{\xi(V_i)\xi(V_{i'})}
    \iint_{V_i \times V_{i'}} g(x,x') \,\mathrm{d}\xi(x) \,\mathrm{d}\xi(x'),
    \quad i, i' \in [n].
    \]
\end{theorem}

The strategy for the proof,
presented in \cref{subapp:proof_of_3_4},
is as follows.
General existence of GW quantizers
is shown by
taking a minimizing sequence 
and applying the pointwise 
bound from \cref{thm:pointwise_lb_gauge}.
Then convergence of the gm-spaces 
reduces to convergence of the measures
which subsequently holds due 
to weak compactness.
The characterization is obtained 
by using the connection of the 
quantization problem 
with 
$Q_{\GW}$ 
and the established properties
of pointwise minimizers 
of the latter,
i.e.\ \cref{thm:pointwise_lb_gauge,thm:Q_GW_measure_min_characterization}.

The characterization in 
\cref{thm:GW_quantization_existence_characterization}
induces a
self-consistency condition.
More precisely,
each $\hat{\pi}_i$ is concentrated on the cells 
$A_i$ determined by
$(\hat{c}_i)_{i=1}^n$,
while the costs $\hat{c}_i$ in turn depend
on $\hat{\pi}$. 
Optimal couplings are therefore fixed points
of the assignment map sending
a coupling to 
the one whose slices are concentrated on $A_i$, $i \in [n]$.
In general, 
not every such fixed point is a global
minimizer,
and identifying the spurious ones is
to our knowledge an open problem.
A fixed-point characterization of the same flavor
underlies the computation of 
the tangential GW barycenters 
in \cite{BB2024tangential}.

\section{Gromov--Wasserstein Quantization for Euclidean Spaces}\label{sec:euclidean}

We turn our attention to Euclidean settings,
that is, 
we consider the quantization problem
\cref{eq:GW_quant}
for
$\XX = (X,g,\xi) \in \GM$,
where $X \subset \R^d$,
$\xi \in \p(\R^d)$
and $g$ is either the Euclidean distance,
the squared Euclidean distance
or the Euclidean scalar product.
The set of finite gm-spaces, 
i.e.\ $\bigcup_{n=1}^\infty \GM_n$,
is dense in $\GM$ \cite[Thm.~5.28]{sturm2023space}.
So
for the quantization value $\qGW(\XX)$
defined in \cref{eq:GW_quant},
it holds
$\qGW(\XX)\to 0$, 
as $n\to\infty$,
for every $\XX \in \GM$.
Below we quantify this convergence,
i.e.\ we show that under weak assumptions on
$X \subset \R^d$ and $\xi \in \p(\R^d)$,
$\qGW(\XX)$ 
behaves asymptotically like
$n^{-\frac{1}{d}}$,
for the aforementioned Euclidean gauges.
In each of these cases,
we do this by
showing the existence of constants $c,C > 0$
for which
\begin{equation}\label{eq:sandwich_bounds}
c \cdot \qW(\xi) \leq \qGW(\XX) \leq C \cdot \qW(\xi).
\end{equation}
The constant for the lower bound $c=c(\xi)$ hereby always depends on $\xi$ (and additionally on the diameter of $X$ for the Euclidean gauge), and the constant $C$ is universal for the Euclidean gauge, depending on $X$ for the squared Euclidean gauge, and depending on $\xi$ for the scalar product gauge.
Here $\qW(\xi)$ is taken with respect to 
the Euclidean metric in every case.
The asymptotic rate of $\qGW(\XX)$
is then established by
Zador's theorem (\cite[Thm.~6.2]{GL2000})
which ensures that generally $\qW(\xi)$ 
behaves asymptotically
like
$n^{-\frac{1}{d}}$.
For the special case of the scalar product,
we also 
establish the existence of Monge maps
from $\XX$ to solutions of 
the quantization problem
\cref{eq:GW_quant}
and give a one-to-one correspondence
to the Wasserstein quantization problem
in one dimension.
All proofs of this section are given in \cref{app:euclidean}.

\subsection{Quantization Rate for the Euclidean Distance and Its Square}

In the following we
establish the two-sided bounds 
\cref{eq:sandwich_bounds}
for $g \in \{\|\cdot - \cdot\|, \|\cdot - \cdot\|^2\}$.
The upper bounds rely on the following
more general result on Lipschitz transformations of metrics.

\begin{proposition}\label{prop:GW_quant_upper_bound}
    Let $\XX = (X,g,\xi) \in \GM$ 
    be a gm-space with 
    $g = \varphi \circ d_X$,
    where $d_X$ is a metric on $X$
    and $\varphi:[0,\infty) \to \R$.
    Let  $\qW(\xi)$
    be the Wasserstein quantization value
    defined in \cref{eq:W_quantization}
    with respect to the metric $d_X$.
    If $\varphi$ is $L$-Lipschitz continuous
    on the range of $d_X$, then
    \begin{equation}\label{eq:GW_quant_bound}
        \qGW(\XX)
        \leq 2 L \,\qW(\xi).
    \end{equation}
\end{proposition}

The proof relies on applying 
the Lipschitz bound together
with the triangle inequality
to bound the double GW-like integral
from above by a single Wasserstein-like integral.
It is presented in \cref{subapp:proof_of_4_1}.

\begin{remark}\label{rem:GW_quant_upper_power}
Consider
$\XX =  (X, \|\cdot - \cdot\|^p, \xi) \in \GM$
for some $p \geq 1$.
If $X \subset \R^d$ is bounded,
the previous result implies an upper bound
of $\qGW(\XX)$
via $\qW(\xi)$ 
with respect to the standard Euclidean metric.
More precisely, 
on $[0, \diam(X)]$, 
the function $\varphi(a) = a^p$
satisfies
\[
\lvert \varphi'(a) \rvert = p\,a^{p-1} \leq p\,\diam(X)^{\,p-1},
\]
and is therefore
$p \diam(X)^{\,p-1}$-Lipschitz
continuous.
In this case, 
we obtain
\[
\qGW(\XX) \leq 2p\,\diam(X)^{\,p-1}\,\qW(\xi).
\]
In particular, $\qGW(\XX) \leq 2\,\qW(\xi)$ for $p = 1$ and
$\qGW(\XX) \leq 4\diam(X)\,\qW(\xi)$ 
for $p = 2$.
\end{remark}

In the following,
$\lambda_{\min}(M)$ and $\lambda_{\max}(M)$
denote the smallest, respectively largest, eigenvalue of a positive semi-definite
matrix $M$.
We turn our attention to 
the derivation of the lower bounds.

\begin{theorem}\label{thm:lb_euclidean_gauges}
    Let $\XX = (X,g,\xi) \in \GM$
    with $X \subset \R^d$ closed convex.
    The mean and covariance matrix of 
    $\xi \in \p(X)$
    are set as
    $\bar{m}_\xi \coloneqq \int_X x \dx\xi(x)$, 
    and $\Sigma_\xi \coloneqq 
    \int_X (x - \bar{m}_\xi)(x - \bar{m}_\xi)^\top \dx\xi(x)$,
    respectively. Let $\qW(\xi)$ be defined as in \cref{eq:W_quantization} with respect to the standard Euclidean metric.
    \begin{enumerate}[label=(\roman*)]
        \item \label{item:lb_sq}
        If $\int_X \lVert x \rVert^4 \dx\xi(x) < \infty$ 
        and 
        $g = \lVert \cdot - \cdot \rVert^2$,
        then
        \[
        \qGW(\XX) \geq 2\sqrt{\lambda_{\min}(\Sigma_\xi)} \cdot \qW(\xi).
        \]
        \item \label{item:lb_dist}
        If $X$ is bounded 
        and 
        $g = \|\cdot - \cdot\|$,
        then
        \[
        \qGW(\XX) \geq \frac{\sqrt{\lambda_{\min}(\Sigma_\xi)}}{\diam(X)} \cdot \qW(\xi).
        \]
    \end{enumerate}
\end{theorem}

The proof, presented in
\cref{subapp:proof_of_4_3},
relies on 
lower bounding the GW value
as the average within-cell variance of 
the gauge
$g(\cdot,y)$, $y \in X$.
The latter can in turn be related
to the Wasserstein quantization value.

\begin{corollary}\label{cor:GW_quant_rate_1}
    Let $\XX =  (X, g, \xi)\in \GM$ 
    with $X \subset \R^d$ bounded
    and $\xi \in \p(X)$ be absolutely continuous.
    For $g \in \{\|\cdot - \cdot\|, \|\cdot - \cdot\|^2\}$,
    it holds
    $\qGW(\XX) \asymp n^{-1/d}$,
    that is
    \[
    \limsup_{n \to \infty} n^{1/d}\, \qGW(\XX) < \infty
    \quad \text{and} \quad 
    \liminf_{n \to \infty} n^{1/d}\, \qGW(\XX) > 0.
    \]
\end{corollary}

\begin{proof}
    Without loss of generality, 
    we assume that $X \subset \R^d$ is additionally closed and convex.
    Otherwise,
    we can exchange $\XX$ 
    with the representative
    whose base space is the closed convex hull of $X$
    and $g,\xi$ unchanged.
    Firstly,
    note that
    the convergence of 
    $\qW(\xi)$
    with respect to the standard Euclidean metric on $\R^d$
    is $n^{-\frac{1}{d}}$
    due to Zador's theorem 
    (\cite[Thm.~6.2]{GL2000}).
    So it suffices in both cases 
    to bound $\qGW(\XX)$
    from below and above by
    $\qW(\xi)$ 
    up to a positive multiplicative constant.
    The upper bound follows directly by 
    \cref{prop:GW_quant_upper_bound},
    see \cref{rem:GW_quant_upper_power}.
    The lower bounds follow by \cref{thm:lb_euclidean_gauges}.
    It remains to establish that the smallest eigenvalue of 
    $\Sigma_\xi = \int_X (x -\bar{m}_\xi)(x- \bar{m}_\xi)^\top \dx \xi(x)$,
    with $\bar{m}_\xi = \int_X x \dx \xi(x)$,
    is non-zero.
    For any $v \in \R^d$ with $v \neq 0$, it holds
    \[
    v^\top \Sigma_\xi v = \int_X (v^\top (x - \bar{m}_\xi))^2 \dx \xi(x) > 0.
    \]
    The last inequality holds as
    $\{x \in X : v^\top (x - \bar{m}_\xi) = 0\}$
    is a Lebesgue null-set.
    Therefore, 
    $\lambda_{\min}(\Sigma_\xi) > 0$
    as desired.
\end{proof}

\subsection{The Scalar Product: Quantization Rate and Monge Maps}

We turn our attention to the scalar product case,
i.e.\ throughout we consider 
$\XX = (X,\langle \cdot, \cdot \rangle, \xi) \in \GM$
with $X \subset \R^d$.
In this setting,
the quantization problem admits 
an accessible alternative formulation.
It allows us to derive
two-sided bounds of the form \cref{eq:sandwich_bounds}
and establish the existence of Monge maps
from $\XX$ to a quantizer.

\begin{theorem}\label{thm:lower_and_upper_bound_scalar_product}
    Let $X \subset \R^d$ be closed convex and
    $\XX = (X, \langle \cdot, \cdot \rangle, \xi) \in \GM$.
    Consider the symmetric 
    positive semi-definite matrix
    $M_\xi = \int_X x x^\top \dx \xi(x) \in \R^{d \times d}$.
    The GW quantization value admits the bounds
    \[
    \sqrt{\lambda_{\min}(M_\xi)} \cdot \qW(\xi)
    \leq 
    \qGW(\XX)
    \leq 
    \sqrt{2 \lambda_{\max}(M_\xi)} \cdot \qW(\xi),
    \]
    where $\qW(\xi)$ is defined in \cref{eq:W_quantization} 
    with respect to the standard Euclidean metric.
\end{theorem}

The proof of this result is
provided in \cref{subapp:proof_of_4_5}.
Therein,
we also establish that 
$\xi \in \p^{(2)}(X)$,
so that
$M_\xi$ and $\qW(\xi)$ 
are well-defined.
The proof relies on auxiliary results
that express
$\qGW(\XX)^2$
as an infimum of the within-cell variance
which is weighted by a matrix contained
between $M_\xi$ and $2M_\xi$.
This bounds
$\qGW(\XX)^2$ between 
$\lambda_{\min}(M_\xi)$ and $2\lambda_{\max}(M_\xi)$ times
the unweighted within-cell variance, 
whose infimum over couplings is
$\qW(\xi)^2$.

\begin{corollary}[Quantization rate]\label{cor:GW_quant_rate_2}
    Let $\XX =  (X, \langle \cdot, \cdot \rangle, \xi)\in \GM$ 
    with $X \subset \R^d$
    and $\xi \in \p(X)$ be absolutely continuous.
    If
    $\int_X \|x\|^{2+\delta} \dx\xi < \infty$ 
    for some $\delta > 0$,
    then
    $\qGW(\XX) \asymp n^{-1/d}$,
    that is
    \[
    \limsup_{n \to \infty} n^{1/d}\, \qGW(\XX) < \infty
    \quad \text{and} \quad 
    \liminf_{n \to \infty} n^{1/d}\, \qGW(\XX) > 0.
    \]
\end{corollary}

\begin{proof}
    Due to the bounds from \cref{thm:lower_and_upper_bound_scalar_product},
    the proof is analogous to \cref{cor:GW_quant_rate_1}.
\end{proof}

The lower constant in 
\cref{eq:sandwich_bounds} 
is necessarily measure-dependent.
In 
\cref{app:example_two_sided_bounds}
we exhibit a family of gauged measure
spaces $\XX^{(n)}$ 
with measures $\xi^{(n)}$, 
valid for all three Euclidean gauges, 
with $\qGW(\XX^{(n)})/\qW(\xi^{(n)}) \to 0$, 
ruling out any uniform
lower constant $c > 0$.

The existence of 
\emph{Monge maps}---optimal couplings concentrated
on the graph
of a map---is a recurring theme in the realm of optimal transport.
In the GW setting,
for the inner-product cost,
\cite{DLV2024} establish the existence of an optimal
Monge map under absolute continuity of the source measure, 
whereas for the
squared-distance cost they obtain, 
in general, 
only an optimal $2$-map.
The latter is a plan that couples source masses 
to at most 2 targets.
In the semi-discrete, inner-product setting, 
\cite{RGK2023} likewise establish
the existence of Monge maps, 
but under absolute continuity of $\xi$
together with a non-degeneracy condition on the discrete marginal 
(guaranteed,
e.g., when its support points have pairwise distinct norms).
In the following,
we present a result ensuring the existence of Monge maps
for the quantization problem \cref{eq:GW_quant}
in the scalar-product setting
that requires neither absolute continuity of $\xi$
nor any such non-degeneracy condition,
since here the discrete object is itself subject to optimization.

For the following, recall
\[
\m_\xi(V) = \frac{1}{\xi(V)} \int_{V} x \dx \xi(x), \qquad V \subset X \text{ measurable},
\]
where $\m_\xi(V) \in X$ is set arbitrarily for $\xi(V) = 0$.

\begin{theorem}[Monge Maps for Scalar Product Spaces]\label{thm:scalar_product_partition_attainment}
    Let $\XX = (X, \langle \cdot, \cdot \rangle, \xi) \in \GM$
    with $X \subset \R^d$.
    Then there exists a measurable map 
    $T \colon X \to [n]$ 
    such that, for the disjoint partition 
    $V_i \coloneqq T^{-1}(\{i\})$, $i \in [n]$, 
    the gm-space 
    \[
    \hat{\YY} \coloneqq ([n], \hat{G}, \hat{\upsilon}) \in \GM_n, 
    \qquad
    \hat{G}_{i,i'} \coloneqq \langle \m_\xi(V_i), \m_\xi(V_{i'}) \rangle, 
    \qquad
    \hat{\upsilon} \coloneqq \sum_{i=1}^n \xi(V_i) \cdot \delta_{i},
    \]
    is a solution of the GW quantization problem \cref{eq:GW_quant} with respect to $\XX$,
    and 
    $(\id, T)_\# \xi \in \Pio(\XX, \hat{\YY})$.
\end{theorem}

The proof, 
presented in \cref{subapp:proof_of_4_8},
establishes the result
by exploiting the reformulation of $\qGW(\XX)$
that we established in \cref{subapp:proof_of_4_5}
for the proof of 
\cref{thm:lower_and_upper_bound_scalar_product}.
The latter characterizes $\qGW$ 
as the minimization
of a concave function over the set $\Pi(\xi, *_n)$.
By the Bauer maximum principle, 
an extreme point of $\Pi(\xi, *_n)$ is a solution.
Showing that the extreme points are
map-induced plans is a standard argument.

The previous theorem shows 
that in the scalar-product setting
$\XX=(X,\langle\cdot,\cdot\rangle,\xi)$ 
with $X\subset\R^d$, 
an optimal quantizer is
induced by a partition $(V_i)_{i=1}^n$ of $X$: 
the optimal coupling is the hard assignment
$\pi_i=\xi\vert_{V_i}$ and the optimal gauge is the block-mean gauge
$\hat G_{i,i'}=\langle \m_\xi(V_i),\m_\xi(V_{i'})\rangle$, 
$i,i' \in [n]$.
The quantization problem
therefore reduces to the partition problem
\begin{equation}\label{eq:GW_quant_as_partitioning_problem}
\qGW(\XX)^2
=
\inf_{
\substack{
V_1,\dotsc,V_n \subset X \text{ disjoint}\\
\bigcup_{i=1}^n V_i = X
}
}
\sum_{i,i'=1}^n \int_{V_i \times V_{i'}}
\bigl(
\langle x, x' \rangle - \langle \m_\xi(V_i),\m_\xi(V_{i'}) \rangle
\bigr)^2
\dx \xi(x) \dx \xi(x').
\end{equation}
This is the natural GW analogue of the Wasserstein partition
formulation \cref{eq:W_quant_as_partitioning_problem}: 
the Wasserstein objective
penalizes the squared deviation of each point $x$ from its cell barycenter
$\m_\xi(V_i)$, 
whereas the GW objective penalizes the squared deviation of each
pairwise scalar product value $\langle x,x'\rangle$ 
from the scalar product value of the block means
$\langle \m_\xi(V_i),\m_\xi(V_{i'})\rangle$ 
over $V_i\times V_{i'}$.
We conclude the section by noting that,
in one dimension,
the Wasserstein and GW quantization problems
reduce to the same problem.

\begin{theorem}\label{thm:closed_form_1d}
Let 
$\XX=(X,\langle\cdot,\cdot\rangle,\xi) \in\GM$ 
with $X\subseteq\spann(u)$,
$u\in\R^d$, $\|u\|=1$, 
and set $T\colon X\to\R$, 
$T(x)=u^\top x$
as well as
$M_2(T_\# \xi) \coloneqq\int_\R t^2\dx(T_\# \xi)(t)$. 
Then
\[
\qGW(\XX)^2=\qW(T_\# \xi)^2\bigl(2M_2(T_\# \xi)-\qW(T_\# \xi)^2\bigr).
\]
\end{theorem}

The proof, 
presented in \cref{subapp:proof_of_4_9}, 
uses
that $\XX$ is homomorphic to 
the one-dimensional gm-space 
$(\R, (s,t) \mapsto st, T_\# \xi)$
and a direct reformulation 
of the appearing integrals
in one dimension.

\section{Algorithm and Convergence}\label{sec:alg}

We turn 
to the practical approximation of
solutions to the GW 
quantization problem \cref{eq:GW_quant}
for which we propose
\cref{alg:1}.
\begin{algorithm}[t]
	\begin{algorithmic}[1]
	    \State \textbf{Input:} 
	    \parbox[t]{300pt}{gm-space $\XX \coloneqq (X,g,\xi)$
        \\
	    number of support points $n \in \N$
        }
        \State Initialize $\pi \in \Pi(\xi, *_n)$ with $\pi_i(X) > 0$
        and set $\upsilon_i \gets \pi_i(X)$
	    \While{not converged}
            \State
            $\hat{G}(\pi)_{i,i'}
            \gets
            \frac{1}{\upsilon_i \upsilon_{i'}}
            \iint_{X^2} g(x,x')
            \dx \pi_{i}(x)
            \dx \pi_{i'}(x')$,
            \Comment{{\footnotesize indices $i,i'$ range over $\supp(\upsilon)$ throughout}}
            \State
            $
            \hat{c}_i^{(\pi)}
            \gets
            \sum_{i'=1}^n \int_X
            \big( g(\cdot,x') - \hat{G}(\pi)_{i,i'} \big)^2
            \dx \pi_{i'}(x'),
            $
            \State
            $\pi_i \gets (1-t)\,\pi_i + t\,\xi|_{V_i}$
            for a Voronoi partition 
            $(V_i)_{i=1}^n$ 
            of $X$ 
            with respect to $(\hat{c}_i^{(\pi)})_{i=1}^n$ 
            and a step size $t \in [0,1]$,
            \State 
            $\upsilon_i \gets \pi_i(X)$
	    \EndWhile
        \State \textbf{Output:} Quantizer $\YY=([n],\hat G(\pi),\upsilon)$, coupling $\pi \in \Pi(\xi,\upsilon)$
        \Comment{{\footnotesize $\hat{G}(\pi) = 0$ outside $\supp(\upsilon)^2$}}
	\end{algorithmic}
	\caption{Conditional Gradient for GW Quantization}
	\label{alg:1}
\end{algorithm}
While the initial idea for the algorithm arose as an analogue to Lloyd's algorithm, we formally derive the algorithm as a conditional gradient method
of a reduced objective $R$ of $Q_{\GW}$ 
defined below. The algorithm is stated for generic convergence criteria. For our experiments, we observe the Voronoi partitions becoming stationary fairly quickly, which we use as the convergence criterion. Regarding the choice of the step size, see the discussion after \cref{prop:R_directional_derivative}.

Consider an input gm-space
$\XX \coloneqq (X,g,\xi)$
and fix the support size of
the quantization $n \in \N$.
We define the gauge-reduced objective
\begin{align}\label{eq:R_def}
R(\pi) \coloneqq 
\inf_{G \in \G_{\sym}^{(n)}} 
Q_{\GW}(G, \pi), 
\qquad \pi \in \Pi(\xi, *_n).
\end{align}
For every 
$\pi \in \Pi(\xi, *_n)$ 
with $\pi_i(X) > 0$, $i \in [n]$, 
the minimization in \cref{eq:R_def} 
is uniquely attained at
\[
\hat G(\pi)_{i,i'} \coloneqq 
\frac{1}{\pi_i(X)\,\pi_{i'}(X)} \iint_{X^2} g \dx\pi_i \dx\pi_{i'},
\qquad i,i' \in [n],
\]
see \cref{thm:pointwise_lb_gauge}.
Thus,
$R(\pi) = Q_{\GW}\bigl(\hat G(\pi), (\pi_i)_{i=1}^n\bigr)$ 
and $\qGW(\XX)^2 = \inf_{\pi \in \Pi(\xi, *_n)} R(\pi)$. 
To analyze the gradient of $R$, 
set
\[
\hat{c}_i^{(\pi)}(x) \coloneqq 
\sum_{i'=1}^n \int_X \bigl( g(x,x') - \hat G(\pi)_{i,i'} \bigr)^2 \dx\pi_{i'}(x'),
\qquad i \in [n].
\]
Since $\pi_{i'} \leq \xi$, $i' \in [n]$,
and $g \in L^2(\xi \otimes \xi)$, 
we have $\hat c_i^{(\pi)} \in L^1(\xi)$. 

The 
directional derivative 
of the reduced objective $R$
can be derived similarly to the one
of the standard
GW objective
\cite{PCS2016, vayer2020fused}.

\begin{proposition}\label{prop:R_directional_derivative}
    Let $\XX = (X,g,\xi) \in \GM$, $n \in \N$, 
    and let $\pi \in \Pi(\xi, *_n)$ with $\pi_i(X) > 0$ for all $i \in [n]$. 
    Then, for every $\eta \in \Pi(\xi, *_n)$, 
    the one-sided directional derivative of $R$ at $\pi$ exists and equals
    \[
    \lim_{t \downarrow 0} 
    \frac{R\bigl((1-t)\pi + t\eta\bigr) - R(\pi)}{t}
    = 2 \sum_{i=1}^n \int_X \hat c_i^{(\pi)} \dx(\eta_i - \pi_i).
    \]
\end{proposition}

The proof is given in \cref{subapp:proof_of_6_1}. 
It relies on the fact that,
along the segment
$\pi_t = (1-t)\pi + t\eta$, 
the block-mean gauge $\hat{G}(\pi_t)$ is a ratio of
polynomials of $t$ with non-vanishing denominator,
and thus differentiable.
Differentiating $R$ at $t = 0$
and simplifying yields the stated linear form.

We minimize $R$ 
over the convex set 
$\Pi(\xi, *_n)$ 
by conditional gradient descent: 
at a current 
$\pi$ 
we move toward a minimizer of the linearization 
of $R$, 
that is, 
of the directional derivative 
of \cref{prop:R_directional_derivative}. 
The linear functional 
$\eta \mapsto \sum_{i=1}^n \int_X \hat c_i^{(\pi)} \dx\eta_i$ 
is minimized over $\Pi(\xi, *_n)$ 
exactly by those $\eta$, 
such that $\eta_i$ is concentrated on 
$\{ x \in X : \hat c_i^{(\pi)}(x) = \min_j \hat c_j^{(\pi)}(x)\}$
in particular by the Voronoi assignment $(\xi|_{V_i})_{i=1}^n$ 
(cf.\ \cref{lem:linear_assignment}),
 $i \in [n]$.
Hence \cref{alg:1} is the conditional-gradient method for 
$\min_{\pi \in \Pi(\xi, *_n)} R(\pi)$: 
line 4 evaluates the gradient through $\hat G(\pi)$, 
lines 5--6 determine a minimizer $\xi|_{V_i}$ of the linearization, 
and the convex combination in line 6 is the conditional-gradient step.

Taking the unit step size ($t = 1$)
in line 6 
ensures the hard assignment 
$\pi_i = \xi|_{V_i}$ 
which may be of particular interest 
when seeking to cluster $\XX$.
In this case, 
it may happen that clusters attain zero mass, 
i.e.\ $\upsilon_i = 0$.
This contrasts the case $t \in [0,1)$
where clusters masses remain positive throughout.
We also note, 
that the unit step $t=1$ need not decrease
$R = Q_{\GW}(\hat{G}(\cdot),\cdot)$ in general.
More precisely,
by
\cref{prop:R_directional_derivative} 
the conditional-gradient direction is one
of descent, 
so $R$ decreases for small $t>0$, 
but without concavity the full
step may increase it.

\begin{remark}[Line Search]\label{rem:line_search}
    Let $(V_i)_{i=1}^n$ be the Voronoi partition
    of $X$ with respect to $(\hat{c}_i^{(\pi)})_{i=1}^n$.
    A monotone decrease of $R$ 
    is guaranteed 
    by the line search 
    \[
    t \in \argmin_{s \in [0,1]} 
    Q_{\GW}\bigl(\hat G(\pi), \pi + s(\eta - \pi)\bigr), 
    \qquad 
    \eta_i \coloneqq \xi|_{V_i}, 
    \quad 
    i \in [n].
    \]
    This is the minimization 
    of the quadratic 
    functional
    $s \mapsto R(\pi) - s\, b + s^2 a$,
    where
    \[
    b \coloneqq -2\sum_{i=1}^n 
    \int_X \hat c_i^{(\pi)} 
    \dx (\eta_i - \pi_i), 
    \quad 
    a \coloneqq 
    \sum_{i,i'=1}^n \iint_{X^2} 
    \bigl( g - \hat G(\pi)_{i,i'} \bigr)^2 
    \dx (\eta_i - \pi_i) \dx (\eta_{i'} - \pi_{i'}).
    \]
    Note that $b \geq 0$ 
    since $\eta$ minimizes the linearization. 
    Hence $t = 1$ if $a \leq 0$ 
    and $t = \min\{1, b/(2a)\}$ if $a > 0$, 
    and in either case $R\bigl((1-t)\pi + t\eta\bigr) \leq R(\pi)$.
\end{remark}

We can use the line search as described in \cref{rem:line_search}
to produce a sequence $(\pi^{(k)})_{k \in \N}$ 
which monotonically decreases
$R$, that is
\[
R(\pi^{(k+1)}) \leq R(\pi^{(k)}).
\]
Since $R$ is non-negative, 
this ensures convergence of
$(R(\pi^{(k)}))_{k \in \N}$.
At the same time,
since $\Pi(\xi, *_n)$ is weakly compact,
the sequence $(\pi^{(k)})_{k \in \N}$
admits cluster points.
Under certain assumptions,
these cluster points 
are also stationary for $R$ 
as the following result shows.

\begin{theorem}\label{thm:convergence}
Let $\XX=(X,g,\xi)\in\GM$
with bounded gauge $g$.
Let $n\in\N$, 
and let
$(\pi^{(k)})_k\subset\Pi(\xi, *_n)$ 
be generated by \cref{alg:1} 
with line search as described in \cref{rem:line_search}.
Assume the iterates satisfy $\pi^{(k)}_i(X)>0$ 
for all $i\in[n]$, $k\in\N$.
Then every cluster point
$\pi^\ast$ of $(\pi^{(k)})_{k \in \N}$ 
with $\pi^\ast_i(X)>0$ 
for all $i\in[n]$ is stationary,
that is, 
no feasible direction is one of descent:
\[
2\sum_{i=1}^n\int_X \hat c_i^{(\pi^\ast)}\dx(\eta_i-\pi^\ast_i)\ \geq\ 0
\qquad\text{for all }\eta\in\Pi(\xi, *_n).
\]
\end{theorem}

The proof is given in \cref{subapp:proof_of_6_3} and follows the standard
conditional-gradient (Frank--Wolfe) argument.
More precisely,
the line search enforces monotone
decrease of $R$, so by \cref{prop:R_directional_derivative} no feasible
direction at a cluster point is one of descent.
Note that the positivity assumption 
$\pi^{(k)}_i(X)>0$, $i\in[n]$, $k\in\N$
is satisfied whenever the step sizes are
strictly less than $1$
when starting from an initial $\pi \in \Pi(\xi, *_n)$
with $\pi_i(X) > 0$, for all $i \in [n]$. 

We remarked before that for the unit step size,
monotone decrease of $R$ is not ensured.
The following result shows that
the scalar product case is an exception, 
where concavity of $R$ renders the unit step monotone.

\begin{corollary}\label{cor:scalar_product_convergence}
    Let $X \subset \R^d$ and $g = \langle \cdot, \cdot \rangle$. 
    Then $R$ is concave on $\Pi(\xi, *_n)$.
    Let $(\pi^{(k)})_{k \in \N}$ 
    be produced by \cref{alg:1} 
    with unit step size
    $t = 1$, and assume 
    $\pi^{(k)}_i(X) > 0$ for all 
    $i \in [n]$, $k \in \N$.
    Then
    \[
    R(\pi^{(k+1)}) \leq R(\pi^{(k)}), 
    \qquad k \in \N.
    \]
    If moreover $X$ is finite, 
    the iteration reaches, in finitely many steps, 
    a coupling $\pi^\star$ such that,
    for each $i \in [n]$, 
    $\pi^\star_i$ is concentrated on 
    $\{ x \in X : \hat c_i^{(\pi^\star)}(x) = \min_j \hat c_j^{(\pi^\star)}(x) \}$.
\end{corollary}

The proof is given in \cref{subapp:proof_of_6_4}.
Concavity bounds the unit-step increment by the directional derivative
of \cref{prop:R_directional_derivative},
which is non-positive since the update minimizes the linearization
(\cref{lem:linear_assignment}), giving monotone decrease.
For finite $X$ the iterates are polytope vertices across which $R$
strictly decreases, forcing termination.

\section{Numerical Examples}\label{sec:numerics}
In the following, we provide various numerical experiments. 
These experiments serve primarily three goals: 
First, showcase that the proposed GW quantization is numerically feasible and leads to structurally different results compared to normal Wasserstein clustering (cf.~particularly \Cref{subsec:exp-mesh}). Second, show that the added modeling freedom provided by gm-spaces is useful for certain applications and that the quantization problem can also be solved effectively in such settings (cf.~\Cref{subsec:pruningnn}). And third, show that in certain simple cases, the proposed algorithm leads to approximation rates in line with the theory for optimal quantization (cf.~\Cref{subsec:numerics_rates}) and that the approximation quality leads to desirable benefits in downstream tasks (cf.~\Cref{subsec:accelerated_pairwise}).

We emphasize that all experiments should be considered as proof-of-concept, and not (yet) improving on state-of-the-art performance for important practical tasks. Nevertheless, we believe these experiments nicely showcase the practical possibilities that GW quantization can provide for data science applications.

All experiments are implemented in Python.\footnote{
The source code is publicly available at 
\url{https://github.com/florian-beier/gw-quantization}.
}
In addition to the self-implemented \cref{alg:1}, we use 
the \textsf{fast\_simplification} package for mesh decimation, as well as \textsf{POT}
\cite{flamary2021pot} for some GW distance and coupling
computations. 
Accordingly, the reported values for the $\GW$-problem are numerical approximations of the true values arising from the respective algorithms used in the \textsf{POT} package.

\subsection{Quantization of a 3D Shape}\label{subsec:exp-mesh}

In this example,
we quantize a 3D surface 
and visualize the resulting quantizers 
at decreasing resolutions.
The surface is a 3D registration 
of a human subject taken
from the training part of the
FAUST dataset~\cite{bogo2014faust}.
It consists of $6890$ vertices and $13\,776$ faces.
We regard the surface as a graph 
whose nodes are the mesh vertices and whose
edges are the edges of the triangulation.
Each edge is weighted by the Euclidean distance between
the vertices it connects.
We convert the 3D shape to a gm-space 
$\XX = (X,g,\xi)$ as follows.
The space $X \subset \R^3$ is set as the vertex set,
the gauge $g$ consists of
the pairwise Dijkstra distances
on the weighted graph,
and $\xi$ is the normalized area-weighted vertex measure, 
i.e.\ the mass of a vertex is proportional to
the total area of its incident triangles.
Note that the gauge is an approximation 
of the geodesic distance on the original surface.

We run \cref{alg:1} for various $n < 6890$
with unit step size $t = 1$.
To initialize the method,
we construct a subset $X_0 \subset X$
by greedily adding
farthest points 
from the current $X_0$.
More precisely,
starting with $X_0 = \{x_1\}$ 
for some initial choice $x_1$,
we add any element from
\[
\argmax_{x \in X} \; \min_{x' \in X_0} g(x, x'),
\]
to $X_0$.
This step is repeated 
until $X_0$ has $n$ distinct elements, 
that is, $X_0 = \{x_1,\dotsc,x_n\}$.
Then \cref{alg:1} is initialized 
with $\pi_i = \xi \vert_{V_i}$,
$i \in [n]$,
where $(V_i)_{i=1}^n$ 
is the Voronoi partition of $X$
with respect to $(g(\cdot,x_i))_{i=1}^n$.
We also apply the Wasserstein quantization
algorithm (Euclidean $k$-means) 
for the same values of $n$
and the same initialization $\pi$.
Both algorithms converge after fewer than 100 iterations,
and return a partitioning 
$(V_i)_{i=1}^n$ of the vertex set $X$
for each $n$.
From the partitions a 3D graph structure 
is reconstructed
as follows.
We associate each cluster $V_i$
with its mean, 
that is
\[
y_i \coloneqq \frac{1}{\xi(V_i)}\sum_{x \in V_i} \xi(\{x\})\, x,
\qquad i \in [n],
\]
and we place an edge between 
$y_i$ and $y_j$, $i \neq j$, 
whenever there exist
$x_i \in V_i$ and $x_j \in V_j$ 
that are connected by an edge 
in the original graph.
The resulting 3D graphs for both methods and various $n$
are shown in \cref{fig:shape_quantization}.

\begin{figure}[t]
  \centering
  \includegraphics[width=\textwidth]{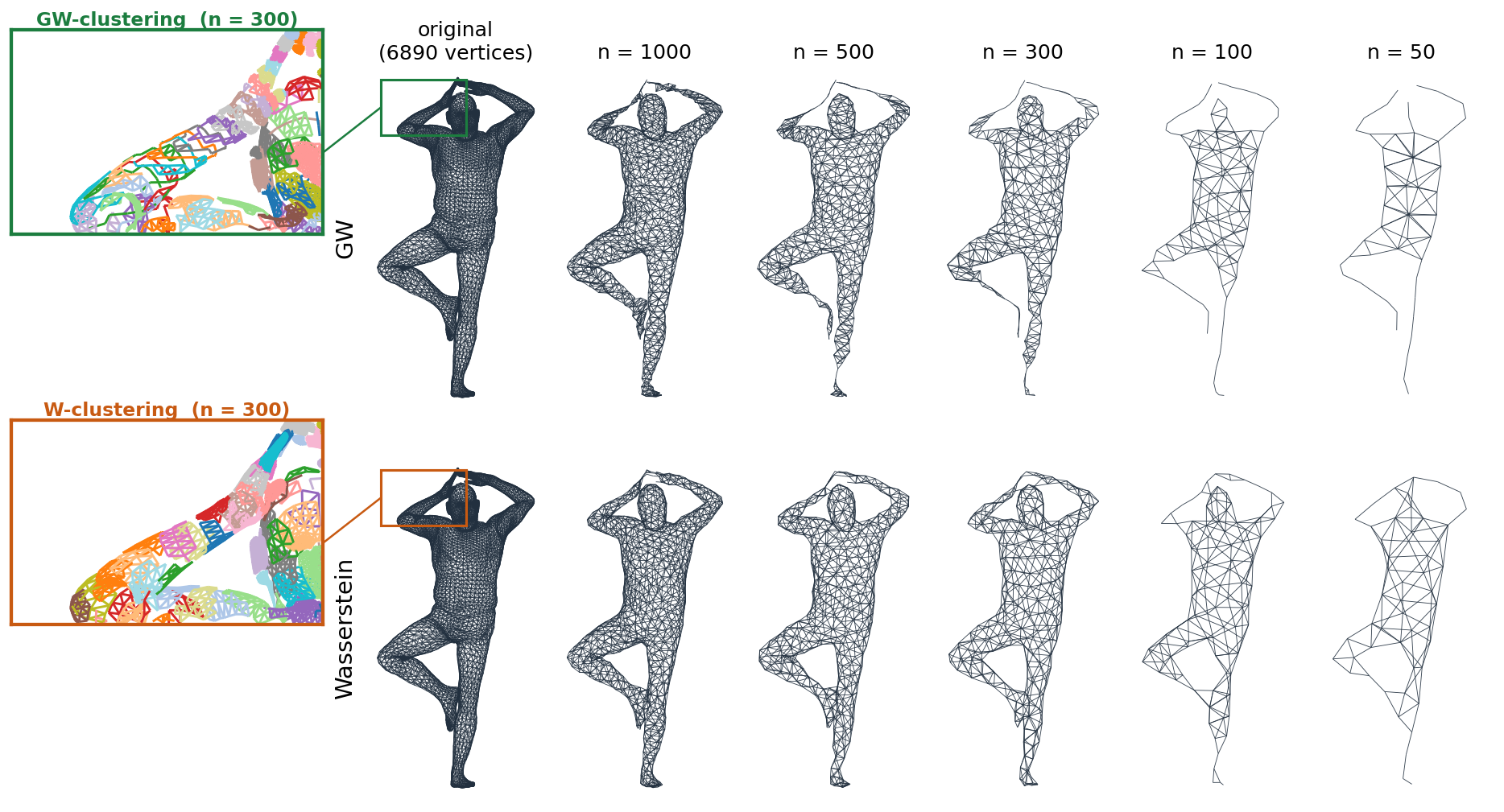}
  \caption{Progressive quantization of a 3D surface
  for $n = 1000, 500, 300, 100, 50$
  together with the original.
  Top: GW
  quantizers with geodesic gauge. 
  Bottom: Wasserstein quantizers with respect to the Euclidean metric.
  Insets (left): a magnified view of the subject's right forearm and hand at
  $n=300$, with each mesh edge colored by the cluster shared by its two
  endpoints -- GW (green frame) versus Wasserstein (orange frame).
  }
  \label{fig:shape_quantization}
\end{figure}

The subject exhibits multiple (almost) self-contacts
between the hands and head,
as well as between both legs.
We notice that the Wasserstein approach merges
the regions of self-contact across all quantizations,
drastically altering the topology of the induced graph.
This is due to the fact
that vertices of these near self-contact regions
possess small Euclidean distances to
vertices of the opposing contact region, 
making them likely to be clustered together.
In contrast,
the GW approach allows us to set the focus
to cluster according to the geodesic
(rather than Euclidean) distance,
so that these self-contact regions are not merged.
We deliberately chose a pose with near self-contacts 
for this experiment
in order to illustrate the distinct behavior of the two methods.
The figure also shows that the GW-based approach
seems to quantize the limbs of the subject
to a one-dimensional geometry faster
than the Wasserstein-based approach with decreasing $n$.
This effect can be seen well when considering the results for $n=300$.
Here, the forearms and hands, 
as well as the lower legs and feet
are one-dimensional lines in the GW case,
but are largely still 3D in the Wasserstein case.
The insets of \cref{fig:shape_quantization} make this explicit for the forearm:
colouring each mesh edge by the cluster of its endpoints, the GW clusters appear
as bands wrapping successive cross-sections, whereas the Wasserstein clusters are
compact patches that ignore the circumference.
Since a limb is thin compared to its length, 
the GW quantization can
collapse each cross-section at almost no cost. 
More precisely,
two points
$x, x' \in X$ on opposite sides of a thin cross-section 
have almost identical
gauge profiles, $g(x,\cdot) \approx g(x',\cdot)$. 
There is thus an incentive to
cluster cross-sections, 
which roughly preserves the distances along the limb.
The Wasserstein quantization objective, 
in contrast,
sees only ambient position
and forms compact three-dimensional clusters. 
Thereby the full thickness of a
limb is retained until $n$ is small enough
that the diameter of these clusters
exceeds the limb thickness.

\subsection{Accelerated Pairwise Distance Computation}\label{subsec:accelerated_pairwise}

In this example,
we show how \cref{alg:1} can be used to
accelerate the computation of all pairwise GW
distances within a dataset of 3D shapes.
We turn to the
deformation-transfer dataset of Sumner and
Popovi\'c~\cite{mesh3d_animals}
and extract exactly one triangulated mesh
of each of the 8 classes
camel, cat, elephant, horse, lion, face, head, flamingo.
We pre-process each mesh and extract a gm-space similarly to \Cref{subsec:exp-mesh}.
More precisely, each mesh is
decimated to approximately 
1000 vertices,
see \cref{fig:mesh_gallery}.
\begin{figure}[t]
  \centering
  \includegraphics[width=0.7\textwidth]{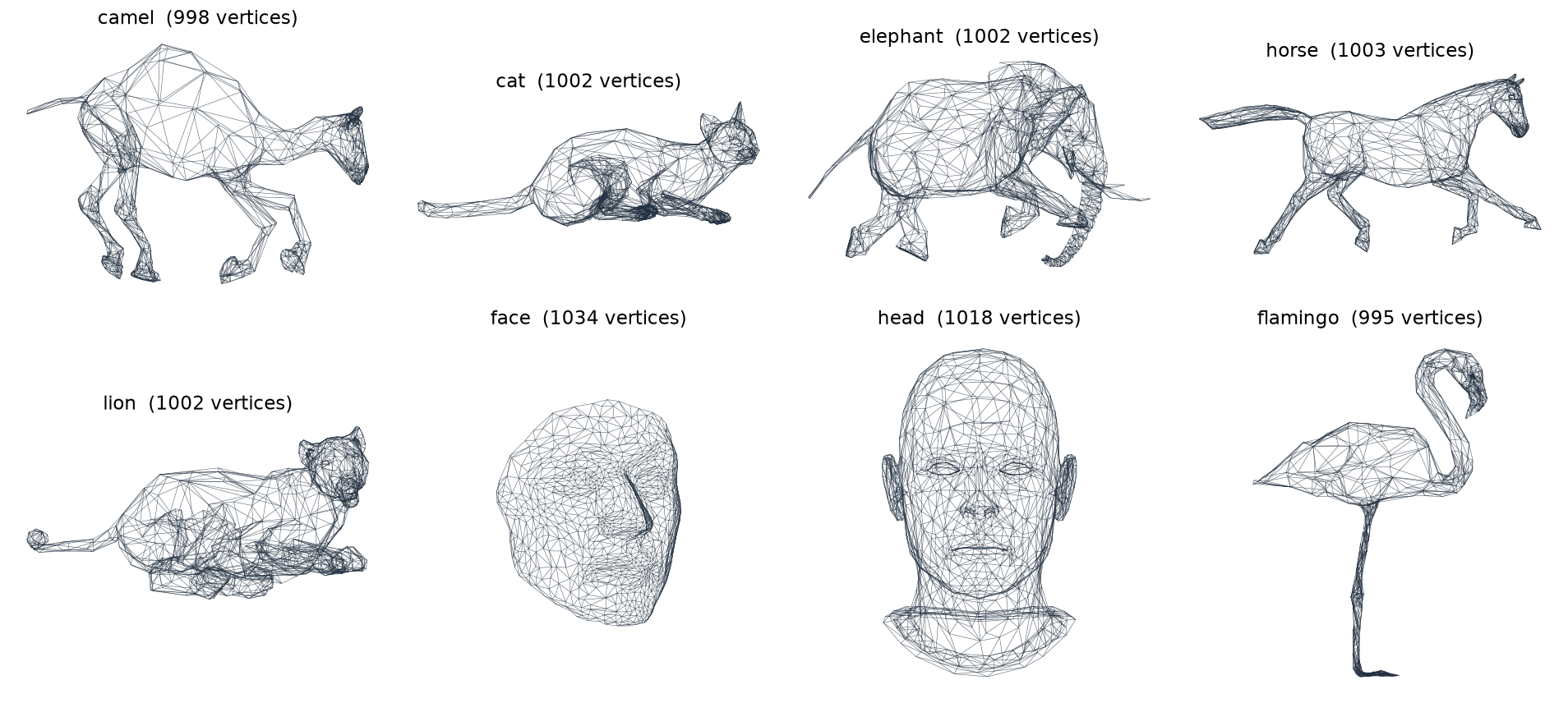}
  \caption{A set of 8 decimated 3D surfaces from Sumner and
Popovi\'c~\cite{mesh3d_animals}.}
  \label{fig:mesh_gallery}
\end{figure}
Then,
for each of the 8 meshes, 
we extract
a gm-space $\XX_k = (X_k,g_k,\xi_k)$,
where 
$X_k \subset \R^3$ are the vertex positions,
$g_k$ are the normalized pairwise Dijkstra distances of 
vertices within the 
graph induced by the mesh
and $\xi_k$ is 
the normalized area-weighted vertex measure
on $X_k$,
$k = 1,\dotsc,8$.
We consider the task of computing all
pairwise distances 
between $\XX_1,\dotsc,\XX_8$
which amounts to 
$\binom{8}{2} = 28$
GW computations.
We use our quantization machinery 
to reduce the computational cost.
More precisely,
fixing some $n < 1000$,
we propose to quantize each object
$\XX_k$ using \cref{alg:1},
yielding some $\YY_k$, 
$k=1,\dotsc,8$.
The initialization of the algorithm is
exactly done as in \Cref{subsec:exp-mesh}.
Then we approximate 
\begin{equation}\label{eq:pairwise_approximation}
\GW_2(\XX_k,\XX_l)
\approx
\GW_2(\YY_k,\YY_l),
\qquad k,l = 1,\dotsc,8.
\end{equation}
Since the quantizations
are much smaller than the original spaces,
we expect a significant speed-up 
at the cost of some approximation error.
Before proceeding,
we remark that the latter 
is approximately bounded.
Indeed, 
for arbitrary quantizers $\YY_k$, 
the triangle and reverse triangle
inequalities yield
\begin{align*}
\lvert \GW_2(\XX_k,\XX_l) - \GW_2(\YY_k,\YY_l) \rvert
\! &\leq \! 
\lvert \GW_2(\XX_k,\XX_l) - \GW_2(\XX_k,\YY_l) \rvert
\!+\! \lvert \GW_2(\XX_k,\YY_l) - \GW_2(\YY_k,\YY_l) \rvert 
\\
&\leq \GW_2(\XX_k,\YY_k) + \GW_2(\XX_l,\YY_l).
\end{align*}
The right-hand side is the total quantization cost 
of the two shapes, 
which is
returned as a byproduct of \cref{alg:1}.
It upper-bounds
$\qGW(\XX_k)+\qGW(\XX_l)$ 
and thus provides an a-priori error 
bound\footnote{
In our experiments the numerical analogue to this bound turns out to be conservative:
the observed errors are usually about an order of magnitude below it.
}
for
\cref{eq:pairwise_approximation}.

\begin{figure}[t]
  \centering
  \includegraphics[width=0.6\textwidth]{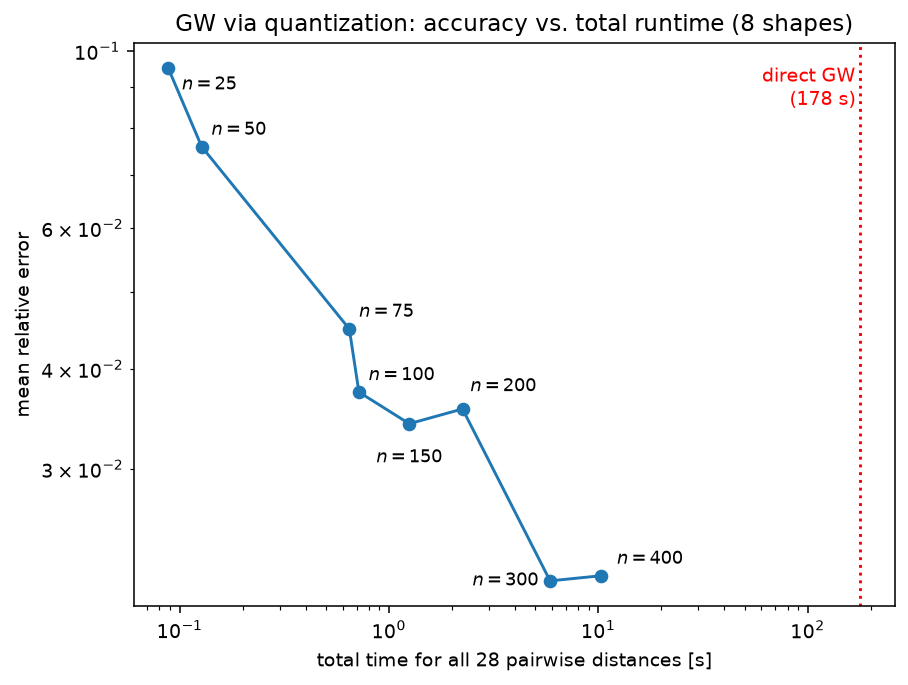}
    \caption{Total runtime (x-axis) versus mean relative error (y-axis)
    for approximating all $28$
    pairwise $\GW_2$ distances by quantization to $n$ points.
    Each marker corresponds to one value of $n$; the dotted red
    line is the direct full-resolution baseline.}
    \label{fig:pairwise_speedup}
\end{figure}

As a baseline, 
we compute the distances
$(\GW_2(\XX_k,\XX_l))_{1 \leq k < l \leq 8}$
directly,
which takes 178 seconds.
We apply the described methodology
for $n = 25,50,75,100,150,200,300,400$.
In \cref{fig:pairwise_speedup},
we plot the total time of computing all 8 quantizations
$\YY_1,\dotsc,\YY_8$
and all their pairwise GW distances 
$(\GW_2(\YY_k,\YY_l))_{1 \leq k < l \leq 8}$
against
the mean relative error 
\[
\frac{1}{28} 
\sum_{1\leq k < l \leq 8}
\frac{
\lvert \GW_2(\YY_k,\YY_l) - \GW_2(\XX_k,\XX_l) \rvert
}{
\GW_2(\XX_k,\XX_l)
}.
\]
As expected, 
the relative error is 
generally decreasing for increasing $n$,
where the slight increase at $n=200$
is likely due to the non-convexity of the GW problem, 
whose solver only attains local optima.
In summary,
already for moderate $n$ the quantization approach reduces the
runtime significantly 
while incurring a relative error of
only around two to ten percent.

\subsection{Verification of Bounds and Rates}\label{subsec:numerics_rates}

In \cref{sec:euclidean}, under suitable conditions
we have shown two-sided bounds of the form
\[
c \cdot \qW(\xi) \leq \qGW(\XX) \leq C \cdot \qW(\xi),
\]
when $\XX = (X,g,\xi) \in \GM$ 
for $X \subset \R^d$,
and $g \in \{\|\cdot - \cdot\|, \|\cdot - \cdot \|^2, \langle \cdot, \cdot \rangle\}$,
for all
$n \in \N$.
The Wasserstein quantization value
$\qW(\xi)$ is always taken 
with respect to the Euclidean distance.
The exact constants $c,C > 0$ 
depend on the choice of gauge $g$
and are summarized in
\cref{tab:sandwich_constants}.

\begin{table}[t]
  \centering
    \caption{Constants in the two-sided bound
  $c\,\qW(\xi)\le\qGW(\XX)\le C\,\qW(\xi)$ for the three Euclidean gauges.
  Here $\lambda_{\min}(\cdot)$ and
  $\lambda_{\max}(\cdot)$ denote the smallest and largest eigenvalue of a matrix,
  $\Sigma_\xi$ is the covariance matrix of $\xi$, and
  $M_\xi$ is its second-moment matrix.}
  \begin{tabular}{c c c}
    \toprule
    $g$ & $c$ & $C$ \\
    \midrule
    $\lVert\cdot-\cdot\rVert$   
    & $\diam(X)^{-1} \cdot \sqrt{\lambda_{\min}(\Sigma_\xi)}$ 
    & $2$ 
    \\
    $\lVert\cdot-\cdot\rVert^2$ 
    & $2\sqrt{\lambda_{\min}(\Sigma_\xi)}$
    & $4\diam(X)$ 
    \\
    $\langle\cdot,\cdot\rangle$ 
    & $\sqrt{\lambda_{\min}(M_\xi)}$ 
    & $\sqrt{2\,\lambda_{\max}(M_\xi)}$
    \\
    \bottomrule
  \end{tabular}
  \label{tab:sandwich_constants}
\end{table}

To showcase the behavior of the bounds in practice,
we sample $N = 5\,000$ points
$x_1,\dotsc,x_N$
from the uniform distribution on 
$[0,1]^d$ for $d = 1,2,3$.
We set
$\XX = (X,g,\xi) \in \GM$,
where
$X \subset [0,1]^d$ is the convex hull of 
$\{x_1,\dotsc,x_N\}$,
$g \in \{\|\cdot - \cdot\|, \|\cdot - \cdot \|^2, \langle \cdot, \cdot \rangle\}$
and
$\xi = \frac{1}{N}\sum_{i=1}^N \delta_{x_i}$.
Note that $\XX$ is covered by the theory of \cref{sec:euclidean} that
concerns the two-sided bounds.
We proceed to run \cref{alg:1}
to approximately compute $\qGW(\XX)$ numerically,
for $n = 8,16,32,64,128,256$,
initialized by taking the Voronoi cells
of a greedy $n$-furthest-point-sampling procedure
on $X$
according to the Euclidean distance, 
cf.\ \Cref{subsec:exp-mesh}.
We also use the classic Lloyd algorithm,
initialized with the same subsampled points,
to approximate $\qW(\xi)$.
In every case, \cref{alg:1} and Lloyd's algorithm
converge in under 200 iterations.
Based on the approximations,
we plot the graph of
$n \mapsto \qGW(\XX)$,
as well as the band
$n \mapsto [c\qW(\xi), C \qW(\xi)]$
for all three gauges 
$g \in \{\|\cdot - \cdot\|, \|\cdot - \cdot \|^2, \langle \cdot, \cdot \rangle\}$
and all choices of dimensions 
$d = 1,2,3$,
as log-log plots in \cref{fig:rates_and_bounds}.

\begin{figure}[t]
  \centering
  \includegraphics[width=0.8\textwidth]{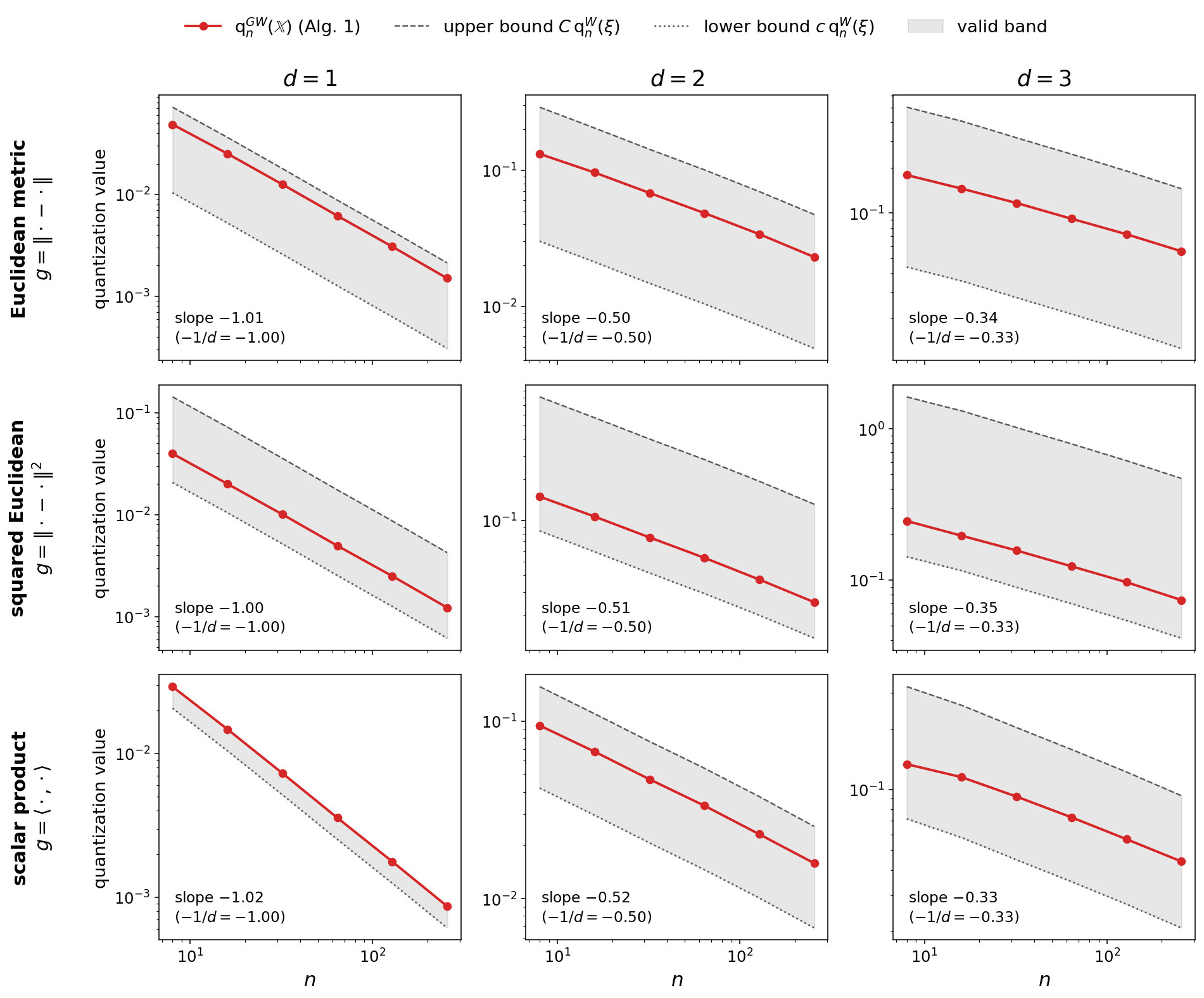}
    \caption{
    Verification of the two-sided bounds and the quantization rate. 
    For
    each dimension $d\in\{1,2,3\}$ (columns) 
    we sample $N=5\,000$ points uniformly
    from $[0,1]^d$ 
    and form $\XX=(X,g,\xi)$ with sampled points $X$, 
    uniform measure $\xi$, 
    and one of the three Euclidean gauges $g$ (rows). 
    Each panel shows
    $\qGW(\XX)$ (approximated via \cref{alg:1}) 
    in red
    against $n$ on log--log axes. 
    It
    lies in the grey band $[\,c\,\qW(\xi),\,C\,\qW(\xi)\,]$
    (constants from \cref{tab:sandwich_constants}, 
    $\qW(\xi)$ approximated via Lloyd's algorithm),
    and its fitted slope approximately matches
    the predicted rate $-1/d$. 
    Here $\qW(\xi)$ is the Wasserstein
    quantization value with respect to the Euclidean metric.}
  \label{fig:rates_and_bounds}
\end{figure}

In every case,
the numerical approximation of the quantization value $\qGW(\XX)$ 
lies within the numerical approximations of the interval $[c\,\qW(\xi),C\,\qW(\xi)]$.
Note that for the scalar product with $d=1$,
the values lie essentially at the upper bound.
This reflects \cref{thm:closed_form_1d},
which provides 
the closed-form expression
$\qGW(\XX) = \qW(\xi)\sqrt{2M_\xi - \qW(\xi)^2}$.
As $d=1$, $M_\xi$ reduces to a scalar
and moreover
 $M_\xi = \lambda_{\max}(M_\xi)$.
Since $\qW(\xi)\to 0$ as $n\to\infty$, 
the ratio $\qGW(\XX)/\qW(\xi)$
converges to $\sqrt{2\lambda_{\max}(M_\xi)} = C$, 
so the upper bound is
asymptotically attained. 

We use a large sample size 
$N = 5\,000 \gg n$
so that $\qGW(\XX)$ serves as a proxy 
for the quantization value of the continuous
uniform measure on $[0,1]^d$. 
Accordingly, the log--log plots display an approximately linear
decay, and the fitted slopes match the predicted rate $-1/d$ 
to within $\approx 0.02$
across all gauges and dimensions.
This reflects the asymptotic rate $\qGW(\XX) \asymp n^{-1/d}$
established in \cref{cor:GW_quant_rate_1,cor:GW_quant_rate_2}.

\subsection{Pruning a Neural Network}\label{subsec:pruningnn}
Pruning neural networks is a way of reducing the size (and thus the computational burden) of a large neural network while trying to retain a high performance. In this example, we show how a feedforward neural network can be encoded as a gm-space and how the respective quantization leads to a smaller network which performs better than baseline pruning procedures. To this end, we work with a simple three-hidden-layer neural network of the form
\[
f(x) = W_4 \circ \varphi \circ W_3 \circ \varphi \circ W_2 \circ \varphi \circ W_1(x),
\]
where $x \in \mathbb{R}^d$ is the input to the neural network, $W_1, W_2, W_3, W_4$ are linear maps (i.e., matrices), and $\varphi$ is an activation function applied element-wise. 

Most pruning methods work layer by layer. A simple pruning baseline is to delete neurons based on certain characteristics, usually related to the magnitudes of the attached weights (see, e.g.~\cite{ding2025neural,hoefler2021sparsity, li2017pruning}) or activation vectors (see, e.g., \cite{sun2024simple}). Beyond just deletion of neurons, pruning can be based on clustering each layer's neurons
and suitably combining the weights of neurons which share a cluster (see, e.g., \cite{akash2022wasserstein,morelli2026partial, singh2020model}). While pruning methods are sometimes combined with retraining, in this example we focus purely on data-free methods, and hence neither retraining nor activation vectors are used.

In the following, we show how to encode and prune a neural network as a whole using the presented GW machinery. 
To this end, we define the gm-space $\XX =([N], g, \xi)$, where $N$ is the total number of neurons (including input and output neurons) in the network and $\xi$ is the uniform distribution. 
In the simple experiment (where the MNIST dataset is used which consists of images of handwritten digits), we have $N = 784+3\cdot h + 10$, where $784$ is the input dimension ($28\times 28$ pixel images), $h \in \{50, 100\}$ is the hidden dimension for our experiments, and $10$ is the output dimension, corresponding to 10 digits.
The gauge $g$ is built via the weight matrices of the network as follows: For $l, \tilde{l} \in \{0, 1, 2, 3, 4\}$ with $l<\tilde{l}$, set $W^{(l, \tilde{l})} \coloneqq W_{\tilde{l}}  W_{\tilde{l}-1} \dots W_{l+1}$ and $W^{(\tilde{l}, l)} \coloneqq (W^{(l, \tilde{l})})^{\top}$. Then, if $i$ is the $a$-th neuron in layer $l$ and $j$ is the $b$-th neuron in layer $\tilde{l}$, where $l < \tilde{l}$, we set $g(i, j) := \lambda^{(|l-\tilde{l}|-1)/2} W^{(l, \tilde{l})}_{b, a}$, where $\lambda \in (0, 1)$ is a factor which we introduce to increase the relative importance of short-range interactions in the network. In the following, we use $\lambda=0.3$, but we observed similar results for $\lambda \in [0.1,0.5]$.
If $l=\tilde{l}$ and $i\neq j$, we set $g(i, j)=0$. Regarding the diagonal $g(i, i)$, we make the following observations: First, in a clustering sense, each input and each output neuron is unique, and hence we set their gauge value $g(i, i)$ each to a unique value separated sufficiently from all other diagonal values. This ensures that no input or output neuron should ever be clustered together with any other neuron. Similarly, so that no neurons are matched across different hidden layers, each hidden layer has a unique value for their corresponding neurons on the diagonal, each sufficiently separated from all other values. Effectively, what these diagonal values as a whole are doing is to restrict the forms of the support which the measures $\pi_i$ in \cref{alg:1} can take. Finally, we symmetrize the gauge by setting $g(j, i) := g(i, j)$.

To be precise, for neurons $i\in [N]$, we set $l(i)\in \{0, 1, 2, 3, 4\}$ as its layer, and we set $p(i) \in \{1, 2, 3, \ldots\}$ as the neuron's position within the layer. 
For $i, j \in [N]$, the gauge is given by
\[
g(i, j) \coloneqq \begin{cases}
    v_i & i=j \text{ and } l(i) \in\{0, 4\} \\
    y_{l(i)} & i=j \text{ and } l(i) \in \{1, 2, 3\}\\
    0 & l(i) = l(j) \text{ and } i \neq j \\
    \lambda^{(|l(i)-l(j)|-1)/2} \, W^{(l(i), l(j))}_{p(j), p(i)} & l(i) \neq l(j)
\end{cases}
\]
where $v_i, y_{l(i)} \in \mathbb{R}$ are just sufficiently different values on the diagonal.

\begin{figure}[t]
    \centering
    \includegraphics[width=1\linewidth]{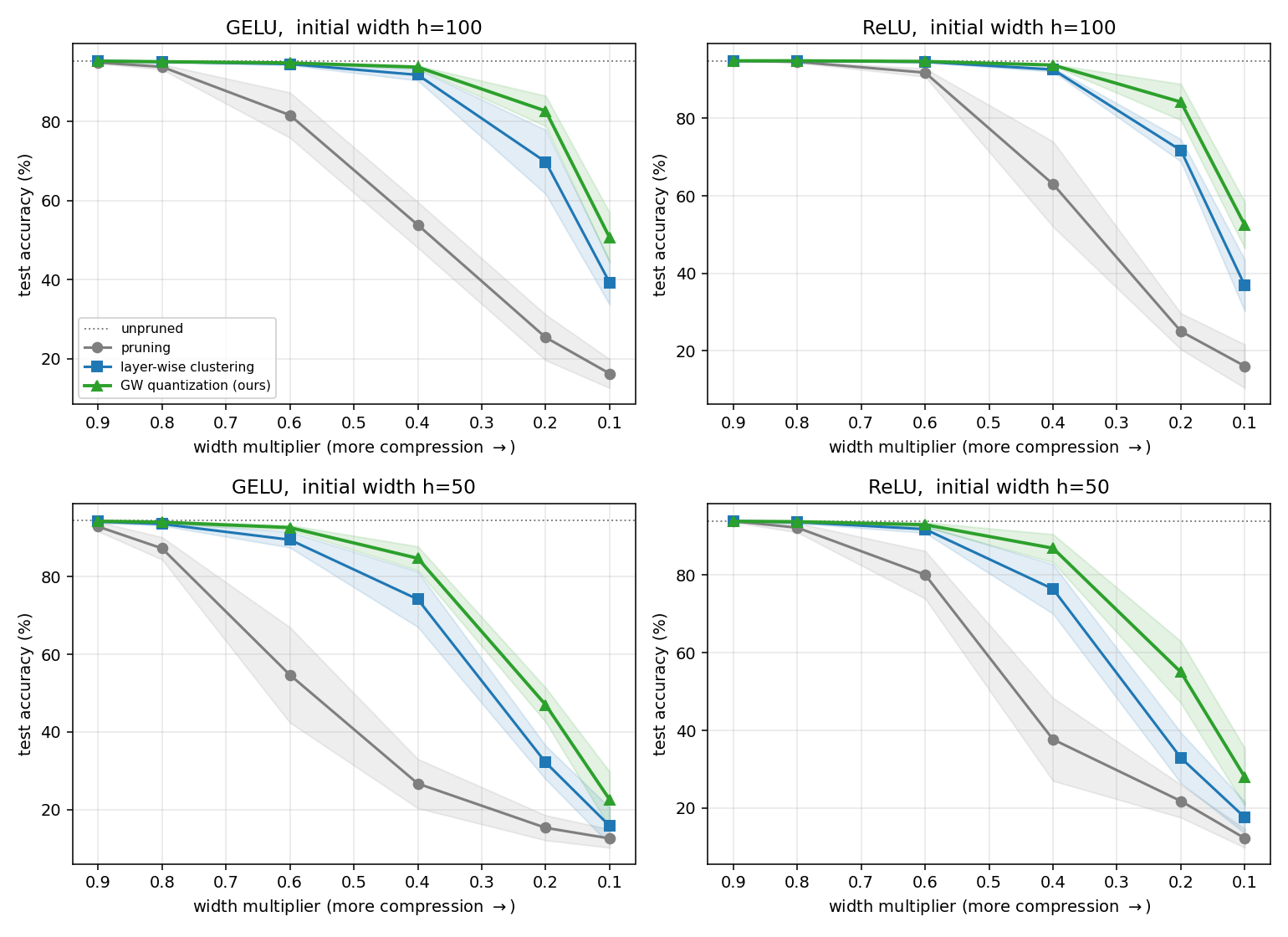}
    \caption{Comparison of different ways to reduce the size of the neural network (three-hidden-layer feedforward neural network trained on MNIST) by reducing each hidden dimension to a width multiplier times the initial width. The baselines are pruning (gray lines, arising from deleting those neurons per layer whose incoming weights have lowest $L^1$ norm) and layer-wise clustering of neurons (blue lines, cf.~\cite{morelli2026partial}). The method introduced in this paper is shown as the green lines, where the network as a whole is modeled as a gm-space and quantized with \cref{alg:1}. The confidence bands are over 10 independent random seeds.}
    \label{fig:pruningnetwork}
\end{figure}

We run \cref{alg:1} to quantize $\XX$. To this end, we note that we merely reduce the size of the hidden layers, so we prune the model size to $n = 784 + 3 \cdot \lfloor\alpha \cdot h\rfloor + 10$, where $\alpha \in [0, 1]$ is a multiplier for the width of hidden layers. To initialize \cref{alg:1}, we use the layer-by-layer hierarchical clustering of neurons as given by \cite{morelli2026partial}.
We then run \cref{alg:1} with step size 1, but note that the results are practically the same when using a line search instead. 
To ensure a fair comparison with the baseline methods, we want to avoid empty clusters in \cref{alg:1}, for which we use standard cluster splitting methods, cf.~\cite[Section 8.2.2]{TSK2006}.
We denote the output of \cref{alg:1} by $([n], G, \upsilon)$. Due to the aforementioned modeling choice of the gauge (effectively restricting $\pi_i$), we can still assign each neuron $i \in [n]$ uniquely to a layer. To extract the pruned neural network from this output, we first fit a gauge of neural-network-form (as $g$, but for a smaller network) to $G$ in $L^2(\upsilon \otimes \upsilon)$, which we denote by $\tilde{G}$.\footnote{We perform this $L^2$-minimization via the Adam optimizer, which appears to converge quickly and robustly.} Finally, for two neurons $i, j \in [n]$ in layers $l$ and $l+1$, the corresponding entry of the weight matrix $\tilde{W}_{l+1}$ of the pruned network is given by $\frac{\upsilon_i}{\xi(\{1\})}\tilde{G}(i, j)$.\footnote{The inclusion of the term $\frac{\upsilon_i}{\xi(\{1\})}$ is important to retain the scale of the weights in each layer. It encodes how many initial neurons are summarized by the new node $i$, which is not included in the gauge value (which is simply an average across initial gauge values). Of course, we could replace $\xi(\{1\})$ by $1/N$.}

The results for networks with varying activation functions and sizes are shown in \cref{fig:pruningnetwork}, which shows results for networks trained on MNIST. We see that the introduced method using GW quantization performs best throughout all four regimes (different sizes and activation functions), particularly leading to better performance at more aggressive compression.
While this example is certainly only a first step in this direction (e.g., no retraining is used and the modeling is agnostic to the activation function), it showcases the potential of the modeling freedom provided by GW to encode neural networks and the relevance of quantization in such settings.

\appendix
\crefalias{section}{appendix}
\crefalias{subsection}{appendix}

\section{Auxiliary Lemmas}

We provide two auxiliary lemmas which are known, 
but stated here for clear reference in our setting.

\begin{lemma}\label{lem:linear_assignment}
    Let $\xi \in \p(X)$ 
    and $(c_i)_{i=1}^n$
    with non-negative 
    $c_1, \dotsc, c_n \in L^1(\xi)$. 
    Then
    \[
    \min_{
    \substack{
    \pi_1,\dotsc,\pi_n \in \M_+(X)
    \\
    \sum_{i=1}^n \pi_i = \xi
    }
    }
    \sum_{i=1}^n \int_{X} c_i \dx \pi_i(x)
    =
    \int_X \min_{j=1,\dotsc,n} c_j \dx \xi.
    \]
    A collection of measures 
    $(\pi_i)_{i=1}^n \subset \M_+(X)$ with $\sum_{i=1}^n \pi_i = \xi$
    attains the minimum 
    if and only if
    $\pi_i$ is concentrated on
    \[
    A_i \coloneqq \{x \in X : c_i(x) = \min_{j=1,\dotsc,n} c_j(x)\},
    \quad 
    \text{for all } i = 1,\dotsc,n.
    \]
    In particular,
    if $(V_i)_{i=1}^n$ 
    is a Voronoi partition of $X$
    according to $(c_i)_{i=1}^n$,
    $(\xi \vert_{V_i})_{i=1}^n$ is a solution
    and it is unique if and only if 
    $\xi(A_i \cap A_j) = 0$,
    $i \neq j$.
\end{lemma}

\begin{proof}
    We abbreviate $c \coloneqq \min_{j=1,\dotsc,n} c_j$.
    Since $c_i \in L^1(\xi)$ 
    and $0 \leq c \leq c_i$,
    we also have $c \in L^1(\xi)$.
    Now let $(\pi_i)_{i=1}^n \subset \M_+(X)$ 
    with
    $\sum_{i=1}^n \pi_i = \xi$ be arbitrary. 
    Then
    \begin{equation}\label{eq:linear_lb}
    \sum_{i=1}^n \int_X c_i \,\mathrm{d}\pi_i
    \geq \sum_{i=1}^n \int_X c \,\mathrm{d}\pi_i
    = \int_X c \,\mathrm{d}\xi.
    \end{equation}
    It holds $c_i - c \geq 0$,
    and $X \setminus A_i = \{x : c_i(x) - c(x) > 0\}$.
    Thus, equality in \cref{eq:linear_lb} 
    holds if and only if
    $\pi_i(X \setminus A_i) = 0$,
    i.e.\ $\pi_i$ is concentrated on $A_i$,
    for every $i = 1,\dotsc,n$.
    By construction,
    the Voronoi partition
    $(V_i)_{i=1}^n$ is a disjoint
    partition of $X$ with $V_i \subset A_i$, 
    so in particular $(\xi|_{V_i})_{i=1}^n$ 
    is feasible and attains
    the lower bound in \cref{eq:linear_lb}.
    
    We turn to the uniqueness claim.
    First, 
    assume $\xi(A_i \cap A_j) = 0$ for all $i \neq j$
    and let $(\pi_i)_{i=1}^n$ be any solution.
    It holds
    \begin{equation}\label{eq:proof_xi_B_V_i}
    \xi(B \cap V_i) 
    = \sum_{j=1}^n \pi_j(B \cap V_i) 
    = \pi_i(B \cap V_i) + \sum_{j \neq i} \pi_j(B \cap V_i)
    , \quad B \subset X \text{ measurable}.
    \end{equation}
    For $j \neq i$, we have
    $V_i \cap A_j \subset A_i \cap A_j$, 
    so that 
    \[
    \pi_j(V_i) = \pi_j(V_i \cap A_j) \leq \pi_j(A_i \cap A_j) \leq \xi(A_i \cap A_j)  = 0.
    \]
    Hence, together with \cref{eq:proof_xi_B_V_i}, 
    we obtain $\pi_i\vert_{V_i} = \xi\vert_{V_i}$.
    In the following, 
    we establish that $\pi_i$ is concentrated on $V_i$.
    If $x \in A_i \setminus V_i$,
    then $x \in V_j \subset A_j$
    for some $j \neq i$ 
    and thus $x \in A_i \cap A_j$.
    Hence, 
    \[
    A_i \setminus V_i \subset \bigcup_{j \neq i} (A_i \cap A_j),
    \]
    which is a $\xi$-null set.
    As $\pi_i \leq \xi$
    and $\pi_i$
    is concentrated on $A_i = V_i \cup (A_i \setminus V_i)$,
    it follows that $\pi_i$ is
    concentrated on $V_i$
    and $\pi_i = \xi\vert_{V_i}$,
    for all $i=1,\dotsc,n$.

    Conversely, 
    assume $\xi(A_k \cap A_l) > 0$
    for some fixed $k \neq l$. 
    Since $(V_i)_{i=1}^n$ is a
    partition of $X$, 
    there exists 
    $i \in \{1,\dotsc,n\}$ 
    such that 
    $M \coloneqq A_k \cap A_l \cap V_i$
    satisfies $\xi(M) > 0$. 
    Choose $j \in \{k, l\} \setminus \{i\}$ 
    and set
    \[
    \pi'_i \coloneqq \xi|_{V_i} - \tfrac{1}{2}\,\xi|_{M}, \qquad
    \pi'_j \coloneqq \xi|_{V_j} + \tfrac{1}{2}\,\xi|_{M}, \qquad
    \pi'_m \coloneqq \xi|_{V_m}, \quad m \in \{1,\dotsc,n\} \setminus \{i, j\}.
    \]
    Since $M \subset V_i$,
    we obtain
    $(\pi'_m)_{m=1}^n \subset \mathcal{M}_+(X)$ 
    with $\sum_{m=1}^n \pi'_m = \xi$.
    Each $\pi'_m$ is concentrated on $A_m$, 
    since $V_m \subset A_m$ 
    and $M \subset A_j$ 
    as well as $M \subset V_i \subset A_i$.
    By the first part, $(\pi'_m)_{m=1}^n$ is a solution, 
    and it differs from
    $(\xi|_{V_m})_{m=1}^n$ as $\xi(M) > 0$, as desired.
\end{proof}

The following is a well-known statement regarding the arithmetic mean as the $L^2$ barycenter.

\begin{lemma}\label{lem:L2_bary}
    Let $A$ be a Polish space, 
    $\alpha \in \M_+(A)$ 
    with $\alpha(A) > 0$
    and $f \in L^2(A, \alpha; \R^d)$. 
    Then, 
    for its associated $L^2$-mean 
    $m \coloneqq \frac{1}{\alpha(A)} \int_A f \dx\alpha \in \R^d$, 
    we have
    \begin{equation}\label{eq:L2_bary_min}
        \lVert f - m \rVert_{L^2(\alpha)} 
        = \inf_{c \in \R^d} \lVert f - c \rVert_{L^2(\alpha)},
    \end{equation}
    and $m$ is the unique minimizer.
    Moreover, it holds
    \[
    \lVert f - a_0 \rVert_{L^2(\alpha)}^2 
    = \lVert f - m \rVert_{L^2(\alpha)}^2 
    + \lVert m - a_0 \rVert^2\, \alpha(A), 
    \qquad \text{for any } a_0 \in \R^d.
    \]
\end{lemma}

\begin{proof}
    For any $a_0 \in \R^d$, 
    expanding the square gives
    \begin{align*}
    \lVert f - a_0 \rVert_{L^2(A, \alpha)}^2
    &= \int_A \lVert f(a) - a_0 \rVert^2 \dx\alpha(a)
    \\
    &= \int_A 
    \Bigl( 
    \lVert f(a) - m \rVert^2 
    + 2 \langle f(a) - m,\, m - a_0 \rangle 
    + \lVert m - a_0 \rVert^2 
    \Bigr) 
    \dx\alpha(a)
    \\
    &= \lVert f - m \rVert_{L^2(A, \alpha)}^2 
    + 2 \Bigl\langle \int_A (f(a) - m) \dx\alpha(a),\, m - a_0 \Bigr\rangle 
    + \lVert m - a_0 \rVert^2\, \alpha(A).
    \end{align*}
    By definition of $m$, 
    it holds 
    $\int_A (f(a) - m) \dx\alpha(a) = \int_A f \dx\alpha - m\, \alpha(A) = 0$, 
    so the middle term vanishes 
    and the asserted identity follows. 
    It holds 
    $\lVert m - a_0 \rVert^2\, \alpha(A) \ge 0$ 
    and equality 
    if and only if $a_0 = m$, 
    this also yields \cref{eq:L2_bary_min} 
    and the uniqueness of $m$.
\end{proof}

\section{Proofs for \texorpdfstring{\cref{sec:gw}}{Section 3}}

\subsection{Fully Supported Representatives in \texorpdfstring{$\GM_n$}{GMₙ}}\label{subapp:proof_of_3_1}

\begin{proof}[Proof of \cref{prop:rep-with-larger-cardinality}]
    Let $\YY = ([n],G,\upsilon) \in \GM_n$ and set
    $Y_1 = \{i \in [n] : \upsilon_i > 0\}$.
    We assume $m \coloneqq \lvert Y_1 \rvert < n$, 
    otherwise we are done.
    Firstly, $\YY_1 \coloneqq (Y_1, G \vert_{Y_1 \times Y_1}, \upsilon \vert_{Y_1})$
    is homomorphic to $\YY$ by construction.
    Moreover, 
    by renaming $Y_1$ according to any bijection 
    $T_1:Y_1 \to [m]$,
    $\YY_1$ 
    is further homomorphic to some 
    $\YY_2 = ([m],G^{(2)},\upsilon^{(2)})$
    with $\upsilon^{(2)} \in \p([m])$ fully supported.
    It suffices to find a to-$\YY_2$ 
    homomorphic gm-space
    of the form
    $\tilde{\YY} = ([m+1],\tilde{G},\tilde{\upsilon})$,
    with fully supported $\tilde{\upsilon} \in \p([m+1])$.
    From there the result follows by induction.
    For this let
    $b = (G_{1,m}^{(2)},\dotsc,G_{m,m}^{(2)})^\top \in \mathbb{R}^m$ 
    and set
    \[
    \tilde{G} = 
\begin{pmatrix}
    G^{(2)} & b 
    \\
    b^\top & G_{m,m}^{(2)}
\end{pmatrix}
    \in \R^{(m+1) \times (m+1)}
    \text{ and }
    \tilde{\upsilon} = \biggl(
    \upsilon_1^{(2)},\dotsc,
    \upsilon_{m-1}^{(2)}, 
    \frac{\upsilon_{m}^{(2)}}{2},
    \frac{\upsilon_{m}^{(2)}}{2}
    \biggr)^\top.
    \]
    By construction, the map 
    $T:[m+1] \to [m]$ 
    given by $T\vert_{[m]} = \id_{[m]}$, and $T(m+1) = m$,
    defines a homomorphism from $\tilde{\YY}$
    to $\YY_2$.
    Hence,
    $\GW_2(\YY_2,\tilde{\YY}) = 0$,
    so that $\YY$ 
    is homomorphic 
    to $\tilde{\YY}$.
    Repeating the construction inductively
    for $m+2,\dotsc,n$,
    we see that $\YY$ is homomorphic to
    a space defined on $[n]$ with fully supported measure.
\end{proof}

\subsection{Optimality of the Block-Mean Gauge}\label{subapp:proof_of_3_2}

\begin{proof}[Proof of \cref{thm:pointwise_lb_gauge}]
    Consider a fixed $\pi \in \p(X \times [n])$
    with $(P_X)_\# \pi = \xi$.
    Clearly,
    $\hat{G}$ is symmetric.
    Furthermore, 
    it holds
    \begin{align}\label{proof:pointwise_lb_1}
        \inf_{G \in \R^{n \times n}}
        Q_{\GW}(G,(\pi_i)_{i=1}^n)
        &=
        \inf_{G \in \R^{n \times n}}
        \sum_{i,i'=1}^n
        \iint_{X^2}
        \lvert g(x,x') - G_{i,i'} \rvert^2
        \dx \pi_i(x) \dx \pi_{i'}(x')
        \nonumber
        \\
        &=
        \sum_{i,i'=1}^n
        \inf_{G_{i,i'} \in \R}
        \|g - G_{i,i'}\|_{L^2(\pi_i \otimes \pi_{i'})}^2
        \nonumber
        \\
        &=
        \sum_{i,i'=1}^n
        \|g - \hat{G}_{i,i'}\|_{L^2(\pi_i \otimes \pi_{i'})}^2,
    \end{align}
    where the last equality follows by 
    \cref{lem:L2_bary}, 
    since $\hat{G}_{i,i'}$ is exactly the $L^2$ mean
    of $g$ with respect to 
    $\pi_i \otimes \pi_{i'}$, 
    $i,i' \in [n]$.
    Since the unconstrained minimizer is symmetric,
    it also solves the problem when constrained to symmetric matrices.
    Consider the case 
    $\pi \in \Pio(\XX,\YY)$
    for the given $\YY = ([n],G,\upsilon)$.
    Then
    \begin{align*}
        \GW_2^2(\XX,\YY)
        &= 
        Q_{\GW}(G,(\pi_i)_{i=1}^n)
        =
        \sum_{i,i'=1}^n
        \|g - G_{i,i'}\|_{L^2(\pi_i \otimes \pi_{i'})}^2
        \\
        &=
        \sum_{i,i'=1}^n
        \biggl(
        \|g - \hat{G}_{i,i'}\|_{L^2(\pi_i \otimes \pi_{i'})}^2
        + 
        \underbrace{\pi_i(X) \pi_{i'}(X)}_{= \upsilon_i \upsilon_{i'}} (G_{i,i'} - \hat{G}_{i,i'})^2
        \biggr)
        \nonumber
        \\
        &=
        \iint_{(X \times [n])^2}
        \lvert g(x,x') - \hat{G}_{i,i'} \rvert^2
        \dx \pi(x,i) \dx \pi(x',i')
        + \|G - \hat{G}\|_{L^2(\upsilon \otimes \upsilon)}^2
        \nonumber
        \\
        &\geq
        \GW_2^2(\XX,\hat{\YY})
        + \|G - \hat{G}\|_{L^2(\upsilon \otimes \upsilon)}^2
    \end{align*}
    where
    the equality in the second line 
    follows again from
    \cref{lem:L2_bary}
    and the final estimate is due to 
    $\pi \in \Pi(\xi,\upsilon)$.
\end{proof}

\subsection{Optimal Couplings Concentrate on Voronoi Cells}\label{subapp:proof_of_3_3}

First,
we show that
$Q_{\GW}(G,\cdot)$ 
attains its
minimum over 
$\Pi(\xi, *_n)$,
for which we require an
intermediary result
which we present in a slightly generalized form
that is required later.

\begin{lemma}\label{lem:product_convergence}
    Let $X$ be a Polish space,
    $\xi \in \p(X)$ and consider two sequences
    $(\pi^{(k)})_{k \in \N}$, $(\sigma^{(k)})_{k \in \N} \subset \p(X \times [n])$
    with
    $(P_X)_\# \pi^{(k)} = (P_X)_\# \sigma^{(k)} = \xi$, $k \in \N$,
    such that $\pi^{(k)} \weakly \pi$ and $\sigma^{(k)} \weakly \sigma$.
    Then,
    for every $h \in L^1(\xi \otimes \xi)$ and all $i, i' \in [n]$,
    \[
    \iint_{X^2} h(x,x') \dx\pi_i^{(k)}(x) \dx\sigma_{i'}^{(k)}(x')
    \xrightarrow{k \to \infty}
    \iint_{X^2} h(x,x') \dx\pi_i(x) \dx\sigma_{i'}(x').
    \]
\end{lemma}

\begin{proof}
    Since $\pi_i^{(k)}(A) = \pi^{(k)}(A \times \{i\}) \leq \xi(A)$ and, likewise,
    $\sigma_{i'}^{(k)}(A) \leq \xi(A)$ for all measurable $A \subset X$,
    it holds $\pi_i^{(k)} \otimes \sigma_{i'}^{(k)} \leq \xi \otimes \xi$.
    Thus
    \begin{equation}\label{eq:proof_uniform_bound}
    \Big\lvert \iint_{X^2} f \dx \pi_i^{(k)} \dx \sigma_{i'}^{(k)} \Big\rvert
    \leq \lVert f \rVert_{L^1(\xi \otimes \xi)}
    \quad \text{for all } f \in L^1(\xi \otimes \xi),\ k \in \N.
    \end{equation}
    As the marginal projection is weakly continuous,
    the limits satisfy $(P_X)_\# \pi = (P_X)_\# \sigma = \xi$,
    so $\pi_i \otimes \sigma_{i'} \leq \xi \otimes \xi$
    and
    \cref{eq:proof_uniform_bound} also holds with
    $\pi_i, \sigma_{i'}$ in place of
    $\pi_i^{(k)}, \sigma_{i'}^{(k)}$.
    
    First,
    let $\tilde{h} \in C_b(X \times X)$, the function
    $((x,j),(x',j')) \mapsto \tilde{h}(x,x') \1_{\{i\}}(j) \1_{\{i'\}}(j')$
    is bounded and continuous on $(X \times [n])^2$.
    Hence, 
    due to
    $\pi^{(k)} \otimes \sigma^{(k)} \weakly \pi \otimes \sigma$
    see \cite[Prop.~2.7.8]{bogachev_weak},
    we obtain
    \begin{equation}\label{eq:proof_limit_h_bounded}
    \iint_{X^2} \tilde{h} \dx\pi_i^{(k)} \dx\sigma_{i'}^{(k)}
    \xrightarrow{k \to \infty}
    \iint_{X^2} \tilde{h} \dx\pi_i \dx\sigma_{i'}.
    \end{equation}

    Now, 
    let $h \in L^1(\xi \otimes \xi)$ and $\varepsilon > 0$. 
    As
    $C_b(X \times X)$ is dense in $L^1(\xi \otimes \xi)$, fix
    $\tilde{h} \in C_b(X \times X)$ with
    $\lVert h - \tilde{h} \rVert_{L^1(\xi \otimes \xi)} < \varepsilon$.
    Applying
    \cref{eq:proof_uniform_bound} 
    to $f = h - \tilde{h}$, for both $\pi^{(k)}$ and $\pi$,
    \begin{align*}
    \Big\lvert \iint_{X^2} h \dx\pi_i^{(k)} \dx\sigma_{i'}^{(k)}
    - \iint_{X^2} h \dx\pi_i \dx\sigma_{i'} \Big\rvert
    &\leq \Big\lvert \iint_{X^2} (h - \tilde{h}) \dx\pi_i^{(k)} \dx\sigma_{i'}^{(k)} \Big\rvert
    + \Big\lvert \iint_{X^2} (\tilde{h} - h) \dx\pi_i \dx\sigma_{i'} \Big\rvert
    \\
    &\quad
    + \Big\lvert \iint_{X^2} \tilde{h} \dx\pi_i^{(k)} \dx\sigma_{i'}^{(k)}
    - \iint_{X^2} \tilde{h} \dx\pi_i \dx\sigma_{i'} \Big\rvert
    \\
    &\leq 2\varepsilon
    + \Big\lvert \iint_{X^2} \tilde{h} \dx\pi_i^{(k)} \dx\sigma_{i'}^{(k)}
    - \iint_{X^2} \tilde{h} \dx\pi_i \dx\sigma_{i'} \Big\rvert .
    \end{align*}
    Taking $\limsup_{k \to \infty}$ 
    and using \cref{eq:proof_limit_h_bounded}
    yields
    \[
    \limsup_{k \to \infty}
    \Big\lvert \iint_{X^2} h \dx\pi_i^{(k)} \dx\sigma_{i'}^{(k)}
    - \iint_{X^2} h \dx\pi_i \dx\sigma_{i'} \Big\rvert
    \leq 2\varepsilon .
    \]
    Since $\varepsilon > 0$ was arbitrary, 
    the left-hand side is zero,
    as desired.
\end{proof}

\begin{proposition}\label{prop:Q_GW_min_exists}
Let $\XX=(X,g,\xi)\in\GM$ and $G\in\G_{\sym}^{(n)}$. 
Then $Q_{\GW}(G,\cdot)$ attains its
minimum over $\Pi(\xi, *_n)$.
\end{proposition}

\begin{proof}
The set $\Pi(\xi, *_n)$ is weakly compact \cite[Cor.~5.21]{villani2008optimal}. 
Writing
\[
Q_{\GW}(G,(\pi_i)_{i=1}^n)=\sum_{i,i'=1}^n\iint_{X^2}(g(x,x')-G_{i,i'})^2\dx\pi_i(x)\dx\pi_{i'}(x')
\]
with 
$(g-G_{i,i'})^2\in L^1(\xi\otimes\xi)$, 
\cref{lem:product_convergence} shows that
$Q_{\GW}(G,\cdot)$ 
is weakly continuous on $\Pi(\xi, *_n)$. 
The claim follows by the
Weierstrass theorem.
\end{proof}

Now that we established existence
of minimizers,
we want to characterize them.
Unfortunately,
the objective
$(\pi_i)_{i=1}^n \mapsto Q_{\GW}(G,(\pi_i)_{i=1}^n)$
is quadratic
and its minimization
does not admit closed form
as in the linear Wasserstein setting.
Following the GW literature \cite{PCS2016,SejViaPey21,vayer2020fused},
a common approach to handle such problems,
is via their bilinear relaxation.
More precisely,
we define
\[
\tilde{Q}_{\GW}
(G,(\pi_i)_{i=1}^n,(\tilde{\pi}_{i})_{i=1}^n)
\coloneqq 
\sum_{i,i'=1}^n \iint_{X^2} (g(x,x') - G_{i,i'})^2 \dx \pi_i(x) \dx \tilde{\pi}_{i'}(x'),
\]
which is linear in
$(\pi_{i})_{i=1}^n$
and 
$(\tilde{\pi}_{i})_{i=1}^n$,
respectively.
For some fixed $\tilde{\pi} \in \p(X \times [n])$
with $(P_X)_\# \tilde{\pi} = \xi$,
we define
\[
c_i^{(\tilde{\pi})}(x) \coloneqq \sum_{i'=1}^n \int_{X} (g(x,x') - G_{i,i'})^2 \dx \tilde{\pi}_{i'}(x'),
\quad 
i \in [n].
\]
Since 
$\tilde{\pi}_i \leq \xi$ 
and $g \in L^2(\xi \otimes \xi)$, 
it holds $c^{(\tilde{\pi})}_i \in L^1(\xi)$.
Using this new notation,
we have
\begin{align}\label{eq:Q_tilde_rewrite}
\tilde{Q}_{\GW}(G,(\pi_i)_{i=1}^n,(\tilde{\pi}_i)_{i=1}^n)
=
\sum_{i=1}^n \int_X c_i^{(\tilde{\pi})}(x) \dx \pi_i(x).
\end{align}

In our setting, 
the rewriting \cref{eq:Q_tilde_rewrite} 
casts the linear minimization
as a linear assignment problem, 
whose solutions the following proposition
characterizes via Voronoi partitions.

\begin{proposition}\label{prop:tilde_Q_GW_biconvex_minimization}
    Consider a fixed gauge matrix $G \in \G_{\sym}^{(n)}$
    and $(\tilde{\pi}_i)_{i=1}^n \subset \M_+(X)$
    with $\sum_{i=1}^n \tilde{\pi}_i = \xi$.
    Then, $(\pi_i)_{i=1}^n$ 
    minimizes 
    \[
    \inf_{
    \substack{
    (\pi_i)_{i=1}^n \subset \M_+(X)
    \\
    \sum_{i=1}^n \pi_i = \xi
    }
    }
    \tilde{Q}_{\GW}(G,(\pi_i)_{i=1}^n,(\tilde{\pi}_i)_{i=1}^n)
    \]
    if and only if $\pi_i$ is concentrated on
    \[
    A_i \coloneqq \{x \in X : c_i^{(\tilde{\pi})}(x) = \min_{j} c_j^{(\tilde{\pi})}(x)\}, \quad \text{for every }i \in [n].
    \]
    In particular, 
    $(\xi \vert_{V_i})_{i=1}^n$ 
    is a solution for a Voronoi partition 
    $(V_i)_{i=1}^n$ of $X$ 
    according to $(c_i^{(\tilde\pi)})_{i=1}^n$.
    It is the unique solution if 
    and only if
    $\xi(A_i \cap A_j) = 0$ for all $i \neq j$.
\end{proposition}

\begin{proof}
    This directly follows from \cref{eq:Q_tilde_rewrite}
    and \cref{lem:linear_assignment}.
\end{proof}

Going via the bilinear relaxation allows us to
weakly characterize the solutions 
of the original quadratic problem.
The idea is to show that $\hat{\pi}$,
as a minimizer of the quadratic problem,
also minimizes the linear function
$\tilde{Q}_{\GW}(G,\cdot,(\hat{\pi}_i)_{i=1}^n)$
from where the result follows by 
\cref{prop:tilde_Q_GW_biconvex_minimization}.

\begin{proof}[Proof of \cref{thm:Q_GW_measure_min_characterization}]
    Note that existence is covered by
    \cref{prop:Q_GW_min_exists}.
    For brevity, we use the notation
    $B(\pi,\tilde{\pi}) \coloneqq \tilde{Q}_{\GW}(G,(\pi_i)_{i=1}^n, (\tilde{\pi}_i)_{i=1}^n)$.
    Let $(V_i)_{i=1}^n$ be the partition from 
    \cref{prop:tilde_Q_GW_biconvex_minimization}
    and let 
    $\gamma \in \p(X \times [n])$ 
    be defined by 
    $\gamma(\cdot \times \{i\}) = \xi \vert_{V_i}$.
    According to the proposition,
    $\gamma$ is a solution to 
    \begin{equation}\label{eq:proof_sub_min_problem}
    \argmin_{
    \substack{
    (\pi_i)_{i=1}^n \subset \M_+(X) \\ 
    \sum_{i=1}^n \pi_i = \xi
    }
    }
    B(\pi,\hat{\pi}).
    \end{equation}
    Hence,
    $
    \delta \coloneqq 
    B(\gamma,\hat{\pi}) - B(\hat{\pi},\hat{\pi})
    \leq 0
    $.
    We show $\delta = 0$.
    For $t \in (0,1)$, we set $\pi^{(t)} =  \hat{\pi} + t (\gamma - \hat{\pi})$.
    Clearly, 
    $(P_X)_\# \pi^{(t)} = \xi$ for all $t \in (0,1)$.
    Due to the bilinearity and symmetry of $B$,
    we obtain
    \[
    B(\pi^{(t)},\pi^{(t)})
    = B(\hat{\pi},\hat{\pi}) 
    + 2t 
    \underbrace{
    B(\hat{\pi},\gamma - \hat{\pi})
    }_{= B(\hat{\pi},\gamma) - B(\hat{\pi},\hat{\pi}) = \delta}
    + t^2 B(\gamma - \hat{\pi},\gamma - \hat{\pi})
    \]
    Minimality of $\hat{\pi}$ 
    ensures
    $B(\pi^{(t)},\pi^{(t)}) \geq B(\hat{\pi},\hat{\pi})$,
    so that
    \[
    0 \leq
    2\delta + t B(\gamma - \hat{\pi},\gamma - \hat{\pi}).
    \]
    As $\delta \leq 0$,
    taking the limit $t \to 0$
    yields $0 = \delta = B(\gamma,\hat{\pi}) - B(\hat{\pi},\hat{\pi})$.
    Hence, $\hat{\pi}$ is a minimizer of \cref{eq:proof_sub_min_problem},
    so that \cref{prop:tilde_Q_GW_biconvex_minimization} yields that each
    $\hat{\pi}_i$ is concentrated on $A_i$.
    If additionally $\xi(A_i \cap A_j) = 0$ for all $i \neq j$,
    the uniqueness part of \cref{prop:tilde_Q_GW_biconvex_minimization}
    gives $\hat{\pi}_i = \xi\vert_{V_i}$, $i \in [n]$, as desired.
\end{proof}

\subsection{Existence and Structure of GW Quantizers}\label{subapp:proof_of_3_4}

\begin{proof}[Proof of \cref{thm:GW_quantization_existence_characterization}]
    Let $\XX = (X, g, \xi) \in \GM$ be arbitrary. 
    We first show existence. 
    Let
    $(\YY^{(k)})_{k \in \N} \subset \GM_n$ be a minimizing sequence of
    \cref{eq:GW_quant} with
    $\YY^{(k)} = ([n], G^{(k)}, \upsilon^{(k)})$, $k \in \N$,
    whose measures are
    fully supported, 
    see \cref{prop:rep-with-larger-cardinality},
    and let
    $\pi^{(k)} \in \Pio(\XX, \YY^{(k)})$, $k \in \N$. Define
    $\hat{\YY}^{(k)} \coloneqq ([n], \hat{G}^{(k)}, \upsilon^{(k)}) \in \GM_n$
    with
    \[
    \hat{G}^{(k)}_{i,i'}
    \coloneqq \frac{1}{\upsilon^{(k)}_i \upsilon^{(k)}_{i'}}
    \iint_{X^2} g(x,x') \dx\pi_i^{(k)}(x) \dx\pi_{i'}^{(k)}(x'),
    \quad i, i' \in [n].
    \]
    By \cref{thm:pointwise_lb_gauge}, it holds
    \[
    \GW_2^2(\XX, \YY^{(k)})
    = Q_{\GW}(G^{(k)}, (\pi_i^{(k)})_{i=1}^n)
    \geq Q_{\GW}(\hat{G}^{(k)}, (\pi_i^{(k)})_{i=1}^n)
    \geq \GW_2^2(\XX,\hat{\YY}^{(k)})
    \geq \qGW(\XX)^2,
    \]
    so that
    \[
    \qGW(\XX)^2
    = \lim_{k \to \infty} 
    Q_{\GW}(\hat{G}^{(k)}, (\pi_i^{(k)})_{i=1}^n) 
    \]
    and $(\hat{\YY}^{(k)})_{k \in \N}$ 
    is also a minimizing sequence of \cref{eq:GW_quant}.
    Since
    $(\pi^{(k)})_{k \in \N}$ 
    is a sequence in the weakly compact set $\Pi(\xi, *_n)$,
    we can select a converging subsequence,
    again denoted by $(\pi^{(k)})_k$.
    Let $\pi \in \p(X \times [n])$ be its limit
    and set $\upsilon \coloneqq (P_{[n]})_\# \pi$
    which is the weak limit of 
    the corresponding subsequence of marginals
    $(\upsilon^{(k)})_{k \in \N}$.
    We set $I \coloneqq \{i \in [n] : \upsilon_i > 0\}$ and note that
    $\pi_i(X) = \upsilon_i = 0$, i.e., $\pi_i = 0$, for $i \notin I$.
    For $i, i' \in I$, \cref{lem:product_convergence} applied with $h = g$ yields
    \[
    \hat{G}^{(k)}_{i,i'}
    \xrightarrow{k \to \infty}
    \hat{G}_{i,i'}
    \coloneqq \frac{1}{\upsilon_i \upsilon_{i'}}
    \iint_{X^2} g(x,x') \dx\pi_i(x) \dx\pi_{i'}(x').
    \]
    In the context of 
    \cref{lem:L2_bary} 
    with $d=1$, $f=g$ and $\alpha = \pi_i^{(k)} \otimes \pi_{i'}^{(k)}$ 
    we get $m = \hat{G}^{(k)}_{i,i'}$,
    so applying the lemma with
    $a_0 = 0$ 
    yields
    \begin{align*}
    \iint_{X^2} \big(g - \hat{G}^{(k)}_{i,i'}\big)^2
    \dx\pi_i^{(k)} \dx\pi_{i'}^{(k)}
    = 
    &\iint_{X^2} g^2 \dx\pi_i^{(k)} \dx\pi_{i'}^{(k)}
    - \big(\hat{G}^{(k)}_{i,i'}\big)^2 \upsilon^{(k)}_i \upsilon^{(k)}_{i'}
    \\
    \xrightarrow{k \to \infty} 
    &\iint_{X^2} g^2 \dx\pi_i \dx\pi_{i'}
    - \big(\hat{G}_{i,i'}\big)^2 \upsilon_i \upsilon_{i'},
    \qquad i,i' \in I,
    \end{align*}
    where the convergence of the first term
    holds again by \cref{lem:product_convergence}.
    Thus, 
    we arrive at
    \begin{align*}
    \qGW(\XX)^2
    &= \lim_{k \to \infty} Q_{\GW}(\hat{G}^{(k)}, (\pi_i^{(k)})_{i=1}^n)
    \geq \sum_{i, i' \in I} \lim_{k \to \infty}
    \iint_{X^2} \big(g - \hat{G}^{(k)}_{i,i'}\big)^2
    \dx\pi_i^{(k)} \dx\pi_{i'}^{(k)} \\
    &= \sum_{i, i' \in I}
    \left(\iint_{X^2} g^2 \dx\pi_i \dx\pi_{i'}
    - \big(\hat{G}_{i,i'}\big)^2 \upsilon_i \upsilon_{i'}\right)
    = \sum_{i, i' \in I}
    \iint_{X^2} \big(g - \hat{G}_{i,i'}\big)^2 \dx\pi_i \dx\pi_{i'},
    \end{align*}
    where the last equality follows as before, from 
    \cref{lem:L2_bary}.
    Now set $\hat{\YY} \coloneqq ([n], \hat{G}, \upsilon) \in \GM_n$, 
    where $\hat{G}$ is
    extended by zero outside of $I \times I$. As $\pi_i = 0$ for $i \notin I$, the
    right-hand side above equals the cost of the coupling
    $\pi \in \Pi(\xi, \upsilon)$ between $\XX$ and $\hat{\YY}$, so that
    \[
    \qGW(\XX)^2
    \geq \iint_{(X \times [n])^2}
    \big(g(x,x') - \hat{G}_{i,i'}\big)^2 \dx\pi(x,i) \dx\pi(x',i')
    \geq \GW_2^2(\XX, \hat{\YY})
    \geq \qGW(\XX)^2,
    \]
    where the last inequality holds since
    $\hat{\YY}\in \GM_n$. Hence, all inequalities are
    equalities and $\hat{\YY}$ is a solution of
    \cref{eq:GW_quant}.

    We turn to the characterization.
    Let $\hat{\YY} = ([n], \hat{G}, \hat{\upsilon})$ 
    be an arbitrary solution of
    \cref{eq:GW_quant},
    with fully supported $\hat{\upsilon} \in \p([n])$.
    Let $\XX = (X, g, \xi) \in \GM$
    and let
    $\hat{\pi} \in \Pio(\XX, \hat{\YY})$. Note that
    $\hat{\pi}_i(X) = \hat{\upsilon}_i > 0$ for all $i \in [n]$.

    We show \cref{item:thm35_gauge}. Define $\tilde{G} \in \G^{(n)}_{\sym}$ by
    \[
    \tilde{G}_{i,i'}
    \coloneqq \frac{1}{\hat{\upsilon}_i \hat{\upsilon}_{i'}}
    \iint_{X^2} g(x,x') \dx\hat{\pi}_i(x) \dx\hat{\pi}_{i'}(x'),
    \quad i, i' \in [n],
    \]
    and set 
    $\tilde{\YY} \coloneqq ([n], \tilde{G}, \hat{\upsilon}) \in \GM_n$. 
    Since
    $\hat{\pi} \in \Pio(\XX, \hat{\YY})$, 
    \cref{thm:pointwise_lb_gauge} yields
    \[
    \GW_2^2(\XX, \hat{\YY})
    \geq \GW_2^2(\XX, \tilde{\YY})
    + \lVert \hat{G} - \tilde{G} \rVert^2_{L^2(\hat{\upsilon} \otimes \hat{\upsilon})}.
    \]
    As $\hat{\YY}$ solves
    \cref{eq:GW_quant}, 
    it holds
    $\GW_2^2(\XX, \tilde{\YY}) 
    \geq \qGW(\XX)^2
    = \GW_2^2(\XX, \hat{\YY})$, so that
    $\lVert \hat{G} - \tilde{G}
    \rVert^2_{L^2(\hat{\upsilon} \otimes \hat{\upsilon})} = 0$.
    Since $\hat{\upsilon}$ is fully supported, 
    this gives $\hat{G} = \tilde{G}$,
    as desired.

    We show \cref{item:thm35_concentration}. 
    With the gauge matrix $\hat{G}$ 
    of $\hat{\YY}$ fixed, 
    since
    $\hat{\pi} \in \Pio(\XX, \hat{\YY})$,
    it holds
    \[
    Q_{\GW}(\hat{G}, (\hat{\pi}_i)_{i=1}^n)
    = \GW_2^2(\XX, \hat{\YY})
    = \qGW(\XX)^2.
    \]
    Thus, 
    $(\hat{\pi}_i)_{i=1}^n$ is a solution
    to
    \[
    \inf_{
    \substack{
    (\pi_i)_{i=1}^n \subset \M_+(X) \\
    \sum_{i=1}^n \pi_i = \xi
    }
    }
    Q_{\GW}(\hat{G}, (\pi_i)_{i=1}^n).
    \]
    Hence, \cref{thm:Q_GW_measure_min_characterization} yields that $\hat{\pi}_i$ is
    concentrated on $A_i$ for every $i \in [n]$, where $\hat{c}_i = c_i^{(\hat{\pi})}$
    is taken with respect to the gauge matrix $\hat{G}$, which is
    \cref{item:thm35_concentration}.

    We show \cref{item:thm35_value}. 
    Since $\hat{\pi} \in \Pio(\XX, \hat{\YY})$, it
    holds
    \[
    \GW_2^2(\XX, \hat{\YY})
    = \sum_{i,i'=1}^n \iint_{X^2}
    \big(g - \hat{G}_{i,i'}\big)^2 \dx\hat{\pi}_i \dx\hat{\pi}_{i'}.
    \]
    By \cref{item:thm35_gauge}, $\hat{G}_{i,i'}$ is the mean of $g$ with respect to
    the measure $\hat{\pi}_i \otimes \hat{\pi}_{i'}$, 
    so that \cref{lem:L2_bary}
    applied with $a_0 = 0$ yields
    \[
    \iint_{X^2} \big(g - \hat{G}_{i,i'}\big)^2 \dx\hat{\pi}_i \dx\hat{\pi}_{i'}
    = \iint_{X^2} g^2 \dx\hat{\pi}_i \dx\hat{\pi}_{i'}
    - \hat{G}_{i,i'}^2 \,\hat{\upsilon}_i \hat{\upsilon}_{i'},
    \quad i, i' \in [n].
    \]
    Summing over $i, i' \in [n]$ and using $\sum_{i=1}^n \hat{\pi}_i = \xi$, we
    obtain
    \[
    \GW_2^2(\XX, \hat{\YY})
    = \iint_{X^2} g^2(x,x') \dx\xi(x) \dx\xi(x')
    - \sum_{i,i'=1}^n \hat{G}_{i,i'}^2 \,\hat{\upsilon}_i \hat{\upsilon}_{i'},
    \]
    which is \cref{item:thm35_value}.
    
    Finally, 
    assume additionally that $\xi(A_i \cap A_j) = 0$ 
    for all $i \neq j$.
    As shown in \cref{item:thm35_concentration}, 
    the family $(\hat{\pi}_i)_{i=1}^n$
    is a solution to
    $\inf Q_{\GW}(\hat{G}, (\pi_i)_{i=1}^n)$ 
    over all $(\pi_i)_{i=1}^n \subset
    \M_+(X)$ with $\sum_{i=1}^n \pi_i = \xi$, 
    so that the second part of
    \cref{thm:Q_GW_measure_min_characterization} yields
    $\hat{\pi}_i = \xi\vert_{V_i}$, $i \in [n]$, 
    for the Voronoi partition
    $(V_i)_{i=1}^n$ of $X$ 
    with respect to $(\hat{c}_i)_{i=1}^n$. 
    Since
    $(V_i)_{i=1}^n$ is a disjoint partition of 
    $X$, the map $T \colon X \to [n]$,
    $T(x) = i$ if $x \in V_i$, 
    is well-defined and it holds
    \[
    \hat{\pi}
    = \sum_{i=1}^n \xi\vert_{V_i} \otimes \delta_{i}
    = (\id, T)_\# \xi.
    \]
    In particular,
    $\hat{\upsilon}_i = \hat{\pi}_i(X) = \xi(V_i)$ for all $i \in [n]$, 
    i.e.,
    $\hat{\upsilon} = \sum_{i=1}^n \xi(V_i) \cdot \delta_{i}$.
    Inserting
    $\hat{\pi}_i = \xi\vert_{V_i}$ 
    into \cref{item:thm35_gauge} gives
    \[
    \hat{G}_{i,i'}
    = \frac{1}{\xi(V_i) \xi(V_{i'})}
    \iint_{V_i \times V_{i'}} g(x,x') \dx\xi(x) \dx\xi(x'),
    \quad i, i' \in [n],
    \]
    as desired.
\end{proof}

\section{Proofs for \texorpdfstring{\cref{sec:euclidean}}{Section 4}}\label{app:euclidean}

\subsection{An Upper Bound for Lipschitz Functions of Metrics}\label{subapp:proof_of_4_1}

\begin{proof}[Proof of \cref{prop:GW_quant_upper_bound}]
    Let $\varepsilon > 0$ be arbitrary.
    By \cref{eq:n-center-covering-radius}, 
    choose $Y=\{y_1,\dots,y_n\}\subset X$
    with
    \[
    \qW(\xi)^2 
    \geq  
    \int_X
    d_X^2(x,Y)
    \dx \xi(x) - \varepsilon,
    \]
    and let 
    $h \coloneqq \varphi \circ d_X\vert_{Y \times Y}$.
    In addition, let
    $\upsilon = T_{\#} \xi \in \p(Y)$ 
    be defined by
    \[
    T: X \to Y, x \mapsto y_i, 
    \quad 
    \text{where }
    i = \min\{
    j : d_X(x,y_j) \leq d_X(x,y_k)
    \text{ for all } k \in [n]
    \}.
    \]
    Let $\YY \coloneqq (Y,h,\upsilon)$ and note that $\qGW(\XX) \leq \GW_2(\XX,\YY)$ since $\YY \in \GM_n$.
    We apply the triangle inequality
    to obtain
    \begin{align*}
        &d_X(x,x') - d_X(T(x),T(x'))
        \\
        \leq 
        &\bigl(
        d_X(x,T(x)) 
        + d_X(T(x),T(x'))
        + d_X(T(x'),x')
        \bigr)
        - d_X(T(x),T(x'))
        \\
        = 
        &d_X(x,T(x)) + d_X(x',T(x')),
    \end{align*}
    for any $x,x' \in X$.
    An analogous application also yields
    the same estimate for
    $d_X(T(x),T(x')) - d_X(x,x')$, 
    so that in total
    \begin{equation}\label{eq:proof_triangle}
        \lvert d_X(x,x') - d_X(T(x),T(x')) \rvert
        \leq
        d_X(x,T(x)) + d_X(x',T(x')). 
    \end{equation}
    For the coupling 
    $\pi \coloneqq 
    (\id,T)_\# \xi \in \Pi(\xi,\upsilon)$, 
    we obtain
    \begin{align*}
        \GW_2^2(\XX,\YY)
        &\leq 
        \iint_{(X \times Y)^2} 
        (\varphi(d_X(x,x')) - \varphi(d_X(y,y')))^2
        \dx \pi(x,y) \dx \pi(x',y')
        \\
        &\leq 
        L^2
        \iint_{(X \times Y)^2} 
        (d_X(x,x') - d_X(y,y'))^2
        \dx \pi(x,y) \dx \pi(x',y')
        \\
        &=
        L^2
        \iint_{X^2} 
        (d_X(x,x') - d_X(T(x),T(x')))^2
        \dx \xi(x) \dx \xi(x').
    \end{align*}
    Now applying \cref{eq:proof_triangle}
    together with the elementary inequality
    $(a+b)^2 \leq 2(a^2 + b^2)$
    gives
    \begin{align*}
        \qGW(\XX)^2 \leq \GW_2^2(\XX,\YY)
        &\leq 
        L^2
        \iint_{X^2} 
        (d_X(x,T(x)) + d_X(x',T(x')))^2
        \dx \xi(x) \dx \xi(x')
        \\
        &\leq 
        2L^2
        \iint_{X^2} 
        d_X^2(x,T(x)) + d_X^2(x',T(x'))
        \dx \xi(x) \dx \xi(x')
        \\
        &\leq
        4 L^2 
        \int_{X} 
        d_X^2(x,T(x))
        \dx \xi(x)
        \\
        &\leq  4 L^2 \qW(\xi)^2 + 4L^2 \varepsilon.
    \end{align*}
    Letting $\varepsilon \rightarrow 0$ and taking square roots gives $\qGW(\XX) \leq 2L\,\qW(\xi)$, as desired.
\end{proof}

\subsection{Lower Bounds for the Euclidean and Squared-Euclidean Gauge}\label{subapp:proof_of_4_3}

For the proof of \cref{thm:lb_euclidean_gauges}, 
we require two auxiliary lemmas, 
and the following additional notation.
For a measure
$\mu \in \p(\R^d)$
and $f:\R^d \to \R$,
we write 
\begin{equation}\label{eq:var_f_def}
\mathrm{Var}_\mu(f) \coloneqq \Big\lVert f - \int_{\R^d} f \dx\mu \Big\rVert^2_{L^2(\mu)}
= \frac{1}{2} \iint_{\R^d\times\R^d} \big(f(x)-f(x')\big)^2 \dx\mu(x)\dx\mu(x'),
\end{equation}
where the second equality follows by expanding the squares (both sides equal
$\int f^2\dx\mu - (\int f\dx\mu)^2$).
\begin{lemma}\label{lem:var_collapse}
    Let $\mu, \nu \in \p(\R^d)$ and $f \in L^2(\mu \otimes \nu)$. Then for every $c \in \R$,
    \[
    \lVert f - c \rVert^2_{L^2(\mu \otimes \nu)}
    \;\geq\;
    \int_{\R^d} \mathrm{Var}_\mu\big( f(\cdot, y) \big) \dx\nu(y).
    \]
\end{lemma}

\begin{proof}
    Let $c \in \R$ be arbitrary.
    By \cref{lem:L2_bary},
    $\lVert f(\cdot,y) - c \rVert^2_{L^2(\mu)}
    \geq 
    \mathrm{Var}_\mu\big(f(\cdot,y)\big)
    $ 
    for $\nu$-a.e.\ fixed $y$.
    Integrating over $y$ with respect to $\nu$ yields the claim.
\end{proof}

\begin{lemma}\label{lem:quadratic_var_lb}
    Let $\mu, \nu \in \p(\R^d)$, where $\mu$ has finite fourth and
    $\nu$ finite second moment, with covariance matrices
    $\Sigma_\mu, \Sigma_\nu$. Then
    \[
    \int_{\R^d} \mathrm{Var}_{\mu}\big( \lVert \cdot - y \rVert^2 \big)
    \dx\nu(y)
    \;\geq\; 4 \tr( \Sigma_\mu \Sigma_\nu ).
    \]
\end{lemma}
\begin{proof}
        Set $m_2 \coloneqq \int_{\R^d} \lVert x \rVert^2 \dx\mu(x)$,
        $\m(\mu) \coloneqq \int_{\R^d} x \dx\mu(x)$ and
        $c_\mu \coloneqq \int_{\R^d} \big( \lVert x \rVert^2 - m_2 \big) x
        \dx\mu(x)$.
        Expanding $\lVert x - y \rVert^2 = \lVert x \rVert^2
        - 2 \langle x, y \rangle + \lVert y \rVert^2$ and using that
        $\mathrm{Var}_\mu$ is invariant under adding the (in $x$) constant
        $\lVert y \rVert^2$,
        \[
        q(y) \coloneqq
        \mathrm{Var}_{\mu}\big( \lVert \cdot - y \rVert^2 \big)
        = \mathrm{Var}_{\mu}\big( \lVert \cdot \rVert^2
          - 2 \langle \cdot,\, y \rangle \big).
        \]
        The function $x \mapsto \lVert x \rVert^2 - 2 \langle x, y \rangle$ has
        $\mu$-mean $m_2 - 2 \langle \m(\mu), y \rangle$, so by the definition of
        the variance,
        \[
        q(y)
        = \int_{\R^d} \Big( \big( \lVert x \rVert^2 - m_2 \big)
        - 2 \langle x - \m(\mu),\, y \rangle \Big)^2 \dx\mu(x).
        \]
        We expand the square.
        Firstly, due to
        $\int_{\R^d} \big( \lVert x \rVert^2 - m_2 \big) \dx\mu(x) = 0$,
        it holds
        \[
        \int_{\R^d} \big( \lVert x \rVert^2 - m_2 \big)
        \langle x - \m(\mu),\, y \rangle \dx\mu(x) 
        = \langle c_\mu, y \rangle.
        \]
        Secondly, we get
        $\int_{\R^d} \langle x - \m(\mu),\, y \rangle^2 \dx\mu(x)
        = y^\top \Sigma_\mu y$.
        Thus, we obtain
        \[
        q(y)
        = \mathrm{Var}_{\mu}\big( \lVert \cdot \rVert^2 \big)
        - 4 \langle c_\mu, y \rangle
        + 4\, y^\top \Sigma_\mu y.
        \]
    Thus, 
    $q$ is a  polynomial of degree (at most) two in $y$ 
    with Hessian
    $\nabla^2 q = 8 \Sigma_\mu$.
    For polynomials of degree two, the second-order Taylor
    expansion is exact, so that
    expanding $q$ around $\m(\nu)$ and using
    $\int_{\R^d} (y - \m(\nu)) \dx\nu(y) = 0$ yields
    (cf.~\cite[Theorem 1.5]{seber2003linear}),
    \begin{align*}
    \int_{\R^d} q \dx\nu
    &= q(\m(\nu))
    + \frac{1}{2} \int_{\R^d}
    (y - \m(\nu))^\top \big( \nabla^2 q \big) (y - \m(\nu)) \dx\nu(y)
    \\
    &= q(\m(\nu))
    + \frac{1}{2} \int_{\R^d}
    \tr\bigl(\big( \nabla^2 q \big) (y - \m(\nu)) (y - \m(\nu))^\top \bigr) \dx\nu(y)
    \\
    &= q(\m(\nu)) + 4 \tr(\Sigma_\mu \Sigma_\nu).
    \end{align*}
    Since $q(\m(\nu)) = \mathrm{Var}_\mu\big( \lVert \cdot - \m(\nu) \rVert^2 \big)
    \geq 0$ as a pointwise variance, the claim follows.
\end{proof}

We finally turn to the proof of \cref{thm:lb_euclidean_gauges}.

\begin{proof}[Proof of \cref{thm:lb_euclidean_gauges}]
    Let $\YY = ([n],G,\upsilon) \in \GM_n$ be arbitrary with fully supported
    $\upsilon \in \p([n])$, see \cref{prop:rep-with-larger-cardinality}, and let
    $\pi \in \Pio(\XX,\YY)$. We set $\bar\pi_i \coloneqq \pi_i/\upsilon_i \in \p(X)$,
    with mean $m_i \coloneqq \int_X x \dx\bar\pi_i(x)$ and covariance matrix
    $\Sigma_i \coloneqq \int_X (x-m_i)(x-m_i)^\top \dx\bar\pi_i(x)$, $i \in [n]$.
    Note that by closed convexity of $X$, 
    $m_i \in X$.
    Since $\pi_i \leq \xi$, all appearing moments are finite.
    Due to $\pi \in \Pio(\XX,\YY)$,
    it holds
    \[
    \GW_2^2(\XX,\YY) = \sum_{i,i'=1}^n \upsilon_i\upsilon_{i'}
    \big\lVert g - G_{i,i'} \big\rVert^2_{L^2(\bar\pi_i \otimes \bar\pi_{i'})}.
    \]
    Applying \cref{lem:var_collapse} to each summand with 
    $f = g$,
    $\mu = \bar\pi_i$,
    $\nu = \bar\pi_{i'}$ 
    and $c = G_{i,i'}$, 
    we obtain
    \begin{equation}\label{eq:conditional_variance_lb}
    \GW_2^2(\XX,\YY)
    \;\geq\;
    \sum_{i,i'=1}^n 
    \upsilon_i
    \upsilon_{i'}
    \int_X
    \mathrm{Var}_{\bar\pi_i}\big( g(\cdot, y) \big) \dx\bar{\pi}_{i'}(y)
    =
    \sum_{i = 1}^n 
    \upsilon_i
    \int_X
    \mathrm{Var}_{\bar\pi_i}\big( g(\cdot, y) \big) \dx\xi(y),
    \end{equation}
    where, in the last equality,
    we used 
    $\sum_{i'=1}^n \upsilon_{i'}\bar\pi_{i'} = \xi$.

    We show \cref{item:lb_sq}, i.e.\ let $g = \lVert \cdot - \cdot \rVert^2$.
    By \cref{lem:quadratic_var_lb} applied with $\mu = \bar\pi_i$ and
    $\nu = \xi$, and since $\tr(AB) \geq \lambda_{\min}(B) \tr(A)$ for
    positive semi-definite $A, B$,
    \begin{align*}
    \GW_2^2(\XX, \YY)
    &\geq 4 \sum_{i=1}^n \upsilon_i \tr( \Sigma_i \Sigma_\xi )
    \geq 4 \lambda_{\min}(\Sigma_\xi) \sum_{i=1}^n \upsilon_i \tr(\Sigma_i)
    \\
    &= 4 \lambda_{\min}(\Sigma_\xi) \sum_{i=1}^n \int_X
      \lVert x - m_i \rVert^2 \dx\pi_i(x)
    \geq 4 \lambda_{\min}(\Sigma_\xi) \qW(\xi)^2,
    \end{align*}
    where the last step follows from \cref{eq:qW_as_min_QW}. Taking
    square roots and the infimum over $\YY$ yields \cref{item:lb_sq}.
    
    We show \cref{item:lb_dist}, i.e.\ $g = \lVert \cdot - \cdot \rVert$.
    Fix $i \in [n]$ and $y \in X$.
    Since $\lVert x-y\rVert + \lVert x'-y\rVert \le 2\diam(X)$ for $x,x'\in X$,
    it holds
    \begin{align*}
    \big(\lVert x-y\rVert^2 - \lVert x'-y\rVert^2\big)^2
    &= \big(\lVert x-y\rVert - \lVert x'-y\rVert\big)^2\big(\lVert x-y\rVert + \lVert x'-y\rVert\big)^2
    \\
    &\leq 4\diam(X)^2 \big(\lVert x-y\rVert - \lVert x'-y\rVert\big)^2.
    \end{align*}
    Integrating over 
    $(x,x')$ 
    with respect to $\bar{\pi}_i \otimes \bar{\pi}_{i}$ 
    and using the second equality 
    in \cref{eq:var_f_def},
    we get
    $\mathrm{Var}_{\bar{\pi}_i}(\lVert \cdot - y\rVert^2) \le 4\diam(X)^2\,\mathrm{Var}_{\bar{\pi}_i}(\lVert \cdot - y\rVert)$.
    Together with
    \cref{eq:conditional_variance_lb} 
    and the estimate from
    \cref{item:lb_sq},
    we arrive at
    \[
    \GW_2^2(\XX, \YY)
    \geq \frac{1}{4 \diam(X)^2}
    \sum_{i=1}^n \upsilon_i \int_X
    \mathrm{Var}_{\bar\pi_i}\big( \lVert \cdot - y \rVert^2 \big)
    \dx\xi(y)
    \geq \frac{\lambda_{\min}(\Sigma_\xi)}{\diam(X)^2} \, \qW(\xi)^2.
    \]
    Taking square roots and the infimum over $\YY$ yields
    \cref{item:lb_dist}.
\end{proof}

\subsection{Two-Sided Bounds for the Scalar-Product Gauge}\label{subapp:proof_of_4_5}

For the proof of \cref{thm:lower_and_upper_bound_scalar_product}, 
we require 
two auxiliary lemmas.
The first one ensures that
for any $\XX = (X,\langle \cdot, \cdot \rangle, \xi) \in \GM$ with $X \subset \R^d$,
the measure $\xi$ has finite second moment.

\begin{lemma}\label{lem:second_moment}
    Let $X \subset \R^d$ and $\xi \in \p(X)$.
    If $\langle \cdot,\cdot \rangle \in L^2(\xi \otimes \xi)$,
    then $\xi \in \p^{(2)}(X)$.
\end{lemma}

\begin{proof}
    Set 
    $I \coloneqq  \iint_{\R^d \times \R^d} \langle x, x' \rangle^2 \dx\xi(x) \dx\xi(x') < \infty$.
    For $R > 0$, 
    let $B_R \coloneqq \{ x \in X : \|x\| \leq R \}$ 
    and
    \[
    M_R \coloneqq \int_{B_R} x x^\top \dx\xi(x) \in \R^{d \times d},
    \]
    which is well-defined and symmetric positive semi-definite, with entries bounded by $R^2$.
    Let $\lambda_1,\dotsc,\lambda_d \geq 0$ be the eigenvalues of $M_R$.
    Then the Cauchy-Schwarz inequality yields
    \[
    \tr(M_R)^2
    = 
    \biggl(\sum_{i=1}^d \lambda_i \biggr)^2
    \leq \Bigl( \sum_{i=1}^d \lambda_i^2 \Bigr) \Bigl( \sum_{i=1}^d 1^2 \Bigr)
    = d \sum_{i=1}^d \lambda_i^2
    = d \, \|M_R\|_F^2 .
    \]
    Furthermore, 
    due to $\tr(x x^\top) = \|x\|^2$
    and $\langle a, b \rangle^2 = \tr(a a^\top b b^\top)$,
    we obtain
    \begin{align*}
    \Bigl( \int_{B_R} \|x\|^2 \dx\xi(x) \Bigr)^2
    &= \tr\Bigl( 
    \int_{B_R} xx^\top \dx\xi(x) \Bigr)^2
    = \tr(M_R)^2
    \leq 
    d \, \|M_R\|_F^2
    = d \tr(M_R M_R)
    \\
    &= d \iint_{B_R^2} \tr(x x^\top x' (x')^\top) \dx\xi(x) \dx\xi(x')
    = d \iint_{B_R^2} \langle x, x' \rangle^2 \dx\xi(x) \dx\xi(x')
    \leq d \, I .
    \end{align*}
    Hence $\int_{B_R} \|x\|^2 \dx\xi \leq \sqrt{d \, I}$ 
    for every $R > 0$, 
    and taking the limit $R \to \infty$,
    monotone convergence yields
    \[
    \int_{X} \|x\|^2 \dx\xi(x) \leq \sqrt{d \, I} < \infty,
    \]
    as desired.
\end{proof}

For the remainder of the section,
for $\pi \in \Pi(\xi, *_n)$,
we define
\begin{equation}\label{eq:def_a_i_and_M}
a_i(\pi) \coloneqq \int_X x \dx\pi_i(x) \in \R^d, \quad i \in [n]
\qquad \text{and} \qquad
M(\pi) \coloneqq \sum_{i=1}^n \frac{a_i(\pi) (a_i(\pi))^\top}{\pi_i(X)} \in \R^{d \times d},
\end{equation}
where we set $0/0 = 0$. 
This makes $M(\pi)$ well-defined as 
$\lVert a_i(\pi) \rVert^2 \leq \pi_i(X) \int_X \lVert x \rVert^2 \dx\pi_i(x)$ 
by Jensen's inequality, 
so that $a_i(\pi) = 0$ whenever $\pi_i(X) = 0$.

In the scalar product case, the quantization problem
is equivalent to maximization of the function 
$\pi \mapsto \|M(\pi)\|_F^2$,
which is the result of the following lemma.

\begin{lemma}\label{lem:scalar_reduction}
Let $\XX = (X,\langle \cdot, \cdot \rangle, \xi) \in \GM$
with $X \subset \R^d$.
Let 
$M_\xi\coloneqq\int_X xx^\top\dx\xi$. 
For every $\pi\in\Pi(\xi, *_n)$
\[
\inf_{G\in\G^{(n)}_{\sym}} Q_{\GW}\bigl(G,(\pi_i)_{i=1}^n\bigr)
= \|M_\xi\|_F^2 - \|M(\pi)\|_F^2.
\]
\end{lemma}

\begin{proof}
    Note that $M_\xi$ is well-defined 
    by \cref{lem:second_moment}.
    Indices $i \in [n]$ with $\pi_i(X) = 0$ 
    contribute to neither side of the claimed identity,
    so we may assume 
    $\pi_i(X) > 0$ 
    for all $i \in [n]$.
    For fixed $\pi \in \Pi(\xi, *_n)$, 
    the minimization of 
    $G \mapsto Q_{\GW}(G, (\pi_i)_{i=1}^n)$ 
    is solved by $\hat{G}$
    from \cref{thm:pointwise_lb_gauge}.
    Due to the bilinearity of the scalar product
    it is given by
    \[
    \hat{G}_{i,i'}
    =\frac{1}{\pi_i(X) \pi_{i'}(X)}
    \iint_{X^2} \langle x,x' \rangle
    \dx \pi_{i}(x)
    \dx \pi_{i'}(x')
    = \frac{\langle a_i(\pi),a_{i'}(\pi)\rangle}{{\pi_i(X)} \pi_{i'}(X)},
    \quad
     i,i' \in [n].
    \]
    Applying \cref{lem:L2_bary}
    with $a_0 = 0$ as usual
    yields
    \begin{equation*}
    \inf_{G \in \G^{(n)}_{\sym}} Q_{\GW}(G, (\pi_i)_{i=1}^n)
    = \iint_{X^2} \langle x, x' \rangle^2 \dx\xi(x) \dx\xi(x')
    - \sum_{i,i'=1}^n \frac{\langle a_i(\pi), a_{i'}(\pi) \rangle^2}{\pi_i(X) \pi_{i'}(X)}.
    \end{equation*}
    Due to
    $\langle a, b \rangle^2 = \tr(a a^\top b b^\top)$ 
    for $a, b \in \R^d$, 
    the first term
    equals $\|M_\xi\|_F^2$ and
    the second term equals 
    $\tr(M(\pi)^2) = \lVert M(\pi) \rVert_F^2$,
    as desired.
\end{proof}

The previous result allows us to
derive appropriate bounds
in the scalar product case.
Before we proceed with this,
we recall the \emph{Löwner order} of symmetric matrices.
That is,
for two symmetric matrices $A,B$,
we write $A \succeq B$ if $A - B$ is positive semi-definite.
In particular,
$A$ is positive semi-definite 
if and only if $A \succeq 0$.
We also write $A \preceq B$
if $B - A$ is positive semi-definite.

The bounds rely on the trace/covariance structure 
of the scalar product GW
problem, 
which has also been exploited e.g.\ in
\cite{delon2022gromov,sliced_gw,ZGMS2024}.

\begin{proof}[Proof of \cref{thm:lower_and_upper_bound_scalar_product}]
    By \cref{lem:second_moment}, 
    it holds
    $\xi \in \p^{(2)}(X)$, 
    so $\qW(\xi)$ and $M_\xi$ are
    well-defined. 
    Before we show the desired bounds, 
    we express $\qGW(\XX)$ via a matrix trace.
    For $\pi \in \Pi(\xi, *_n)$ consider $a_i(\pi)$, $i \in [n]$,
    as well as $M(\pi)$ defined in \cref{eq:def_a_i_and_M}. 
    By \cref{eq:qGW_as_min_Q_GW} and \cref{lem:scalar_reduction},
    \[
    \qGW(\XX)^2
    = \inf_{\pi \in \Pi(\xi, *_n)} \inf_{G \in \G^{(n)}_{\sym}} Q_{\GW}(G,(\pi_i)_{i=1}^n)
    = \inf_{\pi \in \Pi(\xi, *_n)} \bigl( \|M_\xi\|_F^2 - \|M(\pi)\|_F^2 \bigr).
    \]
    Now further define $\tilde M(\pi) \coloneqq M_\xi - M(\pi)$.
    As $M_\xi$ and $M(\pi)$ are symmetric, for every $\pi \in \Pi(\xi, *_n)$
    \[
    \|M_\xi\|_F^2 - \|M(\pi)\|_F^2
    = \tr(M_\xi^2) - \tr(M(\pi)^2)
    = \tr\bigl( \tilde M(\pi)\,(M_\xi + M(\pi)) \bigr),
    \]
    hence
    \begin{equation}\label{eq:proof_sp_master}
    \qGW(\XX)^2 = \inf_{\pi \in \Pi(\xi, *_n)} \tr\bigl( \tilde M(\pi)\,(M_\xi + M(\pi)) \bigr).
    \end{equation}
    We remark that it holds
    \[
    \tilde M(\pi) = \sum_{i=1}^n \int_X \Bigl( x - \frac{a_i(\pi)}{\pi_i(X)} \Bigr)
        \Bigl( x - \frac{a_i(\pi)}{\pi_i(X)} \Bigr)^\top \dx\pi_i(x),
    \]
    which follows by expanding the integrand and using $\sum_{i=1}^n \pi_i = \xi$.
    We note that
    $\tilde{M}(\pi)$
    as well as $M(\pi) = M_\xi - \tilde M(\pi)$ 
    are positive semi-definite,
    hence
    \[
    M_\xi \preceq M_\xi + M(\pi) \preceq 2 M_\xi
    \quad \text{so that}
    \quad 
    \lambda_{\min}(M_\xi) I \preceq M_\xi + M(\pi) \preceq 2 \lambda_{\max}(M_\xi) I.
    \]
    Since $\tilde M(\pi) \succeq 0$,
    the map $A \mapsto \tr(\tilde M(\pi)\,A)$ 
    is
    monotone with respect to the Löwner order 
    $\preceq$. 
    Combining this with
    \cref{eq:proof_sp_master} 
    and the monotonicity of the infimum, 
    we obtain
    \begin{align}\label{eq:proof_both_sides}
    \lambda_{\min}(M_\xi) 
    \inf_{\pi \in \Pi(\xi, *_n)} 
    \tr(\tilde{M}(\pi))
    &\leq 
    \underbrace{
    \inf_{\pi \in \Pi(\xi, *_n)} 
     \tr\bigl( \tilde M(\pi)\,(M_\xi + M(\pi)) \bigr)
    }_{= \qGW(\XX)^2}
    \nonumber
    \\
    &\leq 
    2 \lambda_{\max}(M_\xi)
    \inf_{\pi \in \Pi(\xi, *_n)} 
    \tr(\tilde{M}(\pi)).
    \end{align}
    Finally, note that $\inf_{\pi \in \Pi(\xi, *_n)} \tr(\tilde M(\pi)) = \qW(\xi)^2$.
    Indeed,
    \begin{align*}
    \tr(\tilde M(\pi))
    = \sum_{i=1}^n \int_X \Bigl\| x - \frac{a_i(\pi)}{\pi_i(X)} \Bigr\|^2 \dx\pi_i(x)
    &= \sum_{i=1}^n \int_X \Bigl\| x -
    \Bigl(\frac{1}{\pi_i(X)}\int_{X} x' \dx \pi_i(x')\Bigr)
    \Bigr\|^2 \dx\pi_i(x)
    \\
    &=
    \inf_{(y_i)_{i=1}^n \subset X} 
    \sum_{i=1}^n \int_X \| x -
    y_i\|^2 \dx\pi_i(x),
    \end{align*}
    where the last line follows by 
    \cref{lem:L2_bary} 
    and closed convexity of $X \subset \R^d$.
    Taking the infimum over
    $\pi \in \Pi(\xi,*_n)$ 
    and comparing with 
    \cref{eq:qW_as_min_QW} 
    gives the claim.
    Plugging this into \cref{eq:proof_both_sides} 
    and taking the square root 
    yields the
    desired bounds.
\end{proof}

The smallest eigenvalue in the lower bound of
\cref{thm:lower_and_upper_bound_scalar_product} 
may render it unnecessarily weak,
particularly when it is zero.
A sharper estimate can be achieved
by applying the following result.

\begin{lemma}[Ruhe trace inequality including subspaces]\label{lem:ruhe}
Let $A, B \in \mathbb{R}^{n\times n}$ be symmetric and positive semi-definite with $\textrm{range}(A) \subseteq \textrm{range}(B) \eqqcolon S$ and $s \coloneqq \dim(S)$. Let $\lambda_1(B) \geq \ldots \geq \lambda_s(B) > 0$ be the nonzero eigenvalues of $B$ and $\lambda_1(A) \geq \ldots \geq \lambda_s(A) \geq 0$ the $s$ largest eigenvalues of $A$ (which can include zeros). Then
\[
\tr(AB) \geq \sum_{i=1}^s \lambda_i(A) \lambda_{s+1-i}(B).
\]
\end{lemma}

\begin{proof}
    By symmetry of $A$ and $\textrm{range}(A) \subseteq S$, we have $S^\perp \subseteq \ker(A)$, so $S$ is invariant under $A$, and analogously under $B$. Hence $\tr(AB) = \tr(A_{|S} B_{|S})$. Applying the standard Ruhe trace inequality (see \cite[9.H.1.h.]{marshall2011inequalities}) yields the claim.
\end{proof}

The proof of \cref{thm:lower_and_upper_bound_scalar_product} implicitly relies on the estimate
$\tr(\tilde{M} M_\xi) \geq \lambda_{\min}(M_\xi)\tr(\tilde{M})$ which amounts to
scaling $\xi$ by $\sqrt{\lambda_{\min}}$ along every direction.
Since $0 \preceq \tilde{M} \preceq M_\xi$ implies
$\operatorname{range}(\tilde{M}) \subseteq \operatorname{range}(M_\xi)$, 
applying
\cref{lem:ruhe}
instead
gives the sharper estimate
\[
\tr(\tilde{M} M_\xi)
\;\geq\;
\sum_{i=1}^{s} \lambda_i(\tilde{M})\, \lambda_{s+1-i}(M_\xi),
\qquad s \coloneqq \operatorname{rank}(M_\xi).
\]
This scales the directions of $\tilde{M}$ 
by the varying nonzero eigenvalues of
$M_\xi$. 
In particular, $\lambda_{\min}$ in
\cref{thm:lower_and_upper_bound_scalar_product} 
may be replaced by the smallest
nonzero eigenvalue of $M_\xi$.

\subsection{Monge Maps for the Scalar-Product Gauge}\label{subapp:proof_of_4_8}

In the following,
we prove 
\cref{thm:scalar_product_partition_attainment}.
Due to 
\cref{eq:qGW_as_min_Q_GW} and \cref{lem:scalar_reduction}
the quantization problem
$\qGW(\XX)$
reduces to the problem
$
\sup_{\pi \in \Pi(\xi, *_n)} \|M(\pi)\|_F^2
$.
We use this reduction in the following
to establish the existence of Monge maps
from $\XX$ to solutions of 
the GW quantizers.
By
\cref{thm:GW_quantization_existence_characterization},
this has direct implications 
on the structure of the solutions
themselves and gives a reformulation of the problem
that is akin to the partitioning problem 
\cref{eq:W_quant_as_partitioning_problem}.
Before proving \cref{thm:scalar_product_partition_attainment},
we establish some fundamental properties
of the reduced objective in the following lemma.
The proofs of both the lemma and the theorem rely on the notation
$a_i(\pi)$, $i \in [n]$, and $M(\pi)$,
$\pi \in \Pi(\xi,\ast_n)$,
introduced in \cref{eq:def_a_i_and_M}.

\begin{lemma}\label{lem:M_convex_continuous}
The map $\pi\mapsto\|M(\pi)\|_F^2$ 
is convex and weakly continuous on $\Pi(\xi, *_n)$.
\end{lemma}

\begin{proof}
    Set 
    \[
    F:\Pi(\xi, *_n) \to [0,\infty), \pi\mapsto\|M(\pi)\|_F^2.
    \]
    For convexity,
    let $\pi^{(0)}, \pi^{(1)} \in \Pi(\xi, *_n)$, $t \in [0,1]$ 
    and
    $\pi^{(t)} \coloneqq (1-t)\pi^{(0)} + t\pi^{(1)} \in \Pi(\xi, *_n)$.
    Note that $(a_i(\pi), \pi_i(X))$ depends linearly on $\pi$.
    Furthermore,
    the
    joint convexity of the matrix-fractional
    function $(a,p) \mapsto a a^\top / p$ in the Löwner order
    (see e.g.\ \cite[Sec.~3.1.7]{boyd2004convex}) 
    gives, 
    for each $i \in [n]$ with $\pi_i^{(0)}(X), \pi_i^{(1)}(X) > 0$,
    \[
    \frac{a_i(\pi^{(t)}) (a_i(\pi^{(t)}))^\top}{\pi_i^{(t)}(X)}
    \ \preceq\
    (1-t)\,\frac{a_i(\pi^{(0)}) (a_i(\pi^{(0)}))^\top}{\pi_i^{(0)}(X)}
    + t\,\frac{a_i(\pi^{(1)}) (a_i(\pi^{(1)}))^\top}{\pi_i^{(1)}(X)}.
    \]
    For the remaining indices the inequality holds directly:
    If
    $\pi_i^{(0)}(X) = \pi_i^{(1)}(X) = 0$, then
    $a_i(\pi^{(0)}) = a_i(\pi^{(1)}) = a_i(\pi^{(t)}) = 0$ 
    and all three terms vanish.
    If exactly one vanishes, 
    say 
    $\pi_i^{(0)}(X) = 0$ (so $a_i(\pi^{(0)}) = 0$), 
    then by linearity 
    $a_i(\pi^{(t)}) = t\,a_i(\pi^{(1)})$ 
    and $\pi_i^{(t)}(X) = t\,\pi_i^{(1)}(X)$.
    Thus, for $t \in (0,1)$ both sides equal
    $t\,a_i(\pi^{(1)}) (a_i(\pi^{(1)}))^\top / \pi_i^{(1)}(X)$
    (the cases $t \in \{0,1\}$ being trivial).
    Hence the inequality holds for every 
    $i \in [n]$, 
    and summing over $i \in [n]$ yields
    \[
    0 \preceq M(\pi^{(t)}) \preceq (1-t)\,M(\pi^{(0)}) + t\,M(\pi^{(1)}).
    \]
    Moreover, $\lVert \cdot \rVert_F$
    is monotone with respect to
    $\preceq$ 
    for positive semi-definite matrices,
    i.e.\
    for $0 \preceq A \preceq B$,
    \[
    \lVert A \rVert_F^2 = \tr(A^2) \leq \tr(A B) \leq \tr(B^2) = \lVert B \rVert_F^2,
    \]
    as the trace of a product of two positive semi-definite matrices is non-negative.
    Combining both 
    with the triangle inequality 
    and the convexity of 
    $s \mapsto s^2$, 
    we obtain
    \[
    F(\pi^{(t)})
    = \tr\!\big( M(\pi^{(t)})^2 \big)
    \leq \tr\!\Big( \big( (1-t) M(\pi^{(0)}) + t M(\pi^{(1)}) \big)^2 \Big)
    = \big\lVert (1-t) M(\pi^{(0)}) + t M(\pi^{(1)}) \big\rVert_F^2
    \]
    \[
    \leq \big( (1-t)\lVert M(\pi^{(0)})\rVert_F + t \lVert M(\pi^{(1)})\rVert_F \big)^2
    \leq (1-t) F(\pi^{(0)}) + t F(\pi^{(1)}),
    \]
    that is, $F$ is convex.
    
    For continuity, let 
    $\pi^{(k)} \weakly \pi$ in $\Pi(\xi, *_n)$,
    we show $M(\pi^{(k)}) \to M(\pi)$,
    $k \to \infty$.
    Fix a continuous cutoff $\chi_R \colon X \to [0,1]$ with
    $\1_{\{\|x\| \le R\}} \le \chi_R(x) \le \1_{\{\|x\| \le R+1\}}$,
    $x \in X$.
    Each coordinate of 
    $x \mapsto x\,\chi_R(x)$
    is bounded and continuous, so
    $\pi_i^{(k)} \weakly \pi_i$ gives
    $\int_X x\,\chi_R \dx\pi_i^{(k)} \to \int_X x\,\chi_R \dx\pi_i$ 
    as $k \to \infty$.
    Moreover, 
    since $1 - \chi_R$ is supported on $\{\|x\| > R\}$ and 
    by $\pi_i^{(k)}\le\xi$,
    we obtain
    \[
    \Big\| \int_X x\,(1-\chi_R) \dx\pi_i^{(k)} \Big\|
    \ \le\ \int_{\{\|x\|>R\}} \|x\| \dx\pi_i^{(k)}
    \ \le\ \int_{\{\|x\|>R\}} \|x\| \dx\xi
    \to 0,
    \]
    uniformly in $k$, and the same bound holds with $\pi_i$ in place of
    $\pi_i^{(k)}$. 
    Hence, for every $R>0$,
    \[
    \big\|a_i(\pi^{(k)})-a_i(\pi)\big\|
    \le \Big\|\int_X x\,\chi_R\dx\pi_i^{(k)}-\int_X x\,\chi_R\dx\pi_i\Big\|
    + \int_{\{\|x\|>R\}}\!\|x\|\dx\pi_i^{(k)}
    + \int_{\{\|x\|>R\}}\!\|x\|\dx\pi_i .
    \]
    Letting first $k\to\infty$, the first term vanishes while the
    remaining terms are each bounded by
    $\int_{\{\|x\|>R\}}\|x\|\dx\xi$.
    Letting then $R\to\infty$ gives
    $a_i(\pi^{(k)})\to a_i(\pi)$.
    If $\pi_i(X) > 0$, 
    the $i$-th summand of $M(\pi^{(k)})$ 
    converges to that of $M(\pi)$.
    If $\pi_i(X) = 0$, 
    then for any $R > 0$,
    Jensen's inequality yields
    \[
    \Big\lVert 
    \frac{a_i(\pi^{(k)}) (a_i(\pi^{(k)}))^\top}{\pi_i^{(k)}(X)} 
    \Big\rVert_F
    = \frac{\lVert a_i(\pi^{(k)}) \rVert^2}{\pi_i^{(k)}(X)}
    \leq \int_X \lVert x \rVert^2 \dx\pi_i^{(k)}
    \leq \underbrace{R^2 \pi_i^{(k)}(X)}_{\to 0, k \to \infty}
    + \int_{\{\lVert x \rVert > R\}} \lVert x \rVert^2 \dx\xi.
    \]
    Hence,
    \[
    \limsup_{k \to \infty} 
    \Big\lVert 
    \frac{a_i(\pi^{(k)}) (a_i(\pi^{(k)}))^\top}{\pi_i^{(k)}(X)} 
    \Big\rVert_F
    \leq \int_{\{\lVert x \rVert > R\}} \lVert x \rVert^2 \dx\xi
    \xrightarrow{R \to \infty} 0.
    \]
    Thus, 
    $M(\pi^{(k)}) \to M(\pi)$ 
    and $F$ is weakly continuous on $\Pi(\xi, *_n)$.
\end{proof}

\begin{proof}[Proof of \cref{thm:scalar_product_partition_attainment}]
    Using \cref{eq:qGW_as_min_Q_GW}
    together with \cref{lem:scalar_reduction},
    the quantization problem \cref{eq:GW_quant} 
    becomes equivalent to maximizing
    \[
    F(\pi) \coloneqq \lVert M(\pi) \rVert_F^2
    \]
    over $\Pi(\xi, *_n)$.
    By \cref{lem:M_convex_continuous},
    $F$ is convex and weakly continuous on $\Pi(\xi, *_n)$.
    Moreover, $\Pi(\xi, *_n)$ is convex 
    and weakly compact.
    Thus,
    the Bauer maximum principle,
    see e.g.\ \cite[Thm.~7.69]{aliprantis2006infinite},
    ensures that $F$ attains its maximum over $\Pi(\xi, *_n)$ 
    at an extreme point of $\Pi(\xi, *_n)$.

    Now, 
    we show that every extreme point of $\Pi(\xi, *_n)$ is induced by a map.
    Let $\pi \in \Pi(\xi, *_n)$ be extreme 
    and let 
    $\pi = \int_X \delta_x \otimes \kappa_x \dx\xi(x)$ 
    be its disintegration,
    that is
    \[
    \int_{X \times [n]} f(x,i) \dx\pi(x,i)
    = \int_X \int_{[n]} f(x,i) \dx\kappa_x(i) \dx\xi(x)
    \qquad \text{for all bounded measurable } f.
    \]
    To achieve a contradiction,
    we assume that
    $B \coloneqq \{x \in X : \max_{i \in [n]} \kappa_x(\{i\}) < 1\}$
    satisfies $\xi(B) > 0$.
    For $x \in B$, 
    let $i_x < j_x$ denote the two smallest indices 
    with $\kappa_x(\{i_x\}) > 0$, $\kappa_x(\{j_x\}) > 0$ 
    and define the measurable family of signed measures
    \[
    \sigma_x \coloneqq
    \min\big\{\kappa_x(\{i_x\}), \kappa_x(\{j_x\})\big\}
    \cdot \big(\delta_{i_x} - \delta_{j_x}\big) \1_B(x).
    \]
    Clearly $\sigma_x([n]) = 0$ for all $x \in X$, 
    hence
    $\pi^{\pm} \coloneqq 
    \int_X \delta_x \otimes (\kappa_x \pm \sigma_x) \dx\xi(x)$
    define two distinct elements of $\Pi(\xi, *_n)$.
    Furthermore,
    $\pi = \frac{1}{2}(\pi^+ + \pi^-)$,
    contradicting extremality.
    Hence $\xi(B)=0$ and thus $\kappa_x$ 
    is a Dirac measure for $\xi$-a.e.\ $x \in X$,
    i.e., $\pi = (\id, T)_\# \xi$ 
    for a measurable map $T \colon X \to [n]$.

    By the first part of the proof,
    there exists a measurable $T \colon X \to [n]$ 
    such that 
    $\pi^* \coloneqq (\id, T)_\# \xi$ 
    maximizes $F$ over $\Pi(\xi, *_n)$
    and thus also minimizes 
    \begin{equation}\label{eq:proof_scalar_reduction}
    \inf_{G \in \G^{(n)}_{\sym}} Q_{\GW}(G, (\pi_i)_{i=1}^n)
    = \|M_\xi\|_F^2 - \|M(\pi)\|_F^2
    \end{equation}
    over $\Pi(\xi, *_n)$, cf.~\cref{lem:scalar_reduction}.
    By \cref{eq:qGW_as_min_Q_GW},
    this implies
    $\qGW(\XX)^2 = \inf_{G \in \G_{\sym}^{(n)}} Q_{\GW}(G,(\pi_i^*)_{i=1}^n)$.
    We set $V_i \coloneqq T^{-1}(\{i\})$.
    For indices $i \in [n]$ with $\xi(V_i) = 0$,
    it holds $\pi_i^* = 0$,
    so they contribute to neither side of \cref{eq:proof_scalar_reduction},
    and the corresponding entries of $G$ 
    do not affect $Q_{\GW}(G, (\pi_i^*)_{i=1}^n)$.
    Hence we may apply \cref{thm:pointwise_lb_gauge}
    to the indices with $\xi(V_i) > 0$.
    Thus,
    the latter is minimized by the block-mean gauge, 
    which by bilinearity of the
    scalar product equals
    \[
    \hat{G}_{i,i'}
    = \biggl\langle \frac{a_i(\pi^*)}{\pi_i^*(X)}, \frac{a_{i'}(\pi^*)}{\pi_{i'}^*(X)} \biggr\rangle
    = \biggl\langle
    \dfrac{1}{\xi(V_i)} \int_{V_i} x \dx\xi(x),
    \dfrac{1}{\xi(V_{i'})} \int_{V_{i'}} x \dx\xi(x)
    \biggr\rangle
    = \langle \m_\xi(V_i), \m_\xi(V_{i'}) \rangle,
    \]
    where we used
    $\pi^*_i = \xi\vert_{V_i}$.
    Now set
    $\hat{\upsilon} = \sum_{i=1}^n \xi(V_i) \delta_{i}$
    giving rise to the gm-space $\hat{\YY} = ([n],\hat{G},\hat{\upsilon})$.
    As $\pi^* \in \Pi(\xi,\hat\upsilon)$,
    \[
    \qGW(\XX)^2 \le \GW_2^2(\XX,\hat\YY) \le Q_{\GW}(\hat G,(\pi_i^*)) = \qGW(\XX)^2,
    \]
    so $\hat\YY$ is a solution of the quantization problem and $\pi^* = (\id,T)_\#\xi$ an optimal coupling.
\end{proof}

The crux of the argument
is that every extreme point of $\Pi(\xi, *_n)$ 
is induced by a map.
We establish this directly,
by a perturbation argument tailored to our setting. 
Alternatively,
it can be derived from the theory of Young measures.
More precisely,
under the disintegration
$\pi = \int_X \delta_x \otimes \kappa_x \dx\xi(x)$, 
the set $\Pi(\xi, *_n)$ is identified with the
transition kernels $x \mapsto \kappa_x \in \p([n])$, 
whose extreme points are precisely the
Dirac-valued ones (see e.g.\ \cite{castaing1977convex}).

\subsection{Closed-Form for the Scalar-Product Case in 1D}\label{subapp:proof_of_4_9}

\begin{proof}[Proof of \cref{thm:closed_form_1d}]
As 
$\|u\|=1$, 
we obtain
\[
T(x) \,T(x')
= u^\top x\, u^\top x'
= x^\top \underbrace{u u^\top x'}_{= x'} 
= x^\top x'
= \langle x,x' \rangle,
\qquad x,x' \in X \subset \spann(u).
\]
Thus, setting $\tilde{\xi}\coloneqq T_\#\xi$ 
we see that 
$\XX$ is homomorphic to
$\tilde{\XX} = (\R,(s,t)\mapsto st,\tilde{\xi})$. 
For a partition
$(V_i)_{i=1}^n$ of $\R$
write
\[
S\bigl((V_i)_{i=1}^n\bigr)
\coloneqq
\sum_{i=1}^n \tilde{\xi}(V_i)\,\m_{\tilde{\xi}}(V_i)^2.
\]
Since $\XX$ and $\tilde{\XX}$ are homomorphic, 
it holds
$\qGW(\XX) = \qGW(\tilde{\XX})$.
Using
\cref{eq:GW_quant_as_partitioning_problem},
we obtain
\begin{align*}
\qGW(\XX)^2
= \qGW(\tilde{\XX})^2
&=
\inf_{(V_i)_{i=1}^n}
\sum_{i,i'=1}^n\int_{V_i\times V_{i'}}
\bigl(st-\m_{\tilde{\xi}}(V_i)\m_{\tilde{\xi}}(V_{i'})\bigr)^2
\dx\tilde{\xi}(s)\dx\tilde{\xi}(t)
\\
&=
\inf_{(V_i)_{i=1}^n}
\sum_{i,i'=1}^n\Bigl(
\int_{V_i\times V_{i'}} (st)^2 \dx\tilde{\xi}(s)\dx\tilde{\xi}(t)
- \m_{\tilde{\xi}}(V_i)^2 \m_{\tilde{\xi}}(V_{i'})^2\,\tilde{\xi}(V_i)\tilde{\xi}(V_{i'})
\Bigr)
\\
&=
\inf_{(V_i)_{i=1}^n}
\Bigl(
\Bigl(\int_\R t^2 \dx\tilde{\xi}\Bigr)^2
- \Bigl(\sum_{i=1}^n \tilde{\xi}(V_i)\, \m_{\tilde{\xi}}(V_i)^2\Bigr)^2
\Bigr)
\\
&=\inf_{(V_i)_{i=1}^n}
\bigl(M_2(\tilde{\xi})^2-S((V_i))^2\bigr)
=M_2(\tilde{\xi})^2-\sup_{(V_i)}S((V_i))^2 ,
\end{align*}
where equality in the second line uses $\int_{V_i} s\dx\tilde{\xi}=\tilde{\xi}(V_i)\,\m_{\tilde{\xi}}(V_i)$.
On the other hand, 
a similar derivation yields
\[
\sum_i\int_{V_i}(t-\m_{\tilde{\xi}}(V_i))^2\dx\tilde{\xi}(t) = M_2(\tilde{\xi})-S((V_i)),
\]
so that \cref{eq:W_quant_as_partitioning_problem} gives
$\qW(\tilde{\xi})^2=M_2(\tilde{\xi})-\sup_{(V_i)}S((V_i)_{i=1}^n)$. 
Both problems amount to maximizing $S$, and since $S \ge 0$ it holds
$\sup_{(V_i)} S((V_i))^2 = \bigl(\sup_{(V_i)} S((V_i))\bigr)^2$.
Therefore
\[
\qGW(\XX)^2
= M_2(\tilde{\xi})^2 - \sup_{(V_i)} S((V_i))^2
= M_2(\tilde{\xi})^2 - \bigl(M_2(\tilde{\xi}) - \qW(\tilde{\xi})^2\bigr)^2
= \qW(\tilde{\xi})^2\bigl(2 M_2(\tilde{\xi}) - \qW(\tilde{\xi})^2\bigr),
\]
as desired.
\end{proof}

\section{Proofs of \texorpdfstring{\cref{sec:alg}}{Section 5}}

\subsection{Directional Derivative of \texorpdfstring{$R$}{R}}\label{subapp:proof_of_6_1}

\begin{proof}[Proof of \cref{prop:R_directional_derivative}]
    Let $\eta \in \Pi(\xi, *_n)$.
    Write 
    $\pi_t \coloneqq (1-t)\pi + t\eta = \pi + t (\eta - \pi)$. 
    Since $\eta_i \in \M_+(X)$ and $\pi_i(X) > 0$ 
    by assumption,
    it holds
    \[
    (\pi_t)_i(X) = (1-t)\pi_i(X) + t\,\eta_i(X) > 0, \quad t \in [0,1)
    \]
    so $\hat {G}(\pi_t)$ is well-defined. 
    Let
    \[
    a_{i,i'}(t) \coloneqq \iint_{X^2} g\dx(\pi_t)_i\dx(\pi_t)_{i'},
    \qquad
    b_i(t) \coloneqq (\pi_t)_i(X)=(1-t)\,\pi_i(X)+t\,\eta_i(X),
    \]
    so that 
    $\hat{G}(\pi_t)_{i,i'} = \frac{a_{i,i'}(t)}{b_i(t)b_{i'}(t)}$.
    Since 
    $b_i(t)=(1-t)\,\pi_i(X)+t\,\eta_i(X)\geq(1-t)\,\pi_i(X)>0$ for
    $t\in[0,1)$, 
    each $\hat G(\pi_t)_{i,i'}$ is a
    quotient of polynomials in 
    $t$ 
    with non-vanishing denominator, 
    hence differentiable on $[0,1)$.
    Similarly,
    as $t \mapsto \pi_t$ is affine,
    \[
    t \mapsto Q_{\GW}(G,\pi_t)
    = 
    \sum_{i,i'\in[n]}\iint_{X^2}
    \bigl(g-G_{i,i'}\bigr)^2\dx(\pi_t)_i\dx(\pi_t)_{i'}
    \]
    is a quadratic polynomial.
    In addition,
    $Q_{\GW}$ is a quadratic polynomial in its first argument.
    Therefore,
    as a composition of differentiable functions,
    $t\mapsto R(\pi_t)=Q_{\GW}(\hat G(\pi_t),\pi_t)$ 
    is differentiable on $[0,1)$.
    Hence, the one-sided limit
    $\lim_{t\downarrow0}\tfrac{R(\pi_t)-R(\pi)}{t}$ exists and equals
    $\tfrac{\mathrm d}{\mathrm dt}\big|_{t=0}R(\pi_t)$,
    so it suffices to compute the latter.
    As $\hat G(\pi_t)$ 
    is the unconstrained minimizer 
    of the strictly convex quadratic
    $G\mapsto Q_{\GW}(G,\pi_t)$ (\cref{thm:pointwise_lb_gauge}), 
    its first-order
    optimality gives $\partial_G Q_{\GW}(\hat G(\pi_t),\pi_t)=0$, 
    so the chain rule gives
    \[
    \frac{\mathrm d}{\mathrm dt} R(\pi_t)
    = 
    \underbrace{\partial_G Q_{\GW}(\hat G(\pi_t),\pi_t)
    }_{=\,0}
      \bigl[\tfrac{\mathrm d}{\mathrm dt} \hat G(\pi_t)\bigr]
    + \partial_\pi Q_{\GW}(\hat G(\pi_t),\pi_t)[\underbrace{\frac{\mathrm{d} \pi_t}{ \mathrm{d} t}}_{= (\eta - \pi)}].
    \]
    For the remaining derivative, 
    we see that 
\begin{align*}
\partial_\pi Q_{\GW}(G,\gamma)[\alpha]
&= \lim_{s\to 0}\frac{Q_{\GW}(G,\gamma+s\,\alpha)-Q_{\GW}(G,\gamma)}{s}
\\
&= \lim_{s\to 0}\sum_{i,i'\in[n]}\iint_{X^2}(g-G_{i,i'})^2
   \bigl(\dx \alpha_i\dx\gamma_{i'} + \dx\gamma_i\dx \alpha_{i'} + s\,\dx \alpha_i\dx \alpha_{i'}\bigr)
\\
&= \sum_{i,i'\in[n]}\iint_{X^2}(g-G_{i,i'})^2
   \bigl(\dx \alpha_i\dx\gamma_{i'} + \dx\gamma_i\dx \alpha_{i'}\bigr)
\\
&= 2 \sum_{i,i'\in[n]}
\iint_{X^2}(g-G_{i,i'})^2
\dx \alpha_i
\dx\gamma_{i'},
\qquad G \in \G_{\sym}^{(n)}, \gamma \in \Pi(\xi, *_n),
\end{align*}
where the last line follows by symmetry 
of $g$ and $G$.
In summary, 
we have
\begin{align*}
    \frac{\mathrm d}{\mathrm dt} R(\pi_t)
    = \partial_\pi Q_{\GW}(\hat G(\pi_t),\pi_t)[\eta - \pi]
    = 2\sum_{i,i'\in[n]}\iint_{X^2}\bigl(g-\hat G(\pi_t)_{i,i'}\bigr)^2
   \dx(\pi_t)_{i'}
   \dx(\eta_i-\pi_i).
\end{align*}
The formula holds for all $t \in [0,1)$,
in particular for $t = 0$,
we obtain
\begin{align*}
    \frac{\mathrm d}{\mathrm dt}\Big|_{t=0} R(\pi_t)
    = 2\sum_{i,i'\in[n]}\iint_{X^2}\bigl(g-\hat G(\pi)_{i,i'}\bigr)^2
   \dx \pi_{i'}
   \dx(\eta_i-\pi_i)
   = 2 \sum_{i=1}^n \int_X \hat{c}_i^{(\pi)} \dx (\eta_i - \pi_i),
\end{align*}
as desired.
\end{proof}

\subsection{Convergence to Stationary Quantizers}\label{subapp:proof_of_6_3}

\begin{proof}[Proof of \cref{thm:convergence}]
Let $(\pi^{(k)})_k$ be generated by \cref{alg:1} with the line search of
\cref{rem:line_search}, i.e.\ $\pi^{(k+1)}=(1-t_k)\pi^{(k)}+t_k\eta^{(k)}$ with
$\eta^{(k)}_i=\xi\vert_{V_i^{(k)}}$ for 
a Voronoi partition 
$(V_i^{(k)})_{i=1}^n$
with respect to
$(\hat c_i^{(\pi^{(k)})})_{i=1}^n$, 
and
$t_k \in \argmin_{s\in[0,1]} Q_{\GW}\!\bigl(\hat G(\pi^{(k)}),(1-s)\pi^{(k)}+s\eta^{(k)}\bigr)$.
Analogously to \cref{rem:line_search},
let $a_k$ and $b_k$ be
given by
\begin{align*}
&a_k \coloneqq 
\sum_{i,i'=1}^n \iint_{X^2} 
\bigl( g - \hat G(\pi^{(k)})_{i,i'} \bigr)^2 
\dx (\eta_i^{(k)} - \pi_i^{(k)}) \dx (\eta_{i'}^{(k)} - \pi_{i'}^{(k)}),
\\
&b_k \coloneqq -2\sum_{i=1}^n\int_X \hat c_i^{(\pi^{(k)})}\dx(\eta^{(k)}_i-\pi^{(k)}_i).
\end{align*}
Note that $b_k$ is non-negative 
since, by \cref{lem:linear_assignment}, 
$\eta^{(k)}$ minimizes
$\eta\mapsto\sum_i\int_X\hat c_i^{(\pi^{(k)})}\dx\eta_i$ over $\Pi(\xi, *_n)$.
Before showing stationarity, 
we show that $b_k \to 0$, as $k \to \infty$.
It holds
\[
Q_{\GW}\!\bigl(\hat G(\pi^{(k)}),(1-s)\pi^{(k)}+s\eta^{(k)}\bigr)
= R(\pi^{(k)}) - s\,b_k + s^2 a_k,\qquad s\in[0,1].
\]
By construction, 
$t_k$ minimizes this quadratic 
over $[0,1]$,
so that, with
$\pi^{(k+1)}=(1-t_k)\pi^{(k)}+t_k\eta^{(k)}$,
\[
Q_{\GW}\!\bigl(\hat G(\pi^{(k)}),\pi^{(k+1)}\bigr)
= \min_{s\in[0,1]}\bigl(R(\pi^{(k)}) - s\,b_k + s^2 a_k\bigr).
\]
Since $R(\pi)=\inf_{G}Q_{\GW}(G,\pi)\leq Q_{\GW}(\hat G(\pi^{(k)}),\pi)$ 
for every $\pi\in\Pi(\xi, *_n)$,
in particular $R(\pi^{(k+1)})\leq Q_{\GW}(\hat G(\pi^{(k)}),\pi^{(k+1)})$, 
and therefore
\[
R(\pi^{(k)})-R(\pi^{(k+1)})
\ \geq\ R(\pi^{(k)}) - \min_{s\in[0,1]}\bigl(R(\pi^{(k)}) - s\,b_k + s^2 a_k\bigr)
= \max_{s\in[0,1]}\bigl(s\,b_k - s^2 a_k\bigr).
\]
Each $\hat G(\pi^{(k)})_{i,i'}$ is an average of $g$ and
since $g$ is bounded, we have
$\lvert \hat G(\pi^{(k)})_{i,i'}\rvert \leq \lVert g\rVert_\infty$
and thus
$(g - \hat G(\pi^{(k)})_{i,i'})^2 \leq 4\lVert g\rVert_\infty^2$.
As each signed measure $\eta_i^{(k)} - \pi_i^{(k)}$ has total variation at most $2$,
it follows that $a_k$ is bounded above by some
$A \in (0,\infty)$ uniformly in $k$.
Using this bound,
we get
\[
R(\pi^{(k)})-R(\pi^{(k+1)})
\ \geq\ \max_{s\in[0,1]}\bigl(s\,b_k - s^2 A\bigr)
\ \geq\ \min\Bigl\{\tfrac{b_k}{2},\,\tfrac{b_k^2}{4A}\Bigr\},
\]
where the last inequality holds since the concave function $s\mapsto s\,b_k - s^2 A$ is
maximized on $[0,1]$ at $s^\ast=\min\{1,\tfrac{b_k}{2A}\}$, yielding $\tfrac{b_k^2}{4A}$ (if
$s^\ast=\tfrac{b_k}{2A}$) or $b_k-A\geq\tfrac{b_k}{2}$ (if $s^\ast=1$, i.e., $b_k\geq 2A$).
Summing the per-step bound,
for every $K\in\N$,
we obtain
\[
R(\pi^{(0)})
\geq
R(\pi^{(0)}) - R(\pi^{(K)})
=
\sum_{k=0}^{K-1}\bigl(R(\pi^{(k)})-R(\pi^{(k+1)})\bigr)
\geq
\sum_{k=0}^{K-1}\min\Bigl\{\tfrac{b_k}{2},\tfrac{b_k^2}{4A}\Bigr\}.
\]
Letting $K\to\infty$ yields
$\sum_{k=0}^{\infty}\min\bigl\{\tfrac{b_k}{2},\tfrac{b_k^2}{4A}\bigr\}\leq R(\pi^{(0)})<\infty$,
so $b_k\to 0$
as $k \to \infty$.

Finally, 
we show the desired stationarity.
Let $\pi^\ast$ be a cluster point with
$\pi^\ast_i(X)>0$ for all $i$. 
By compactness of $\Pi(\xi, *_n)$, 
pass to subsequences
$\pi^{(k_j)}\weakly\pi^\ast$
and $\eta^{(k_j)}\weakly\bar\eta\in\Pi(\xi, *_n)$.
We first record the convergence
\begin{equation}\label{eq:proof_cost_convergence}
\int_X \hat c_i^{(\pi^{(k_j)})}\dx\sigma_i^{(k_j)}
\xrightarrow{j\to\infty}
\int_X \hat c_i^{(\pi^\ast)}\dx\sigma_i,
\qquad i\in[n],
\end{equation}
valid for any $(\sigma^{(k_j)})_j\subset\Pi(\xi, *_n)$ with $\sigma^{(k_j)}\weakly\sigma$.
Indeed, 
firstly since $\pi_i^{(k_j)}(X)\to\pi_i^\ast(X)>0$ 
and, 
by \cref{lem:product_convergence}
with $h=g$, 
the numerators of $\hat G(\pi^{(k_j)})_{i,i'}$ converge,
we have
$\hat G(\pi^{(k_j)})_{i,i'}\to\hat G(\pi^\ast)_{i,i'}$ for all $i,i'\in[n]$. 
Expanding the square in
\[
\int_X \hat c_i^{(\pi^{(k_j)})}\dx\sigma_i^{(k_j)}
= \sum_{i'=1}^n \iint_{X^2}\bigl(g-\hat G(\pi^{(k_j)})_{i,i'}\bigr)^2
\dx\sigma_i^{(k_j)}\dx\pi_{i'}^{(k_j)}
\]
and applying \cref{lem:product_convergence} to $h\in\{g^2, g, 1\}\subset L^1(\xi\otimes\xi)$
(using that $g$ is bounded) yields \cref{eq:proof_cost_convergence}.
Applying \cref{eq:proof_cost_convergence} with $\sigma^{(k_j)}=\eta^{(k_j)}$ and with
$\sigma^{(k_j)}=\pi^{(k_j)}$ gives
$b_{k_j}\to -2\sum_i\int_X\hat c_i^{(\pi^\ast)}\dx(\bar\eta_i-\pi^\ast_i)$
as $j \to \infty$.
By the first part, $(b_{k_j})_{j \in \N}$ also converges to $0$,
so that
\begin{equation}\label{eq:proof_stationary_point_condition}
2\sum_i\int_X\hat c_i^{(\pi^\ast)}\dx(\bar\eta_i-\pi^\ast_i)=0.
\end{equation}
We conclude the proof
by showing
\begin{equation}\label{eq:proof_to_show}
2\sum_i\int_X \hat c_i^{(\pi^\ast)}\dx \eta_i
\geq
2\sum_i\int_X \hat c_i^{(\pi^\ast)}\dx \bar\eta_i
\qquad\text{for all }\eta\in\Pi(\xi, *_n).
\end{equation}
By construction,
\cref{lem:linear_assignment}
ensures that
$\eta^{(k)}$ minimizes
$\eta\mapsto 2\sum_i\int_X \hat c_i^{(\pi^{(k)})}\dx\eta_i$ over $\Pi(\xi, *_n)$,
so for every $\eta\in\Pi(\xi, *_n)$,
\[
2\sum_i\int_X \hat c_i^{(\pi^{(k)})}\dx\eta_i
\geq 
2\sum_i\int_X \hat c_i^{(\pi^{(k)})}\dx\eta^{(k)}_i.
\]
We evaluate this inequality along the subsequence 
$(k_j)_j$ and pass to the limit
via \cref{eq:proof_cost_convergence}.
Firstly, with the constant sequence $\sigma^{(k_j)}\equiv\eta$
the left-hand side converges to
$2\sum_i\int_X \hat c_i^{(\pi^\ast)}\dx \eta_i$.
Secondly, with $\sigma^{(k_j)}=\eta^{(k_j)}$
the right-hand side converges to
$2\sum_i\int_X \hat c_i^{(\pi^\ast)}\dx \bar\eta_i$
and we obtain \cref{eq:proof_to_show}.
Combining this inequality 
with \cref{eq:proof_stationary_point_condition}, 
for every
$\eta\in\Pi(\xi, *_n)$,
\[
2\sum_i\int_X\hat c_i^{(\pi^\ast)}\dx(\eta_i-\pi^\ast_i)
\ \geq\ 2\sum_i\int_X\hat c_i^{(\pi^\ast)}\dx(\bar\eta_i-\pi^\ast_i)
\ =\ 0,
\]
so $\pi^\ast$ is stationary,
as desired.
\end{proof}

\subsection{Unit-Step Descent for the Scalar-Product Gauge}\label{subapp:proof_of_6_4}

\begin{proof}[Proof of \cref{cor:scalar_product_convergence}]
    From \cref{alg:1},
    let $(\pi^{(k)})_{k \in \N}$ and $((V_i^{(k)})_{i=1}^n)_{k \in \N}$
    be the sequence of measures and Voronoi partitions, respectively.
    By \cref{lem:scalar_reduction} 
    and \cref{lem:M_convex_continuous},
    $R$ is concave on $\Pi(\xi, *_n)$,
    that is
    \[
    (1-t) R(\pi) + t R(\gamma)
    \leq
    R((1-t)\pi + t \gamma)
    \]
    which can be rearranged as
    \[
    R(\gamma)-R(\pi)
    \leq
    \frac{R((1-t)\pi+t\gamma)-R(\pi)}{t},
    \qquad \pi,\gamma \in \Pi(\xi, *_n), \, t \in (0,1].
    \]
    Let $k \in \N$ and take $\pi = \pi^{(k)}$
    as well as
    $\gamma = \pi^{(k+1)}$
    in the above inequality.
    Taking the limit $t \downarrow 0$
    and applying \cref{prop:R_directional_derivative}
    yields
    \[
    R(\pi^{(k+1)}) 
    - R(\pi^{(k)})
    \leq  2 \sum_{i=1}^n \int_X \hat c_i^{(\pi^{(k)})} \dx(\pi_i^{(k+1)} - \pi_i^{(k)}).
    \]
    Furthermore, 
    by construction
    $\pi_i^{(k+1)} = \xi \vert_{V_i^{(k)}}$,
    so \cref{lem:linear_assignment}
    yields
    \[
    \sum_{i=1}^n \int_X \hat c_i^{(\pi^{(k)})} \dx (\pi_i^{(k+1)} - \pi_i^{(k)})
    \leq 
    0.
    \]
    In summary $R(\pi^{(k+1)}) \leq R(\pi^{(k)})$.

    Assume now that $X$ is finite. 
    Then $\Pi(\xi, *_n)$ is a polytope whose vertices are the
    hard assignments $\pi_i=\xi|_{V_i}$,
    and the unit-step update
    $\pi^{(k+1)}_i=\xi|_{V_i^{(k)}}$
    produces such a vertex at every step. 
    If $\pi^{(k)}_i$ is not concentrated on
    $\{x \in X : \hat{c}_i^{(\pi^{(k)})}(x) = \min_j \hat{c}_j^{(\pi^{(k)})}(x)\}$,
    $i \in [n]$,
    then by \cref{lem:linear_assignment} the inequality
    $\sum_i\int \hat c_i^{(\pi^{(k)})}\dx(\pi_i^{(k+1)}-\pi_i^{(k)})\leq 0$ is strict,
    so that
    $R(\pi^{(k+1)})<R(\pi^{(k)})$. 
    Since $R$ then strictly decreases over the finitely many vertex
    values, 
    this can occur only finitely often; 
    hence after finitely many steps $\pi^{(k)}$ is
    concentrated on the desired sets.
\end{proof}

\section{Non-Uniformity of the Lower Quantization Bound}\label{app:example_two_sided_bounds}

Although \cref{cor:GW_quant_rate_1,cor:GW_quant_rate_2} 
pin down the rate, 
the
constants $c$ to bound $\qGW$ by $c\cdot\qW$ from below genuinely depend 
on the measure and cannot be chosen
uniformly over $\GM$. 
A single family witnesses this
simultaneously
for any of the three gauges
$g \in \{\|\cdot-\cdot\|,\ \|\cdot-\cdot\|^2,\ \langle\cdot,\cdot\rangle\}$.
More precisely,
let $\XX \coloneqq \XX^{(n)} \coloneqq (\R^2, g, \xi^{(n)})$,
where
\[
\xi^{(n)} \coloneqq \frac{1}{n+1}\Bigl(
\delta_{p_+} + \delta_{p_-} + \sum_{j=1}^{n-1}\delta_{q_j}\Bigr),
\qquad p_\pm \coloneqq (0,\pm 1),\ \ q_j \coloneqq (j,0).
\]
We consider a quantization with $n$ points.
Since each
$q_j$ has vanishing second coordinate, 
$\langle p_\pm, q_j\rangle = 0$ and
$\|p_\pm - q_j\| = \sqrt{j^2+1}$ 
are independent of the sign,
so $p_+$ and $p_-$ carry
identical gauge to every $q_j$, for each such $g$. 
Consider the $n$-point quantizer obtained
by merging $p_+, p_-$ into a single (abstract) point $z$.
Then, we set $\YY = (Y,h,\upsilon)$,
where $Y \coloneqq \{z,q_1,\dotsc,q_{n-1}\}$,
\[
h(y,y') =
\begin{cases}
\dfrac{1}{4}\displaystyle\sum_{s,s'\in\{+,-\}} g(p_s,p_{s'})
    & y = y' = z,\\[2.2ex]
g(p_+, q_j)
    & \{y,y'\} = \{z, q_j\},\ j \in [n-1],\\[1ex]
g(q_j, q_k)
    & y = q_j,\ y' = q_k,\ j,k \in [n-1].
\end{cases}
\]
and 
$\upsilon = \frac{2}{n+1}\,\delta_z + \frac{1}{n+1}\sum_{j=1}^{n-1}\delta_{q_j}$.
Under the canonical coupling,
\[
\pi = \frac{1}{n+1}\Bigl(
\delta_{(p_+,\, z)} + \delta_{(p_-,\, z)}
+ \sum_{j=1}^{n-1} \delta_{(q_j,\, q_j)}\Bigr)
\;\in\; \Pi\bigl(\xi^{(n)}, \upsilon\bigr),
\]
all pairs involving some $q_j$ contribute zero distortion.
Consequently, the only nonzero distortion under $\pi$ stems from the four pairs in
$\{p_+,p_-\}^2$, each of which is mapped to $(z,z)$. Hence
\begin{align*}
\GW_2^2(\XX,\YY)
&\leq
\iint \bigl(g(x,x') - h(y,y')\bigr)^2 \dx\pi(x,y)\dx\pi(x',y')
\\
&= \frac{1}{(n+1)^2} \underbrace{\sum_{s,s'\in\{+,-\}} \bigl(g(p_s,p_{s'}) - h(z,z)\bigr)^2}_{\eqqcolon C_g}.
\end{align*}
The sum in $C_g$ depends only on the fixed points
$p_\pm = (0,\pm 1)$ and not on $n$,
so it is a
constant $C_g < \infty$ for each of the three gauges. 
Moreover $C_g > 0$.
Indeed,
since
$g(p_+,p_-) \neq g(p_+,p_+)$, 
the four values $g(p_s,p_{s'})$ are not all equal, 
and hence
their deviation from the mean $h(z,z)$ does not vanish.
Since $\YY \in \GM_n$, this yields
\[
\qGW(\XX)^2 \leq \GW_2^2(\XX,\YY) \leq \frac{C_g}{(n+1)^2}.
\]
For the Wasserstein value, 
recall that $\qW(\xi^{(n)})$ 
is taken with respect to the
Euclidean metric in all three cases 
and
(cf.~\cref{eq:n-center-covering-radius})
\[
\qW(\xi^{(n)})^2 = \inf_{\lvert Y'\rvert \leq n} \int d^2(x,Y') \dx\xi^{(n)}(x),
\qquad d(x,Y') \coloneqq \min_{y \in Y'} \|x - y\| .
\]
Fix any $Y' \subset \R^2$ with $\lvert Y'\rvert \leq n$. 
Since $\lvert \supp(\xi^{(n)})\rvert = n+1$ 
and $\lvert Y' \rvert \leq n$,
the nearest-point assignment
that maps
$x \in \supp(\xi^{(n)})$ 
to a minimizer of $Y' \to [0,\infty), y \mapsto \|x - y\|$,
is not injective.
Thus, 
we can fix two distinct atoms 
$x_0 \neq x_0' \in \supp(\xi^{(n)})$
with a
common nearest point $y_0 \in Y'$.
As the atoms of $\xi^{(n)}$ are pairwise at Euclidean
distance at least $1$,
it holds
\[
1 \leq \|x_0 - x_0'\| \leq \|x_0 - y_0\| + \|y_0 - x_0'\| \leq 2 \max\{ d(x_0,Y'),\, d(x_0',Y') \}.
\]
Hence, 
at least one of them, 
say $x_0$, 
satisfies $d(x_0,Y') \geq \tfrac12$. 
Keeping only this atom's
contribution to the integral,
\[
\int d^2(x,Y') \dx\xi^{(n)}(x)
\geq \xi^{(n)}(\{x_0\})\, d^2(x_0,Y')
\geq \frac{1}{n+1}\cdot\frac14 .
\]
Since $Y'$ was arbitrary,
$\qW(\xi^{(n)})^2 \geq \tfrac{1}{4(n+1)}$.
Combining the two bounds,
\[
\frac{\qGW(\XX)}{\qW(\xi^{(n)})}
\leq \frac{\sqrt{C_g/(n+1)^2}}{\sqrt{1/(4(n+1))}}
= \frac{2\sqrt{C_g}}{\sqrt{n+1}}
\xrightarrow{n\to\infty} 0 .
\]
Hence no constant $c>0$ with $\qGW \geq c\,\qW$ can hold uniformly over $\GM$, for any of the
three gauges.

\section*{AI Tool Disclosure}
We used AI, primarily Claude Opus 5 and Fable 5, as follows: First, to identify many inaccuracies in earlier versions of this paper. Second, to generate most of the code for the experiments; we only reviewed the correctness via structural tests and by spot-checking parts of the code. Third, many ideas for the proofs of the lower bounds in \Cref{sec:euclidean} were suggested by AI. Fourth, many relevant references were located by AI. Finally, AI generated first suggestions for the proofs of some auxiliary lemmas. Aside from the aforementioned suggestions, the text in this paper was human-written.

\bibliographystyle{siamplain}
\bibliography{reference}

\end{document}